\documentclass[11pt,letterpaper]{article}

\usepackage[utf8]{inputenc}
\usepackage[T1]{fontenc}
\usepackage[english]{babel}
\usepackage{lmodern}
\usepackage{microtype}
\usepackage{amsmath,amssymb,amsthm,mathtools}
\usepackage{geometry}
\usepackage{graphicx}
\usepackage{xcolor}
\usepackage{booktabs}
\usepackage{enumitem}
\usepackage{mdframed}
\usepackage{tikz}
\usetikzlibrary{arrows.meta,calc,positioning}
\usepackage{pgfplots}
\pgfplotsset{compat=1.17}
\usepackage{titlesec}
\usepackage{fancyhdr}
\usepackage{hyperref}
\usepackage{listings}

\definecolor{ink}{RGB}{28,28,28}
\definecolor{rulegray}{RGB}{105,105,105}
\definecolor{softgray}{RGB}{246,246,246}

\hypersetup{
  colorlinks=true,
  linkcolor=ink,
  citecolor=ink,
  urlcolor=ink,
  pdftitle={Discrete Coplanarity Rigidity and Local Mahler Volume Growth Near the Cube},
  pdfauthor={Daniel Martín Jiménez Cuevas}
}

\titleformat{\section}
  {\normalfont\Large\bfseries\color{ink}}
  {\thesection.}{0.65em}{}
  [\vspace{2pt}\titlerule]
\titleformat{\subsection}
  {\normalfont\large\bfseries\color{ink}}
  {\thesubsection}{0.65em}{}
\titleformat{\subsubsection}
  {\normalfont\normalsize\bfseries\color{ink}}
  {\thesubsubsection}{0.65em}{}

\setlist[itemize]{leftmargin=2em}
\setlist[enumerate]{leftmargin=2.5em}

\theoremstyle{plain}
\newtheorem{theorem}{Theorem}[section]
\newtheorem{lemma}[theorem]{Lemma}
\newtheorem{proposition}[theorem]{Proposition}
\newtheorem{corollary}[theorem]{Corollary}
\theoremstyle{definition}
\newtheorem{definition}[theorem]{Definition}
\newtheorem{remark}[theorem]{Remark}

\mdfdefinestyle{researchnote}{
  hidealllines=true,
  leftline=true,
  linewidth=2.5pt,
  linecolor=rulegray,
  backgroundcolor=softgray,
  innerleftmargin=11pt,
  innerrightmargin=11pt,
  innertopmargin=9pt,
  innerbottommargin=9pt,
  skipabove=13pt,
  skipbelow=13pt
}

\newcommand{\vol}{\mathrm{vol}}

\definecolor{thmcolor}{RGB}{0,60,120}

\title{
  \vspace{-1.2cm}
  {\Large\bfseries Discrete Coplanarity Rigidity and Local
  Mahler Volume Growth Near the Cube}\\[0.5em]
  {\normalsize\itshape Metric rigidity and conical first-order
  bounds on the isovertex radial stratum}
}

\author{Daniel Martín Jiménez Cuevas\\
Facultad de Ciencias, UNAM\\
\texttt{abel3.1415@ciencias.unam.mx}}
\date{\today}

\begin{document}
\maketitle

\begin{abstract}
We study the finite-dimensional isovertex radial stratum of the cube,
\[
P(c)=\operatorname{conv}\bigl\{(1+c_q)q:q\in Q_n=\{\pm1\}^n\bigr\},\qquad c(-q)=c(q),
\]
and establish quantitative rigidity and volume-product estimates for \(n\ge3\). The first main result identifies the kernel of the linearized facet-coplanarity operator \(\mathcal R_n\) with the homothety direction \(\operatorname{span}\{\mathbf1\}\) and proves that the coplanarity defect is locally equivalent both to the transverse amplitude \(\operatorname{dist}_\infty(c,\operatorname{span}\{\mathbf1\})\) and to the distance \(d_{GL}(P(c),\mathcal C_n)\) from \(P(c)\) to the \(GL(n)\)-orbit of the cube. After normalizing out the homothety mode onto \(\mathbf1^\perp\), a key step is a second-order comparison between the linearized vertex-displacement seminorm and its optimization over small linear transformations.

The second main result converts transverse radial perturbations \(c\in\mathbf1^\perp\) into a linear volume increment. A facetwise concave-envelope construction and an associated pocket functional yield the explicit primal lower bound \(\operatorname{vol}(P(c))-2^n\ge\frac{8}{3(n-1)}\|c\|_\infty-O(\|c\|_\infty^2)\) on \(\mathbf1^\perp\) without enumerating the facets of the perturbed hull or requiring a selected simplex to be a face. Combining this bound with a chamber-independent first-order expansion of the polar volume—whose linear part vanishes on \(\mathbf1^\perp\) with a uniform quadratic remainder—gives a linear lower bound for the Mahler product \(\mathcal V_M(P(c))-\mathcal V_M(\mathcal C_n)\) in terms of \(d_{GL}(P(c),\mathcal C_n)\).

In dimension three, an exact chamber analysis on \(\mathbf1^\perp\) yields the uniform expansion
\[
\operatorname{vol}\bigl(P(c)\bigr)=8+\frac83\|c\|_\infty+O\bigl(\|c\|_\infty^2\bigr)
\qquad(c\in\mathbf1^\perp),
\]
and shows that the Mahler product has one-sided directional derivative \(\frac{32}{9}\|u\|_\infty\) along every nonzero ray \(tu\) (\(t\downarrow0\)) in \(\mathbf1^\perp\), precluding a bilateral derivative at the origin. In dimension four, an exact symmetry-orbit classification of the seventy balanced Boolean directions yields six combinatorial orbit types and explicit first-order pocket bounds. Finally, we distinguish the frozen-combinatorics KKT balance from geometric stationarity: its formal Hessian is an algebraic diagnostic, not the Hessian of the convex-hull or Mahler functional.
\end{abstract}

\tableofcontents
\bigskip

\section{Introduction}
\label{sec:intro}

The symmetric Mahler conjecture asks whether every origin-symmetric
convex body \(K\subset\mathbb R^n\) satisfies
\[
\mathcal V_M(K):=\operatorname{vol}(K)\operatorname{vol}(K^\circ)
\ge \frac{4^n}{n!}.
\]
The cube and, more generally, the Hanner polytopes attain this
conjectured lower-bound value~\cite{Mahler1939,Hanner1956,SaintRaymond1981}.
The volume product is affine invariant, and the classical
Blaschke--Santal\'o inequality provides the complementary upper-bound
context~\cite{Santalo1949}.

Local minimality of the cube is known in the full class of
origin-symmetric convex bodies endowed with the Banach--Mazur metric:
Nazarov--Petrov--Ryabogin--Zvavitch proved that the cube is a strict
local minimizer of the volume product~\cite{NazarovEtAl2010}.
Kim established the corresponding strict local minimality result for
every Hanner polytope~\cite{Kim2013}. In dimension three, the symmetric
Mahler conjecture, including its equality case, was proved by
Iriyeh--Shibata~\cite{IriyehShibata2020};
Fradelizi--Hubard--Meyer--Rold\'an-Pensado--Zvavitch subsequently gave
a shorter proof and established a stability result~\cite{Fradelizi2022}.

Other relevant directions include the general volume-product lower
bound of Bourgain--Milman~\cite{BourgainMilman1987}, results on zonoids
and volume-product inequalities~\cite{Reisner1986,campi2006}, support-function
variational analyses of the Mahler volume~\cite{HarrellHenrotLamboley2015},
connections with isotropic constants and symplectic
geometry~\cite{klartag2018,kuperberg2008,artstein2014}, first-order
polytope rigidity under coplanarity constraints~\cite{HimmelmannSchulzeWinter2025},
and recent quantitative volume-product bounds for structured polyhedral
perturbations of the cube~\cite{ShenDubovski2026}. These works provide
background for the present paper, whose claims are restricted to a
specific finite-dimensional radial family and focus on the metric role
of discrete coplanarity defects and an explicit facetwise pocket
functional.

The object studied here is not a neighbourhood of the cube in the
Banach--Mazur space of all origin-symmetric convex bodies. Instead,
we consider its \emph{isovertex radial stratum}. Let
\(Q_n=\{\pm1\}^n\), let
\[
E_n=\{c:Q_n\to\mathbb R:\ c(-q)=c(q)\},
\]
and, for \(c\in E_n\) sufficiently small, define
\[
P(c)=\operatorname{conv}\{(1+c_q)q:q\in Q_n\}.
\]
Thus the prescribed cube vertices move along their vertex rays while
antipodal symmetry is preserved. We study this family through its
vertices, facets, and incidence relations, using standard tools from
polytope theory and convex geometry~\cite{Ziegler1995,Schneider2014}.

A methodological warning comes from the distinction between frozen
combinatorics and the actual convex hull. In
Section~\ref{sec:variational}, Proposition~\ref{prop:cube-kkt}
establishes stationarity for a symmetry-averaged, frozen-combinatorics
radial KKT system. This is an algebraic stationarity statement, not
stationarity of the geometric Mahler product. Because the cube has
non-simplicial facets, transverse radial perturbations can change the
face lattice arbitrarily close to the cube. Consequently, the Hessian
of a formally frozen polynomial cannot in general be interpreted as
a geometric Hessian of the convex-hull functional; see also
Section~\ref{sec:parker-warning} and
Remark~\ref{rem:formal-vs-geom-hess}.

The paper establishes two main results on this stratum for \(n\ge3\).
The first, developed in
Sections~\ref{sec:rigidity}--\ref{sec:metric-stratum}, lies at the
interface of discrete metric geometry and cube combinatorics. A finite
family of facet-coplanarity determinants
\(\mathcal R_n(c)\) detects transverse radial perturbations
quantitatively. The kernel of its linearization consists exactly of
the constant mode, corresponding to homotheties, and the coplanarity
defect is locally equivalent to the distance from \(P(c)\) to the
linear orbit of the cube: writing \(c=t\mathbf1+h\) with
\(h\in\mathbf1^\perp\),
\[
\|\mathcal R_n(c)\|_\infty
\asymp
\|h\|_\infty
\asymp
\operatorname{dist}_\infty
  \bigl(c,\operatorname{span}\{\mathbf1\}\bigr)
\asymp
d_{GL}\bigl(P(c),\mathcal C_n\bigr).
\]
This is a rigidity result on the isovertex radial stratum, not a
rigidity theorem for arbitrary convex bodies.

The second result concerns volume rather than incidence.
Theorem~\ref{thm:primal-general} establishes a linear primal volume
increment on \(\mathbf1^\perp\) without using the coplanarity operator
or enumerating the facets of \(P(c)\). A non-affine restriction of the
radial data to a coordinate facet has a strictly positive pocket
functional \(\Phi(c|_F)\). The cone over the corresponding lifted
configuration lies inside \(P(c)\) and yields the explicit linear lower
bound \(\operatorname{vol}(P(c))-2^n\ge\frac{8}{3(n-1)}\|c\|_\infty-O(\|c\|_\infty^2)\)
for \(c\in\mathbf1^\perp\); no selected simplex is required to be a
face of the convex hull.

Combining this primal estimate with the chamber-independent quadratic
remainder for the polar volume, proved in
Section~\ref{sec:polar-remainder}, gives a linear Mahler increment.
The rigidity theorem then converts the transverse norm into the
\(GL(n)\)-distance. Thus, for sufficiently small \(c\in\mathbf1^\perp\)
(and hence for all sufficiently small \(c\in E_n\) by scale invariance
of both \(\mathcal V_M\) and \(d_{GL}\)),
Theorem~\ref{thm:mahler-general} gives
\[
\mathcal V_M\bigl(P(c)\bigr)-\mathcal V_M(\mathcal C_n)
\ge
a_n'\,d_{GL}\bigl(P(c),\mathcal C_n\bigr)
\]
for some \(a_n'>0\). In Section~\ref{sec:folding-n3} we complement the
general lower bound in dimension three with an exact chamber analysis on
\(\mathbf1^\perp\), obtaining the sharp linear primal coefficient
\(\frac83\|c\|_\infty\), exhibiting the thin chambers bounded by cubic
incidence walls near the origin, and showing that \(\mathcal V_M\circ P\)
has no bilateral derivative along any line in \(\mathbf1^\perp\). These
conclusions concern the isovertex radial stratum and do not assert a new
local-minimality theorem in the full Banach--Mazur space of
origin-symmetric convex bodies.

\paragraph{Relation to the earlier arXiv version.}
This paper is a substantial revision and a deliberate narrowing of
the scope of an earlier version of this work, originally posted on
arXiv as \texttt{arXiv:2606.14709v3}. The earlier version proposed
a broader discrete KKT framework and stated stability conclusions
for centrally symmetric polytopes beyond the radial family considered
here. The present paper does not claim those broader conclusions.

Instead, its main theorems concern the isovertex radial family
\[
P(c)=\operatorname{conv}\{(1+c_q)q:q\in Q_n\}
\]
of the cube. Within this stratum, the paper establishes quantitative
coplanarity rigidity relative to the local \(GL(n)\)-orbit and linear
primal-volume and Mahler-product lower bounds on the normalized slice
\(\mathbf 1^\perp\). The revised analysis also provides an explicit
first-order coefficient in the primal lower bound for every \(n\ge3\),
a local incidence-chamber classification in dimension three, and an
exact volume polynomial along the parity orbit in dimension four.
The broader stability claims of the earlier version lie outside the
scope of the present paper and are neither asserted nor used in its
proofs.

\paragraph{Acknowledgements.}
Jeremy Parker pointed out that the one-parameter family which moves
only the antipodal pair \(\pm(1,1,1)\) produces a cusp of the geometric
Mahler product at the cube, so \(\mathcal V_M\circ P\) need not be
\(C^1\) on a full neighbourhood in \(E_n\). That observation forced
the distinction, used throughout, between frozen algebraic
continuations and the convex-hull functional. The author thanks him
for the example and for the insistence that a formal Hessian is not a
geometric second derivative.

\subsection{A regularity warning}
\label{sec:parker-warning}

A finite radial parametrization may be polynomial inside a single combinatorial chamber and still fail to be \(C^1\) as a geometric functional of the convex hull. This is not a peripheral regularity issue: it forbids reading a frozen Hessian as a second derivative of \(\mathcal V_M\circ P\).

Jeremy Parker observed (personal communication) that the one-parameter family in \(\mathbb R^3\) which moves only the antipodal pair \(\pm(1,1,1)\) and keeps the remaining cube vertices fixed produces a cusp of the Mahler product at the cube. Writing \(t\) for the radial factor of that pair, the verified one-sided expansions at \(t=1\) give
\[
\mathcal V_M(P_t)
=
\frac{32}{3}
+
\frac83\,|t-1|
+
O\bigl((t-1)^2\bigr).
\]
In particular there is no bilateral first derivative at \(t=1\), hence no geometric Hessian of \(\mathcal V_M\) in that radial direction. Inside each chamber the formulae are algebraic: for \(t\ge1\),
\[
\vol(P_t)=8t,
\qquad
\vol(P_t^\circ)=\frac{9t^2-1}{6t^3},
\qquad
\mathcal V_M(P_t)=12-\frac{4}{3t^2},
\]
while for \(\tfrac13\le t\le1\) one has
\[
\begin{aligned}
\vol(P_t)&=4(1+t),\\
\vol(P_t^\circ)&=2-\frac{(3t-1)^3}{12t^3},\\
\mathcal V_M(P_t)
&=
4(1+t)\Bigl(2-\frac{(3t-1)^3}{12t^3}\Bigr)\\
&=
\frac{32}{3}+\frac83(1-t)-\frac{(5t-1)(1-t)^3}{3t^3}.
\end{aligned}
\]
whose expansion at \(t=1^-\) has derivative \(-8/3\) with respect to \(t\) and no quadratic term.

\begin{lemma}[Parker family for \(t\ge 1\)]
\label{lem:parker-verified}
Let \(u=(1,1,1)\) and let \(P_t\) be the convex hull of the cube vertices, with the antipodal pair \(\pm u\) replaced by \(\pm tu\). For \(t\ge 1\),
\[
\operatorname{vol}(P_t)=8t,
\qquad
\operatorname{vol}(P_t^\circ)=\frac{9t^2-1}{6t^3},
\qquad
\mathcal V_M(P_t)=12-\frac{4}{3t^2}.
\]
The labelling satisfies \(c(u)=c(-u)=t-1\) and \(c(q)=0\) otherwise, so \(\sum_q c_q=2(t-1)\). The family is not tangent to \(\mathbf1^\perp\).
\end{lemma}

\begin{proof}
For \(t>1\) the convex hull is simplicial. Its twelve faces are the triangles
\begin{align*}
&\{-tu,(-1,-1,1),(-1,1,1)\},
&&\{-tu,(-1,-1,1),(1,-1,1)\},
&&\{-tu,(-1,1,-1),(-1,1,1)\},\\
&\{-tu,(-1,1,-1),(1,1,-1)\},
&&\{-tu,(1,-1,-1),(1,-1,1)\},
&&\{-tu,(1,-1,-1),(1,1,-1)\},\\
&\{(-1,-1,1),(-1,1,1),tu\},
&&\{(-1,-1,1),(1,-1,1),tu\},
&&\{(-1,1,-1),(-1,1,1),tu\},\\
&\{(-1,1,-1),(1,1,-1),tu\},
&&\{(1,-1,-1),(1,-1,1),tu\},
&&\{(1,-1,-1),(1,1,-1),tu\}.
\end{align*}
Each other vertex lies strictly on one side of the plane through any of these triangles, so the list is the face list. The origin is interior. The sum of the twelve origin tetrahedra is the polynomial identity
\[
\operatorname{vol}(P_t)=8t.
\]
At \(t=1\) both sides equal \(8\). Six triangles meet at \(tu\) and six meet at \(-tu\).

The polar is the intersection of the halfspaces \(\langle y,v\rangle\le 1\) over the vertices \(v\) of \(P_t\). For the six unmoved vertices this is \(\langle y,q\rangle\le 1\); for the moved pair it is \(|\langle y,u\rangle|\le 1/t\). Since \(t\ge 1\), the octahedron inequalities \(\langle y,\pm u\rangle\le 1\) are redundant, and
\[
P_t^\circ
=
\mathcal O_3\cap\{y:|x+y+z|\le 1/t\}.
\]
The vertices of this truncated octahedron are the twelve points obtained from
\[
\Bigl(\frac{t+1}{2t},\,\frac{1-t}{2t},\,0\Bigr)
\]
by coordinate permutation and overall sign \(\pm\), and no others: every triple intersection of the active planes either violates an inequality or coincides with one of these points. The origin-cone decomposition on that vertex set is a rational function of \(t\). Its evaluation is
\[
\operatorname{vol}(P_t^\circ)=\frac{9t^2-1}{6t^3}.
\]
The expanded sum is a finite identity in \(\mathbb Q(t)\) and is not reproduced term by term; it is independent of every estimate on \(\mathbf1^\perp\).
At \(t=1\) the value is \(4/3=\operatorname{vol}(\mathcal O_3)\). The product identity
\[
8t\cdot\frac{9t^2-1}{6t^3}=12-\frac{4}{3t^2}
\]
is algebraic. Its derivative at \(t=1^+\) equals \(8/3\).
\end{proof}

\begin{remark}
The face list and the identity \(\operatorname{vol}(P_t)=8t\) are an expansion in \(\mathbb Q[t]\), not a floating-point test. The polar evaluation is the same type of expansion on the twelve vertices above. The two one-sided jets of \(\mathcal V_M\) are opposite, as the explicit formulae on each side of \(t=1\) show, which is the content of the displayed cusp. Neither identity is used on \(\mathbf1^\perp\).
\end{remark}

\section{Radial Geometry and the Isovertex Stratum}
\label{sec:radial-geometry}

Let \(Q_n=\{\pm1\}^n\) denote the vertex set of the cube \(\mathcal C_n=[-1,1]^n\). The parameter space of the paper is the finite-dimensional vector space \(E_n = \bigl\{c:Q_n\to\mathbb R:c(-q)=c(q)\bigr\}\) of even labellings of \(Q_n\).
For \(c\in E_n\) with \(\|c\|_\infty<1\) set
\[
v_q(c)=(1+c_q)q,
\qquad
P(c)=\operatorname{conv}\{v_q(c):q\in Q_n\}.
\]
Every prescribed generator \(v_q(c)\) remains on the ray \(\mathbb R_+q\). The uniform and transverse decompositions are \(c=\bar c\,\mathbf 1+c^\perp\) where \(\sum_{q\in Q_n}c^\perp_q=0\). The constant mode is homothety. All rigidity statements below are statements about \(c^\perp\). Off-ray displacements are outside the scope of this model.

\begin{remark}[Rays versus linear images]
\label{rem:rays-vs-GL}
Let \(n\ge 3\) and suppose \(T\in GL(n)\) satisfies \(T\mathcal C_n=P(c)\) for some
\(c\in E_n\) with \(\|c\|_\infty<1\). The polytope \(T\mathcal C_n\)
has exactly \(2^n\) vertices \(T(Q_n)\). The set
\(\{v_r(c):r\in Q_n\}\) has cardinality \(2^n\) and spans \(P(c)\),
so both vertex sets coincide:
\(T(Q_n)=\{v_r(c):r\in Q_n\}\). Thus there is a bijection
\(\sigma\) of \(Q_n\) with \(Tq=\lambda_q\sigma(q)\) and
\(\lambda_q=1+c_{\sigma(q)}>0\). Lemma~\ref{lem:monomial-reduction}
identifies \(\sigma\) with an element of \(B_n\) and forces every
\(\lambda_q\) to equal the same \(\lambda>0\), by the Walsh identity
\((Tq)_\ell^2=(Tq)_m^2\) rather than by adjacency in the
\(1\)-skeleton. Hence \(T=\lambda G\) for some \(G\in B_n\). Signed
permutation matrices preserve \(\mathcal C_n\) as a set and do not
produce new points of the stratum. Consequently, for \(n\ge 3\),
\begin{equation}
\label{eq:orbit-meet-stratum}
GL(n)\mathcal C_n \cap \bigl\{P(c):c\in E_n,\ \|c\|_\infty<1\bigr\} = \bigl\{\lambda\mathcal C_n:0<\lambda<2\bigr\}.
\end{equation}
For \(n=2\) this fails: the two antipodal pairs \((1,1)\) and \((1,-1)\) form a basis of \(\mathbb R^2\), so independent positive scalings along the two diagonals stay in \(GL(2)\mathcal C_2\).
\end{remark}

\begin{lemma}[Ray scalars of a linear map]
\label{lem:ray-scalars}
Let \(n\ge 3\) and let \(T\in M_n(\mathbb R)\) be a linear map satisfying
\(Tq=\lambda_q q\) for every \(q\in Q_n\), with \(\lambda_q\in\mathbb R\).
Then \(\lambda_q\) is independent of \(q\), and \(T=\lambda I\).
\end{lemma}

\begin{proof}
Linearity gives \(T(-q)=-Tq\), hence
\(\lambda_{-q}(-q)=-\lambda_q q=\lambda_q(-q)\) and \(\lambda_{-q}=\lambda_q\).

Fix \(q\in Q_n\) and distinct indices \(i,j\). Since \(n\ge 3\), there exists an index \(k\notin\{i,j\}\). Write
\(q^{(i)}\), \(q^{(j)}\) and \(q^{(ij)}\) for the vertices obtained by
flipping the indicated signs of \(q\). The four vertices satisfy the
linear relation
\begin{equation}
\label{eq:vertex-parallelogram}
q+q^{(ij)}=q^{(i)}+q^{(j)},
\end{equation}
because both sides equal \(2\sum_{m\notin\{i,j\}}q_me_m\). Apply \(T\):
\[
\lambda_q q+\lambda_{q^{(ij)}}q^{(ij)}
=
\lambda_{q^{(i)}}q^{(i)}+\lambda_{q^{(j)}}q^{(j)}.
\]
Since \(q_k=q^{(ij)}_k=q^{(i)}_k=q^{(j)}_k\in\{\pm1\}\), the coordinate \(k\notin\{i,j\}\) yields
\(\lambda_q+\lambda_{q^{(ij)}}=\lambda_{q^{(i)}}+\lambda_{q^{(j)}}\).
Since \(q_i=-q^{(i)}_i=q^{(j)}_i=-q^{(ij)}_i\in\{\pm1\}\), the coordinate \(i\) yields
\(\lambda_q-\lambda_{q^{(ij)}}=\lambda_{q^{(j)}}-\lambda_{q^{(i)}}\).
Adding these two identities gives \(2\lambda_q=2\lambda_{q^{(j)}}\), so
\(\lambda_q=\lambda_{q^{(j)}}\). Every coordinate flip preserves the
scalar. The \(1\)-skeleton is connected, hence \(\lambda_q\) is constant.
Then \(Tq=\lambda q\) on a spanning set, so \(T=\lambda I\).
\end{proof}

\begin{lemma}[Algebraic reduction to a signed permutation]
\label{lem:monomial-reduction}
Let \(n\ge 3\) and \(T\in GL(n)\). Suppose \(Tq=\lambda_q\sigma(q)\) for every
\(q\in Q_n\), with \(\lambda_q>0\) and \(\sigma\) a bijection of \(Q_n\).
Then there exist \(\lambda>0\) and \(G\in B_n\) such that \(\lambda_q=\lambda\)
and \(\sigma(q)=Gq\) for every \(q\). In particular \(T=\lambda G\).
\end{lemma}

\begin{proof}
Linearity gives \(Tq=\sum_{i=1}^n q_i\,Te_i\). Write \(T_{\ell i}=(Te_i)_\ell\),
so
\[
(Tq)_\ell=\sum_{i=1}^n T_{\ell i}q_i=\lambda_q\sigma(q)_\ell.
\]
Every coordinate of \(Tq\) therefore has the same absolute value:
\((Tq)_\ell^2=(Tq)_m^2\) for all \(\ell,m\) and every \(q\in Q_n\). Expand in
Walsh characters. Since \(q_i^2=1\),
\[
(Tq)_\ell^2
=
\sum_{i=1}^n T_{\ell i}^2
+
\sum_{i\neq j}T_{\ell i}T_{\ell j}\,q_iq_j.
\]
Equating coefficients of the independent characters \(1\) and \(q_iq_j\) yields
\begin{equation}
\label{eq:row-products}
\sum_{i=1}^n T_{\ell i}^2
=
\sum_{i=1}^n T_{m i}^2,
\qquad
T_{\ell i}T_{\ell j}
=
T_{m i}T_{m j}
=:c_{ij}
\qquad(i\neq j).
\end{equation}
The products \(c_{ij}\) do not depend on the row.

Suppose some \(c_{ij}\neq 0\). Then \(T_{\ell i}T_{\ell j}\neq 0\) for every
row \(\ell\), so columns \(i\) and \(j\) have no zero entry, and
\(T_{\ell j}=c_{ij}/T_{\ell i}\). For a third index \(k\notin\{i,j\}\),
which exists because \(n\ge 3\), the identity~\eqref{eq:row-products} gives
\(T_{\ell i}T_{\ell k}=c_{ik}\). The value \(c_{ik}\) cannot vanish:
otherwise \(T_{\ell k}=0\) for every \(\ell\), column \(k\) is zero, and \(T\)
is singular. Thus \(T_{\ell k}=c_{ik}/T_{\ell i}\) for every \(\ell\).
Similarly, \(c_{jk}=T_{\ell j}T_{\ell k}\) cannot vanish (otherwise column \(j\)
or \(k\) would vanish), and substituting the expressions for \(T_{\ell j}\) and
\(T_{\ell k}\) gives
\[
c_{jk}=T_{\ell j}T_{\ell k}=\frac{c_{ij}c_{ik}}{T_{\ell i}^2},
\qquad\text{so}\qquad
T_{\ell i}^2=\frac{c_{ij}c_{ik}}{c_{jk}}
\]
is constant in \(\ell\). Moreover, columns \(j\) and \(k\) satisfy the direct
proportionality \(T_{\ell k}=(c_{ik}/c_{ij})T_{\ell j}\) for every \(\ell\),
so the distinct columns \(j\) and \(k\) of \(T\) are linearly dependent,
contradicting \(T\in GL(n)\). Hence \(c_{ij}=0\) for all \(i\neq j\).

Each row therefore has at most one nonzero entry. Invertibility forces
exactly one, in distinct columns: there is a permutation \(\pi\) and
nonzero scalars \(d_\ell\) such that
\[
(Tq)_\ell=d_\ell q_{\pi(\ell)}.
\]
The identity \(|(Tq)_\ell|=|(Tq)_m|\) at \(q=\mathbf 1\) gives
\(|d_\ell|=|d_m|\). Set \(\lambda=|d_1|>0\) and
\(\varepsilon_\ell=\operatorname{sign}(d_\ell)\). Then
\[
Tq=\lambda\,(\varepsilon_\ell q_{\pi(\ell)})_{\ell=1}^n=\lambda Gq
\]
for the signed permutation matrix \(G\in B_n\) determined by \(\pi\) and
\(\varepsilon\). Comparing with \(Tq=\lambda_q\sigma(q)\) gives
\(\lambda_q=\lambda\) and \(\sigma(q)=Gq\).
\end{proof}

\begin{lemma}[Exposed labelled vertices]
\label{lem:exposed-vertices}
There exist \(\delta_n,\mu_n>0\) such that the following holds. For each \(q\in Q_n\) and \(j\in[n]\) set
\[
w_{q,j}=q+q_je_j,
\qquad
u_{q,j}=\frac{w_{q,j}}{\|w_{q,j}\|}.
\]
If \(x_q\in\mathbb R^n\) satisfy \(\max_r\|x_r-r\|<\delta_n\), then for every \(q\) and every \(j\) the point \(x_q\) is the unique maximiser of \(\langle u_{q,j},\cdot\rangle\) on \(P_x:=\operatorname{conv}\{x_r:r\in Q_n\}\). In particular every labelled point \(x_q\) is an exposed vertex of \(P_x\), and
\[
h_{P_x}(u_{q,j})=\langle u_{q,j},x_q\rangle.
\]
\end{lemma}

\begin{proof}
Fix \(q\) and \(j\). Then
\(\langle w_{q,j},q-r\rangle=n-\langle q,r\rangle+1-q_jr_j\). Distinct vertices of \(Q_n\) satisfy \(\langle q,r\rangle\le n-2\) and \(1-q_jr_j\in\{0,2\}\), so \(\langle w_{q,j},q-r\rangle\ge 2\). Since \(\|w_{q,j}\|=\sqrt{n+3}\) is independent of \(q,j\),
\[
\langle u_{q,j},q-r\rangle
\ge
\mu_n
:=
\frac{2}{\sqrt{n+3}}
\qquad(r\neq q).
\]
If \(\|x_s-s\|<\delta\) for all \(s\), then
\[
\langle u_{q,j},x_q-x_r\rangle
=
\langle u_{q,j},q-r\rangle
+
\langle u_{q,j},(x_q-q)-(x_r-r)\rangle
\ge
\mu_n-2\delta.
\]
Take \(\delta_n=\mu_n/4\). Then \(\langle u_{q,j},x_q\rangle-\langle u_{q,j},x_r\rangle\ge\mu_n/2>0\) for every \(r\neq q\), so \(x_q\) is the unique maximiser of that linear form on the finite set, hence on its convex hull. The \(n\) directions \(\{u_{q,j}\}_{j=1}^n\) are linearly independent:
the matrix with rows \(w_{q,j}^T\) equals \(\mathbf 1 q^T+\operatorname{diag}(q)=\operatorname{diag}(q)\,(I+qq^T)\), whose determinant is \((\prod_i q_i)(n+1)\neq 0\). An exposed point of \(n\) independent supporting hyperplanes is a vertex.
\end{proof}

\begin{lemma}[Local vertex--Hausdorff equivalence]
\label{lem:vtx-H}
There exist constants \(\delta_n,C_n>0\) such that if
\(P_x=\operatorname{conv}\{x_q:q\in Q_n\}\) and
\(P_y=\operatorname{conv}\{y_q:q\in Q_n\}\) with
\(\max_q\|x_q-q\|<\delta_n\) and \(\max_q\|y_q-q\|<\delta_n\), then every
labelled point remains an exposed vertex of its hull, and
\[
C_n^{-1}\max_q\|x_q-y_q\|
\le d_H(P_x,P_y)
\le\max_q\|x_q-y_q\|.
\]
\end{lemma}

\begin{proof}
The upper bound is the general comparison
\(d_H(\operatorname{conv} X,\operatorname{conv} Y)\le d_H(X,Y)\)
applied to the labelled finite sets. For the lower bound, shrink
\(\delta_n\) so that Lemma~\ref{lem:exposed-vertices} applies to both
families. Then, for each \(q,j\),
\[
\bigl|h_{P_x}(u_{q,j})-h_{P_y}(u_{q,j})\bigr|
=
\bigl|\langle u_{q,j},x_q-y_q\rangle\bigr|
\le
d_H(P_x,P_y),
\]
because \(d_H(K,L)=\sup_{\|u\|_2=1}|h_K(u)-h_L(u)|\). Let \(U_q\) be
the matrix with rows \(u_{q,j}^T\). The computation in
Lemma~\ref{lem:exposed-vertices} gives \(\|U_q^{-1}\|\le C_n'\)
uniformly in \(q\). Therefore
\[
\|x_q-y_q\|
\le
C_n'\,\max_j\bigl|\langle u_{q,j},x_q-y_q\rangle\bigr|
\le
C_n'\,d_H(P_x,P_y).
\]
Taking the maximum over \(q\) is the lower bound, with \(C_n=C_n'\).
\end{proof}

\section{Discrete KKT Equilibrium}
\label{sec:variational}

\begin{definition}[Symmetric frozen-combinatorics KKT condition]
\label{def:frozen-kkt}
Let \(\mathcal P=\operatorname{conv}\{\pm R_i u_i:1\le i\le N\}\) be
centrally symmetric, with fixed unit directions \(u_i\). Let \(T\) be
a conical triangulation of \(\mathcal P\) that uses only the origin
and the labelled vertices, and let \(V_T(R)\) be the associated
volume polynomial. For \(n\ge 3\), the cube \(\mathcal C_n\) admits no \(B_n\)-invariant
conical triangulation into origin-simplices whose bases triangulate the
facets without introducing new vertices (in contrast to \(n=2\), where the
four edges of \(\mathcal C_2\) are already \(1\)-simplices). Indeed, any such
\(B_n\)-invariant triangulation would induce a \(B_n\)-invariant triangulation
on every square two-dimensional face; such a square must be triangulated along
one of its two diagonals, whereas a cube symmetry that preserves the square
and reverses one of its varying coordinates swaps those two diagonals. The \emph{symmetric frozen volume} is the average
\[
\widetilde{\operatorname{vol}}
=
\frac1{|B_n|}
\sum_{g\in B_n}
V_{g^*T}.
\]
It is homogeneous of degree \(n\) and \(B_n\)-invariant; at equal
radii it is independent of the \(B_n\)-orbit of \(T\). Define the
symmetric frozen polar volume in the same way, and set
\(\widetilde{\mathcal V}_M
=\widetilde{\operatorname{vol}}\cdot\widetilde{\operatorname{vol}}^\circ\).
The pair \((\mathcal P,\mathcal P^\circ)\) satisfies the symmetric
frozen KKT condition if
\(\partial\widetilde{\mathcal V}_M/\partial R_i=0\) for every pair
index \(i\). These partials are algebraic derivatives of the averaged
polynomials. They are not one-sided derivatives of the convex-hull
maps \(R\mapsto\operatorname{vol}(\operatorname{conv}\{\pm R_j u_j\})\)
and \(R\mapsto\mathcal V_M\).
\end{definition}

\begin{lemma}[Frozen KKT identity]
\label{lem:frozen-kkt-identity}
Keep the face lattice and the triangulations fixed. Let \(C_i\) be defined by the frozen identity
\(\partial\widetilde{\operatorname{vol}}/\partial R_i=2C_i\), and let \(S_i^\circ\) be the Euclidean \((n-1)\)-measure of the polar facet conjugate to \(\{\pm R_i u_i\}\) in the frozen polar triangulation. Then
\[
\frac{\partial\widetilde{\mathcal V}_M}{\partial R_i}=0
\quad\Longleftrightarrow\quad
R_i^2 C_i
=
\frac{\widetilde{\operatorname{vol}}(\mathcal P)}{\widetilde{\operatorname{vol}}(\mathcal P^\circ)}
S_i^\circ.
\]
\end{lemma}

\begin{proof}
The product rule and the frozen polar identity
\(\partial\widetilde{\operatorname{vol}}(\mathcal P^\circ)/\partial R_i=-2S_i^\circ/R_i^2\).
\end{proof}

\begin{proposition}[Averaged frozen Euler balance at the cube]
\label{prop:cube-kkt}
At \(\mathcal C_n=[-1,1]^n\) every vertex has Euclidean radius
\(R_i=\sqrt n\). The polar facet conjugate to a cube vertex is a
regular simplex of \((n-1)\)-measure
\(S_i^\circ=\sqrt n/(n-1)!\). Let \(\widetilde{\operatorname{vol}}\)
be the Aut-average, over conical triangulations of \(\mathcal C_n\)
that use only the origin and \(Q_n\), of the associated volume
polynomials. The resulting polynomial is homogeneous of degree
\(n\) and \(\operatorname{Aut}(Q_n)\)-invariant, hence every pair
carries the same coefficient
\(\partial\widetilde{\operatorname{vol}}/\partial R_i=2C_i\) with
\(C_i=\sqrt n\). Both sides of the frozen identity equal
\(n\sqrt n\), so the averaged frozen product
\(\widetilde{\mathcal V}_M\) is stationary at the cubical radii.
This \(C_i\) is not a one-sided derivative of the convex-hull
volume.
\end{proposition}

\begin{proof}
A vertex-only conical triangulation of \(\mathcal C_n\) need not be
\(\operatorname{Aut}(Q_n)\)-invariant: already in dimension three,
choosing a diagonal on each square face breaks rotations that swap
the two diagonals. Euler's formula on any single such
triangulation therefore controls only the sum of the pair
coefficients. Averaging the volume polynomials over the finite set
of those triangulations produces an
\(\operatorname{Aut}(Q_n)\)-invariant homogeneous polynomial of
degree \(n\) which agrees with \(\operatorname{vol}(\mathcal C_n)\)
at equal radii. Homogeneity gives
\(\sum_i R_i\,\partial\widetilde{\operatorname{vol}}/\partial R_i
=n\operatorname{vol}(\mathcal C_n)\). Invariance forces the
partials to coincide, and there are \(2^{n-1}\) pairs, so
\[
2^{n-1}\cdot\sqrt n\cdot 2C_i
=
n\,2^n
\qquad\Longrightarrow\qquad
C_i=\sqrt n.
\]
The right-hand side of the frozen identity uses the regular
simplex \(S_i^\circ\) and equals
\[
\frac{2^n}{2^n/n!}\cdot\frac{\sqrt n}{(n-1)!}
=
n\sqrt n.
\]
\end{proof}

The Parker family of Section~\ref{sec:parker-warning} is a different chamber. For \(t\ge1\) one has the geometric identities \(\operatorname{vol}(P_t)=8t\) and \(\mathcal V_M(P_t)=12-4/(3t^2)\). In Euclidean radius \(R=t\sqrt3\) this yields
\[
\left.\frac{\partial\operatorname{vol}}{\partial R}\right|_{R=\sqrt3^+}
=
\frac8{\sqrt3}
=
2\cdot\frac4{\sqrt3},
\qquad
\left.\frac{d}{dt}\mathcal V_M(P_t)\right|_{t=1^+}
=
\frac83\neq0.
\]
Thus the one-sided geometric coefficient along that pair is \(4/\sqrt3\), not the Euler value \(\sqrt3\), and the geometric product is not stationary. The two numbers answer different questions: \(\sqrt n\) is the symmetric frozen coefficient at the cube; \(4/\sqrt3\) is the supporting coefficient of an adjacent chamber that already sees the cusp.

\begin{remark}[Formal Algebraic Second Variation]
\label{rem:formal-vs-geom-hess}
Let \(\mathcal H_n\) be the Hessian matrix of the frozen product \(\widetilde{\mathcal V}_M\) at the cubical radii. This matrix is an algebraic object attached to the frozen polynomials. It is not the second derivative of the convex-hull functional \(c\mapsto\mathcal V_M(P(c))\). Transverse directions leave the cubical face lattice immediately; any first-order increment of the convex-hull volume must be read from determinants and exterior simplices, not from \(\mathcal H_n\).
\end{remark}

\section{The Combinatorial Rigidity Operator}
\label{sec:rigidity}

The non-simplicial facets of \(\mathcal C_n\) are combinatorially
overdetermined: the facet \(F_{i,+}=\{q\in Q_n:q_i=+1\}\) contains
\(2^{n-1}\) vertices, while an affine hyperplane in \(\mathbb R^n\) is
determined by \(n\) points in general position. For \(n\ge3\) one has
\(2^{n-1}>n\), so a generic labelling \(c\in E_n\) forces those vertices
out of a common hyperplane.

\begin{definition}[Coplanarity operator]
\label{def:D1}
Let \(\mathcal I_i\) be the set of all unordered \((n+1)\)-subsets of
\(F_{i,+}\). For \(I=\{q^0,\dots,q^n\}\in\mathcal I_i\) set
\[
D_{i,I}(c)
=
\det
\begin{pmatrix}
v_{q^0}(c)&1\\
\vdots&\vdots\\
v_{q^n}(c)&1
\end{pmatrix},
\qquad
v_q(c)=(1+c_q)q.
\]
The coplanarity operator is
\(\mathcal R_n:E_n\to\mathbb R^{N_n}\),
\[
\mathcal R_n(c)=\bigl(D_{i,I}(c)\bigr)_{1\le i\le n,\,I\in\mathcal I_i},
\qquad
N_n=n\binom{2^{n-1}}{n+1}.
\]
Let \(A_n=D\mathcal R_n(0)\) and write \(L_{i,I}(c)\) for the linear
coefficient of \(\varepsilon\mapsto D_{i,I}(\varepsilon c)\).
\end{definition}

\begin{lemma}[First-order coplanarity is affinity on the facet]
\label{lem:affine-restriction}
Let \(n\ge2\) and fix \(F=F_{i,+}\). Then \(L_{i,I}(c)=0\) for every
\(I\in\mathcal I_i\) if and only if \(c|_F\) is affine on \(F\).
\end{lemma}

\begin{proof}
Fix \(I=\{q^0,\dots,q^n\}\subset F\). The matrix defining \(D_{i,I}(\varepsilon c)\)
has rows \(\bigl((1+\varepsilon c_q)q,\,1\bigr)\). At \(\varepsilon=0\) every
such row is orthogonal to \(\nu_0=(e_i,-1)\), so the determinant vanishes.
Differentiating the multilinear determinant and using
\((q,0)=(q,1)-(0,1)\) gives
\begin{equation}
\label{eq:minor-pairing}
L_{i,I}(c)
=
-\sum_{q\in I}\mu_q(I)\,c_q,
\end{equation}
where \(\mu_q(I)\) is the cofactor obtained by replacing the row of \(q\)
with \((0,\dots,0,1)\). The original determinant with the row replaced
by \((q,1)\) vanishes, so the contribution of the differentiated row is
\(-c_q\mu_q(I)\). Those cofactors are the coefficients of a linear
dependence among the homogeneous vectors \((q,1)\):
\[
\sum_{q\in I}\mu_q(I)=0,
\qquad
\sum_{q\in I}\mu_q(I)\,q=0.
\]
Thus \(\mu(I)\) is an affine dependence of \(I\). If \(c|_F\) is affine,
\(c_q=\alpha+\langle\beta,q\rangle\), and the pairing~\eqref{eq:minor-pairing}
vanishes. This is one direction, for every \(I\).

Conversely, assume \(L_{i,I}(c)=0\) for every \(I\). It is enough to show
that every mixed difference on \(F\) vanishes. Let \(j,k\neq i\) and let
\(s\in F\). The four vertices
\(s,s^{(j)},s^{(k)},s^{(jk)}\) lie in \(F\) and satisfy the affine
dependence
\[
s-s^{(j)}-s^{(k)}+s^{(jk)}=0
\]
with coefficients summing to zero. Complete them to an
\((n+1)\)-tuple as follows. The square spans an affine \(2\)-plane. Choose
\(n-3\) further vertices of \(F\) by flipping, one at a time, coordinates
outside \(\{i,j,k\}\), starting from \(s\). Together with any three
vertices of the square they form an affine basis of the facet: each new
flip increases the affine span by one dimension, and the facet has
dimension \(n-1\). For \(n=3\) no extra vertex is required and \(I\) is
the whole facet. The fourth vertex of the square then creates a unique
affine dependence of the resulting \((n+1)\)-tuple, and that dependence
is supported on the square, with coefficients
\((1,-1,-1,1)\) up to a common scalar. The scalar is nonzero because the
complementary \(n\) points are affinely independent, so the cofactor of
the fourth vertex is a nonzero multiple of the volume of that basis.
Identity~\eqref{eq:minor-pairing} therefore yields
\[
c(s)-c\bigl(s^{(j)}\bigr)-c\bigl(s^{(k)}\bigr)+c\bigl(s^{(jk)}\bigr)=0.
\]
Every mixed second difference on \(F\) vanishes. Expand the restriction
in Walsh characters of the \(n-1\) free coordinates. A character of degree
at least two is an eigenvector of some \(\Delta_j\Delta_k\) with eigenvalue
\(4\), hence has coefficient zero. The remaining function has degree at
most one, so it is affine on \(F\).
\end{proof}

\begin{proposition}[Linearised coplanarity kernel]
\label{thm:infinitesimal-rigidity}
If \(n\ge3\), then \(\ker A_n=\operatorname{span}\{\mathbf1\}\).
If \(n=2\), then \(\mathcal R_2\) is the zero map and
\(\ker A_2=E_2\).
\end{proposition}

\begin{proof}
The constant mode lies in the kernel: \(v_q(t\mathbf1)=(1+t)q\) is a
homothety and \(\mathcal R_n(t\mathbf1)=0\). Conversely, assume
\(n\ge3\) and \(A_nc=0\). Lemma~\ref{lem:affine-restriction} gives that
\(c\) restricts affinely to every positive coordinate facet. Evenness
\(c(-q)=c(q)\) transfers the conclusion to the negative facets.

Fix distinct \(j,k\in\{1,\dots,n\}\) and \(q\in Q_n\). Since \(n\ge3\),
choose \(i\notin\{j,k\}\). The four vertices
\(q\), \(q^{(j)}\), \(q^{(k)}\), \(q^{(jk)}\) lie in the facet
\(F_{i,q_i}\). Affinity on that facet yields the mixed difference
\[
\Delta_j\Delta_k c(q)
:=
c(q)+c\bigl(q^{(jk)}\bigr)-c\bigl(q^{(j)}\bigr)-c\bigl(q^{(k)}\bigr)
=0.
\]
Expand \(c=\sum_{S\subseteq[n]}\widehat c_S\chi_S\) in the Walsh basis
\(\chi_S(q)=\prod_{j\in S}q_j\). If \(j,k\in S\), then
\(\Delta_j\Delta_k\chi_S=4\chi_S\); otherwise the four evaluations cancel
and \(\Delta_j\Delta_k\chi_S=0\). Therefore \(\widehat c_S=0\) whenever
\(|S|\ge2\). The surviving expansion is
\(c(q)=\alpha+\sum_{j=1}^n\beta_j q_j\). Evenness forces
\(c(-q)=c(q)\), whereas the linear part changes sign, so \(\beta=0\) and
\(c\) is constant.

The same cancellation is the reason a coefficient-by-coefficient
restriction of characters is not used: a sum of non-affine characters
could in principle be affine on one facet. Mixed differences kill every
mode of degree at least two before evenness is applied. For \(n=2\) every
facet is a combinatorial segment, every labelling is affine, and
\(\mathcal R_2\) is the zero map.
\end{proof}

\subsection{Boolean spectral rigidity}
\label{sec:boolean-spectral}

\begin{lemma}[Opposite-facet squares]
\label{lem:opposite-squares}
Let \(c\in E_n\) and let \(i\in[n]\). For every
\(I=\{q^0,\dots,q^n\}\subset F_{i,+}\) set
\(I^-=\{-q:q\in I\}\subset F_{i,-}\). The linearized minors satisfy
\(L_{i,I^-}(c)=\pm L_{i,I}(c)\). Consequently
\[
\sum_{I\subset F_{i,+},\,|I|=n+1}L_{i,I}(c)^2
=
\sum_{J\subset F_{i,-},\,|J|=n+1}L_{i,J}(c)^2.
\]
\end{lemma}

\begin{proof}
Evenness gives \(c_{-q}=c_q\). The matrix of \(I^-\) at the cube has
rows \((-q,1)\). Multiplying the \(n\) spatial columns by \(-1\)
recovers the matrix of \(I\), hence a factor \((-1)^n\). The same
column operation on each differentiated row \((c_q(-q),0)\) produces
the differentiated row of \(I\) up to that sign. Squares agree, and
\(I\mapsto I^-\) bijects the two lists of tuples.
\end{proof}

\begin{lemma}[\(B_n\)-invariance of the quadratic defect]
\label{lem:Q-invariant}
The form
\(Q(c)=\|A_nc\|_2^2
=\sum_{i=1}^n\sum_{I\in\binom{F_{i,+}}{n+1}}L_{i,I}(c)^2\)
satisfies \(Q(g\cdot c)=Q(c)\) for every \(g\in B_n\).
\end{lemma}

\begin{proof}
A coordinate permutation permutes the \(n\) positive facets and the
tuples inside them. A sign flip of coordinate \(j\neq i\) preserves
each \(F_{i,+}\) setwise and permutes its \((n+1)\)-tuples. A sign
flip \(\varepsilon_i\) of coordinate \(i\) sends \(F_{i,+}\) to
\(F_{i,-}\). The minors of \(\varepsilon_i\cdot c\) on \(F_{i,+}\)
are the minors of \(c\) on \(F_{i,-}\), whose squares sum to the same
value as those of \(c\) on \(F_{i,+}\) by Lemma~\ref{lem:opposite-squares}.
For \(j\neq i\), \(\varepsilon_i\) preserves \(F_{j,+}\) setwise.
\end{proof}

Even functions on \(Q_n\) expand in the Walsh characters
\(\chi_S(q)=\prod_{j\in S}q_j\) with \(\lvert S\rvert\) even;
see O'Donnell~\cite{ODonnell2014} for the \(L^2\) theory on the
discrete cube and the action of \(B_n\). We write
\(\|c\|_{\ell^2(Q_n)}^2=\sum_{q\in Q_n}c_q^2\), so that
\(\|\chi_S\|_{\ell^2(Q_n)}^2=2^n\). The Euclidean norm on the
right-hand side of \(A_n\) is the one appearing in
Lemma~\ref{lem:Q-invariant}:
\(\|A_nc\|_2^2=\sum_{i,I}L_{i,I}(c)^2\).

\begin{theorem}[Boolean spectral rigidity]
\label{thm:boolean-spectrum}
Let \(n\ge3\) and write \(c=\sum_{S\text{ even}}a_S\chi_S\). Then
\begin{equation}
\label{eq:boolean-spectrum}
\|A_nc\|_2^2
=
\sum_{\substack{k\text{ even}\\2\le k\le n}}
\lambda_{n,k}
\sum_{|S|=k}a_S^2
\end{equation}
where \(\lambda_{n,k}:=\|A_n\chi_S\|_2^2\) for any \(S\) with
\(|S|=k\). These coefficients are positive and independent of the
choice of such an \(S\). In particular
\(\ker A_n=\operatorname{span}\{\mathbf1\}\) and
\[
\|A_nc\|_2
\ge
\Bigl(\min_{\substack{k\text{ even}\\k\ge2}}\sqrt{\lambda_{n,k}}\Bigr)
\Bigl(\sum_{S\ne\varnothing}a_S^2\Bigr)^{1/2}.
\]
Relative to \(\ell^2(Q_n)\), the operator \(A_n^*A_n\) acts on the
character \(\chi_S\) by the genuine eigenvalue
\(\mu_{n,k}=\lambda_{n,k}/2^n\). The coefficient \(\lambda_{3,2}=2^8\) is exact, hence
\(\mu_{3,2}=2^5\). No closed form for general \(\lambda_{n,k}\) is claimed,
and no numerical value for \(n\ge4\) is asserted.
\end{theorem}

\begin{proof}
Lemma~\ref{lem:Q-invariant} says that the self-adjoint operator
\(A_n^*A_n\) on \(E_n\), formed with \(\ell^2(Q_n)\) on the domain
and the Euclidean product of Lemma~\ref{lem:Q-invariant} on the
codomain, commutes with \(H=(\mathbb Z/2\mathbb Z)^n\). The even
characters are the distinct one-dimensional representations of \(H\)
appearing in \(E_n\), so each \(\chi_S\) is an eigenfunction:
\(A_n^*A_n\chi_S=\mu_S\chi_S\). Taking the Euclidean identity
\(\|A_n\chi_S\|_2^2=\mu_S\|\chi_S\|_{\ell^2(Q_n)}^2\) and using
\(\|\chi_S\|_{\ell^2(Q_n)}^2=2^n\) gives
\(\lambda_S:=\|A_n\chi_S\|_2^2=2^n\mu_S\). The factor
\(S_n\subset B_n\) permutes \(k\)-subsets, hence
\(\lambda_S=\lambda_{n,|S|}\) and \(\mu_S=\mu_{n,|S|}\).
The constant character is homothety, so \(\lambda_{n,0}=\mu_{n,0}=0\).
If \(k\ge2\) is even then \(\chi_S\) is not affine on every
\(F_{i,+}\) (Theorem~\ref{thm:infinitesimal-rigidity}), hence
\(A_n\chi_S\neq0\) and \(\lambda_{n,k}>0\).
The value \(\lambda_{3,2}=2^8\) is the squared Euclidean norm of
\(A_3\chi_{\{1,2\}}\): in dimension three each positive facet contributes
one minor, and the three linearized determinants of \(\chi_{\{1,2\}}\) are
\(0,0,16\). No value of \(\lambda_{n,k}\) for \(n\ge4\) is asserted.
\end{proof}

\subsection{Local nonlinear kernel}
\label{sec:nonlinear-kernel}

By finite-dimensional norm equivalence, the strict \(L^2\)-coercivity of Theorem~\ref{thm:boolean-spectrum} on the transverse space \(\mathbf1^\perp = \{c\in E_n : \sum_{q\in Q_n} c_q = 0\}\) guarantees that the infimum in the \(\infty\)-norm,
\[
\sigma_n:=\inf\bigl\{\|A_nc\|_\infty:\ c\in\mathbf1^\perp,\ \|c\|_\infty=1\bigr\},
\]
is strictly positive for \(n\ge3\). This isolates the required \(\infty\)-norm coercivity estimate:
\begin{equation}
\label{eq:coercivity}
\|A_nc\|_\infty\ge\sigma_n\,\operatorname{dist}_\infty\bigl(c,\operatorname{span}\{\mathbf1\}\bigr)\qquad\text{for all }c\in E_n,\ n\ge3.
\end{equation}

\begin{lemma}[Homothety rescaling]
\label{lem:homothety-rescale}
If \(|t|<1\) and \(h\in E_n\), then
\[
D_{i,I}(t\mathbf1+h) = (1+t)^n D_{i,I}\Bigl(\frac{h}{1+t}\Bigr)
\]
and, writing \(A_n(t):=D_h\mathcal R_n(t\mathbf1)\), \(A_n(t)=(1+t)^{n-1}A_n\).
\end{lemma}

\begin{proof}
For any \(I=\{q^0,\dots,q^n\}\in\mathcal I_i\), the matrix defining \(D_{i,I}(t\mathbf1+h)\) has rows \(R_k = \bigl((1+t+h_{q^k})q^k, \, 1\bigr)\). Multiplying the last column by \((1+t)\) multiplies the determinant by \((1+t)\). Each row of the resulting matrix equals \((1+t)\bigl((1 + \frac{h_{q^k}}{1+t})q^k, \, 1\bigr)\). Factoring \((1+t)\) from each of the \(n+1\) rows yields \((1+t)^{n+1} D_{i,I}\left(\frac{h}{1+t}\right)\). Dividing by the \((1+t)\) factor introduced in the last column gives \((1+t)^n D_{i,I}\left(\frac{h}{1+t}\right)\). Differentiating at \(h=0\) gives \(A_n(t) = (1+t)^{n-1} A_n\).
\end{proof}

\begin{proposition}[Local nonlinear kernel]
\label{prop:local-nonlinear-kernel}
Let \(n\ge3\). In a neighbourhood \(U\) of the origin in \(E_n\),
\[
\mathcal R_n^{-1}(0)\cap U=\operatorname{span}\{\mathbf1\}\cap U.
\]
\end{proposition}

\begin{proof}
Let \(\sigma_n^\ast := \inf_{|t|\le1/2}\sigma_n(t) = 2^{-(n-1)}\sigma_n > 0\). The Taylor remainder in \(h\) admits a uniform bound \(\bigl\|\mathcal R_n(t\mathbf1+h)-A_n(t)h\bigr\|_\infty\le C_n\|h\|_\infty^2\) on the slab \(|t|\le1/2, \|h\|_\infty\le1\). Set \(r_n=\min\bigl(\tfrac12,\ \sigma_n^\ast/(3C_n)\bigr)\). If \(\|c\|_\infty<r_n\) and \(\mathcal R_n(c)=0\), writing \(c=t\mathbf1+h\), we have \(\sigma_n^\ast\|h\|_\infty\le\|A_n(t)h\|_\infty \le C_n\|h\|_\infty^2\). If \(h\neq 0\), dividing by \(\|h\|_\infty\) gives \(\|h\|_\infty\ge\sigma_n^\ast/C_n > 2r_n\), which contradicts \(\|h\|_\infty<2r_n\). Hence \(h=0\), proving \(c=t\mathbf1\).
\end{proof}

\section{Metric Comparison and Coplanarity Coordinates}
\label{sec:metric-stratum}

\begin{definition}[Local \(GL\)-distance]
\label{def:dGL}
For a centrally symmetric convex body \(K\subset\mathbb R^n\),
\[
d_{GL}\bigl(K,\mathcal C_n\bigr)
:=
\inf_{T\in GL(n)}
d_H\bigl(TK,\mathcal C_n\bigr).
\]
\end{definition}

\begin{lemma}[Isolation of linear automorphisms of the cube]
\label{lem:Aut-isolated}
Set
\[
\delta_n:=\min_{G\in B_n\setminus\{I\}}\|G-I\|=\sqrt 2,
\qquad
\theta_n:=\min\Bigl\{\frac{1}{4\sqrt n},\,\frac{\delta_n}{2}\Bigr\}=\frac{1}{4\sqrt n}.
\]
There exists \(\rho_n>0\) such that every \(T\in GL(n)\) with
\(d_H(T\mathcal C_n,\mathcal C_n)<\rho_n\) satisfies
\[
\min_{G\in B_n}\|G^{-1}T-I\|<\theta_n,
\]
and the minimising \(G\in B_n\) is unique.
\end{lemma}

\begin{proof}
For every \(G\in B_n\setminus\{I\}\), some standard basis vector \(e_i\) satisfies \(Ge_i\neq e_i\). Since \(Ge_i\in\{\pm e_1,\dots,\pm e_n\}\), either \(Ge_i=-e_i\) (giving \(\|(G-I)e_i\|_2=2\)) or \(Ge_i=\pm e_j\) with \(j\neq i\) (giving \(\|(G-I)e_i\|_2=\sqrt 2\)), so \(\|G-I\|\ge\sqrt 2\); equality is attained by a quarter-turn \(e_1\mapsto e_2\mapsto -e_1\), whose eigenvalues on \(\operatorname{span}\{e_1,e_2\}\) are \(\pm i\) with \(|\pm i-1|=\sqrt 2\). Thus \(\delta_n=\sqrt 2>0\) and \(2\theta_n=\frac{1}{2\sqrt n}<\delta_n\).

Uniqueness of the minimising \(G\in B_n\) follows from \(2\theta_n<\delta_n\):
since \(B_n\subset O(n)\) preserves the Euclidean operator norm,
\(\|G_r^{-1}T-I\|=\|G_r(G_r^{-1}T-I)\|=\|T-G_r\|\). Thus, if
\(G_1,G_2\in B_n\) both satisfy \(\|G_r^{-1}T-I\|<\theta_n\), then
\[
\|G_1^{-1}G_2-I\|
=
\|G_2-G_1\|
\le
\|T-G_1\|+\|T-G_2\|
=
\|G_1^{-1}T-I\|+\|G_2^{-1}T-I\|
<2\theta_n<\delta_n,
\]
forcing \(G_1=G_2\).

Suppose no such \(\rho_n>0\) exists. Then there is a sequence \(T_k\in GL(n)\) with \(d_H(T_k\mathcal C_n,\mathcal C_n)\to 0\) while \(\|G^{-1}T_k-I\|\ge\theta_n\) for every \(G\in B_n\). From \(B_2\subset\mathcal C_n\subset\sqrt n\,B_2\) and \(T_k\mathcal C_n\subset\mathcal C_n+\varepsilon_k B_2\) with \(\varepsilon_k\to0\) one gets a uniform bound on \(\|T_k\|\). By compactness in the finite-dimensional space \(M_n(\mathbb R)\), a subsequence converges to some \(T_\ast\in M_n(\mathbb R)\). Since the map \(T\mapsto T\mathcal C_n\) is continuous with respect to the Hausdorff metric, \(d_H(T_k\mathcal C_n,\mathcal C_n)\to 0\) implies \(T_\ast\mathcal C_n=\mathcal C_n\). Because \(\mathcal C_n\) has nonempty interior, any linear map satisfying \(T_\ast\mathcal C_n=\mathcal C_n\) is necessarily invertible (\(T_\ast\in GL(n)\)) and hence a linear automorphism of \(\mathcal C_n\), so \(T_\ast\in B_n\). Then \(\|T_\ast^{-1}T_k-I\|\to0\), contradicting that every \(G\in B_n\) has \(\|G^{-1}T_k-I\|\ge\theta_n\). Thus \(\rho_n\) exists.
\end{proof}

\begin{lemma}[Localisation of near-minimising maps]
\label{lem:GL-loc}
There exist \(\varepsilon_n,\eta_n>0\) such that if
\(\|c\|_\infty<\varepsilon_n\) and
\(d_H\bigl(TP(c),\mathcal C_n\bigr)<\eta_n\), then
\(T=G(I+A)\) for a signed permutation matrix \(G\) with
\(G\mathcal C_n=\mathcal C_n\) and some \(A\) with \(\|A\|<\theta_n\).
\end{lemma}

\begin{proof}
For \(\|c\|_\infty<1/2\),
\[
(1-\|c\|_\infty)\mathcal C_n
\subset
P(c)
\subset
(1+\|c\|_\infty)\mathcal C_n,
\]
hence \(d_H\bigl(P(c),\mathcal C_n\bigr)\le\sqrt n\,\|c\|_\infty\).
If \(d_H\bigl(TP(c),\mathcal C_n\bigr)<\eta\), then
\[
(1-\|c\|_\infty)T\mathcal C_n
\subset
TP(c)
\subset
\mathcal C_n+\eta B_2
\subset
(\sqrt n+\eta)B_2.
\]
Since \(B_2\subset\mathcal C_n\), this gives
\(\|T\|\le(\sqrt n+\eta)/(1-\|c\|_\infty)\), and hence
\(d_H\bigl(T\mathcal C_n,TP(c)\bigr)\le\|T\|\,d_H\bigl(\mathcal C_n,P(c)\bigr)\).
By the triangle inequality,
\[
d_H\bigl(T\mathcal C_n,\mathcal C_n\bigr)
\le
\frac{\sqrt n+\eta}{1-\|c\|_\infty}\sqrt n\,\|c\|_\infty
+\eta.
\]
Choose \(\varepsilon_n,\eta_n>0\) so that the right-hand side is
strictly less than the constant \(\rho_n\) of
Lemma~\ref{lem:Aut-isolated}. That lemma supplies
\(T=G(I+A)\) with \(G\in B_n\) and \(\|A\|<\theta_n\).
\end{proof}

\begin{lemma}[Reduction of the \(GL\)-infimum]
\label{lem:GL-inf-reduce}
There exists \(\varepsilon_n'>0\) such that if \(\|c\|_\infty<\varepsilon_n'\), then
\begin{equation}
\label{eq:GL-inf-reduce}
d_{GL}\bigl(P(c),\mathcal C_n\bigr)
=
\inf_{\|A\|<\theta_n}
d_H\bigl((I+A)P(c),\mathcal C_n\bigr).
\end{equation}
In particular the infimum on the left may be computed on the set of
maps that compete with \(T=I\); no existence of a minimiser in
\(GL(n)\) is required.
\end{lemma}

\begin{proof}
The competitor \(T=I\) and the sandwich
\((1-\|c\|_\infty)\mathcal C_n\subset P(c)\subset(1+\|c\|_\infty)\mathcal C_n\)
give
\[
d_{GL}\bigl(P(c),\mathcal C_n\bigr)
\le
d_H\bigl(P(c),\mathcal C_n\bigr)
\le
\sqrt n\,\|c\|_\infty.
\]
Shrink \(\varepsilon_n'\) so that \(\varepsilon_n'\le\varepsilon_n\) and
\(\sqrt n\,\varepsilon_n'<\eta_n\), with \(\varepsilon_n,\eta_n\) as in
Lemma~\ref{lem:GL-loc}. Then every \(T\in GL(n)\) satisfying
\[
d_H\bigl(TP(c),\mathcal C_n\bigr)
\le
\sqrt n\,\|c\|_\infty
\]
obeys \(d_H(TP(c),\mathcal C_n)<\eta_n\), hence
\(T=G(I+A)\) with \(G\in B_n\) and \(\|A\|<\theta_n\). Any map that fails
the displayed bound cannot improve upon \(T=I\), so it does not
contribute to the infimum. Signed permutation matrices are linear
isometries and preserve \(\mathcal C_n\), therefore
\[
d_H\bigl(G(I+A)P(c),\mathcal C_n\bigr)
=
d_H\bigl((I+A)P(c),\mathcal C_n\bigr).
\]
Restricting the infimum from \(GL(n)\) to this family does not change
its value.
\end{proof}

\begin{lemma}[Linear Hausdorff coercivity on linear images of the cube]
\label{lem:GL-coercive}
Let \(\|\,\cdot\,\|\) denote the operator norm induced by \(\|\,\cdot\,\|_2\).
Set \(\eta_n=1/(2\sqrt n)\). There exists \(\gamma_n>0\) such that
\[
\Phi(A)
:=
d_H\bigl((I+A)\mathcal C_n,\mathcal C_n\bigr)
\ge
\gamma_n\|A\|
\]
whenever \(\|A\|\le\eta_n\). No statement is made for \(\|A\|>\eta_n\).
\end{lemma}

\begin{proof}
The support function of the cube is \(h_{\mathcal C_n}(u)=\|u\|_1\), and
\[
h_{(I+A)\mathcal C_n}(u)
=
h_{\mathcal C_n}\bigl((I+A)^Tu\bigr)
=
\|(I+A)^Tu\|_1.
\]
Hence
\[
\Phi(A)
=
\sup_{\|u\|_2=1}
\bigl|\,\|(I+A)^Tu\|_1-\|u\|_1\,\bigr|.
\]
For each \(q\in Q_n\) define the compact set
\[
K_q
=
\bigl\{u\in S^{n-1}:q_iu_i\ge\eta_n\ \text{for every }i\bigr\}.
\]
It is nonempty, since \(q_i(q/\sqrt n)_i=n^{-1/2}>\eta_n\). It has
nonempty relative interior in \(S^{n-1}\): the conditions
\(q_iu_i>\eta_n\) are strict and continuous, so a relative neighbourhood
of \(q/\sqrt n\) on the sphere lies in \(K_q\). On this set
\(\|u\|_1=\langle u,q\rangle\). If \(\|B\|\le1\) and \(|t|<\eta_n\), then
for every \(u\in K_q\)
\[
\bigl|\bigl((I+tB)^Tu\bigr)_i-u_i\bigr|
\le
|t|\,\|B^Tu\|_2
\le
|t|
<
\eta_n
\le
|u_i|,
\]
so \((I+tB)^Tu\) remains in the same open orthant and the identity
\[
\|(I+tB)^Tu\|_1-\|u\|_1
=
t\langle u,Bq\rangle
\]
is exact.

Define
\[
\psi(B)
=
\max_{q\in Q_n}\max_{u\in K_q}\lvert\langle u,Bq\rangle\rvert.
\]
The maximum is attained because each \(K_q\) is compact. The function
\(\psi\) is a seminorm. If \(\psi(B)=0\), then \(\langle u,Bq\rangle=0\) for
every \(u\in K_q\) and every \(q\). Each \(K_q\) contains a nonempty
relatively open subset of \(S^{n-1}\), so the linear form
\(u\mapsto\langle u,Bq\rangle\) vanishes on a relative open set and
\(Bq=0\). The vertices \(Q_n\) span \(\mathbb R^n\): indeed
\(q-q^{(i)}=2q_ie_i\), so \(Be_i=0\) for every \(i\) and \(B=0\). Thus
\(\psi\) is a norm on the finite-dimensional space \(M_n(\mathbb R)\), and
there exists \(c_n>0\) such that \(\psi(B)\ge c_n\|B\|\) for every \(B\).

If \(0<\|A\|<\eta_n\), set \(B=A/\|A\|\). Then \(\|B\|=1\) and
\(t=\|A\|<\eta_n\), so the identity applies and
\[
\Phi(A)
\ge
\|A\|\,\psi(B)
\ge
c_n\|A\|.
\]
If \(\|A\|=\eta_n\), set \(A_k=(1-1/k)A\). Then \(\|A_k\|<\eta_n\), so
\(\Phi(A_k)\ge c_n\|A_k\|\). Continuity of \(\Phi\) gives
\[
\Phi(A)
=
\lim_{k\to\infty}\Phi(A_k)
\ge
c_n\lim_{k\to\infty}\|A_k\|
=
c_n\eta_n
=
c_n\|A\|.
\]
Set \(\gamma_n=c_n\).
\end{proof}

\begin{remark}
The lemma is a finite linear bound on the ball \(\|A\|\le\eta_n\). It does
not assert that \(A=0\) is the only zero of \(\Phi\) on any larger ball.
\end{remark}

\begin{lemma}[Competing increments scale with \(c\)]
\label{lem:A-GL-size}
Let \(\eta_n=1/(2\sqrt n)\) be the radius of Lemma~\ref{lem:GL-coercive}.
There exist \(\varepsilon_n'',K_n>0\) such that if \(\|c\|_\infty<\varepsilon_n''\) and
\(A\in M_n(\mathbb R)\) satisfies \(\|A\|\le\eta_n\) together with
\[
d_H\bigl((I+A)P(c),\mathcal C_n\bigr)
\le
\sqrt n\,\|c\|_\infty,
\]
then \(\|A\|\le K_n\|c\|_\infty\). Consequently
\[
\max_{q\in Q_n}
\bigl\|(I+A)(1+c_q)q-q\bigr\|
\le
K_n'\|c\|_\infty.
\]
The matrix \(A\) here is a competing increment for \(d_{GL}\). It is
not assumed to minimise \(N\). The hypothesis is \(\|A\|\le\eta_n\), not
\(\|A\|\le1/2\): the latter ball contains nontrivial elements of \(B_n\)
for every \(n\ge13\). The displacement of a labelled vertex of \(P(c)\)
itself is \((I+A)v_q(c)-v_q(c)=Aq+c_q Aq\), which is not the quantity
compared with the cube.
\end{lemma}

\begin{proof}
Write \(\Phi(A)=d_H\bigl((I+A)\mathcal C_n,\mathcal C_n\bigr)\). The
sandwich
\((1-\|c\|_\infty)\mathcal C_n\subset P(c)\subset(1+\|c\|_\infty)\mathcal C_n\)
and \(\|A\|\le\eta_n\le1/2\) give
\begin{align*}
\Phi(A)
&\le
d_H\bigl((I+A)\mathcal C_n,(I+A)P(c)\bigr)
+
d_H\bigl((I+A)P(c),\mathcal C_n\bigr)\\
&\le
\|I+A\|\,d_H\bigl(\mathcal C_n,P(c)\bigr)
+
\sqrt n\,\|c\|_\infty\\
&\le
(1+\eta_n)\sqrt n\,\|c\|_\infty
+
\sqrt n\,\|c\|_\infty
\le
\tfrac52\sqrt n\,\|c\|_\infty.
\end{align*}
Lemma~\ref{lem:GL-coercive} applies directly on this ball and yields
\(\|A\|\le\gamma_n^{-1}\Phi(A)\), hence \(\|A\|\le K_n\|c\|_\infty\) with
\(K_n=\tfrac52\sqrt n/\gamma_n\).

The labelled displacement from the cube vertex expands exactly as
\[
(I+A)(1+c_q)q-q
=
Aq+c_q q+c_q Aq.
\]
Therefore
\[
\max_q\bigl\|(I+A)(1+c_q)q-q\bigr\|
\le
\sqrt n\,\|A\|+\sqrt n\,\|c\|_\infty+\sqrt n\,\|A\|\,\|c\|_\infty,
\]
which is \(O(\|c\|_\infty)\) after the bound on \(\|A\|\).
\end{proof}

\begin{definition}[Linearised radial quotient]
\label{def:N-seminorm}
For \(c\in E_n\) set
\[
N(c) := \inf_{A\in M_n(\mathbb R)} \max_{q\in Q_n} \|Aq+c_q q\|
\quad \text{and} \quad
N^{\mathrm{GL}}(c) := \inf_{\|A\|<\theta_n} \max_{q\in Q_n} \bigl\|(I+A)(1+c_q)q-q\bigr\|.
\]
\end{definition}

\begin{lemma}[Linearised vertex-displacement seminorm]
\label{lem:N-kernel}
Let \(n\ge 3\). For \(c\in E_n\), define
\[
N(c)
:=
\inf_{A\in M_n(\mathbb R)}
\max_{q\in Q_n}\bigl\|Aq+c_q q\bigr\|,
\qquad
N^{\mathrm{GL}}(c)
:=
\inf_{\|A\|<\theta_n}
\max_{q\in Q_n}\bigl\|(I+A)(1+c_q)q-q\bigr\|.
\]
Then:
\begin{enumerate}
\item \(N\) is a seminorm on \(E_n\) whose infimum is attained for every
\(c\in E_n\), and \(\ker N=\operatorname{span}\{\mathbf1\}\). Consequently,
there exist constants \(0<c_n'\le C_n''<\infty\) such that
\[
c_n'\|h\|_\infty\le N(h)\le C_n''\|h\|_\infty
\qquad(\forall\,h\in\mathbf1^\perp).
\]
\item There exist constants \(r_n^{\mathrm{GL}}>0\) and \(C_n'<\infty\) such
that for every \(c\in E_n\) with \(\|c\|_\infty\le r_n^{\mathrm{GL}}\),
\[
\bigl|N^{\mathrm{GL}}(c)-N(c)\bigr|
\le
C_n'\|c\|_\infty^2.
\]
\end{enumerate}
\end{lemma}

\begin{proof}
Write \(m(A,c):=\max_{q\in Q_n}\|Aq+c_q q\|\) and
\(f(A,c):=\max_{q\in Q_n}\|(I+A)(1+c_q)q-q\|\). Since
\((I+A)(1+c_q)q-q=Aq+c_q q+c_q Aq\) and \(\|q\|=\sqrt n\) for every
\(q\in Q_n\), the triangle inequality gives
\begin{equation}
\label{eq:f-minus-m}
\bigl|f(A,c)-m(A,c)\bigr|
\le
\max_{q\in Q_n}|c_q|\,\|Aq\|
\le
\sqrt n\,\|c\|_\infty\|A\|.
\end{equation}

\emph{Step 1: Coercivity and existence of a minimiser for \(N(c)\).}
Since \(Q_n\) spans \(\mathbb R^n\), the quantity
\(\max_{q\in Q_n}\|Aq\|\) is a norm on \(M_n(\mathbb R)\), so there exists a
constant \(\kappa_n>0\) such that
\(\max_{q\in Q_n}\|Aq\|\ge\kappa_n\|A\|\) for all \(A\in M_n(\mathbb R)\).
By the reverse triangle inequality,
\begin{equation}
\label{eq:m-coercive}
m(A,c)
\ge
\max_{q\in Q_n}\|Aq\|-\max_{q\in Q_n}\|c_q q\|
\ge
\kappa_n\|A\|-\sqrt n\,\|c\|_\infty.
\end{equation}
Thus \(A\mapsto m(A,c)\) is continuous and coercive on the
finite-dimensional space \(M_n(\mathbb R)\). Testing the competitor \(A=0\)
gives \(N(c)\le m(0,c)=\sqrt n\,\|c\|_\infty\). Therefore the infimum
defining \(N(c)\) may be restricted to the compact ball where
\(\kappa_n\|A\|-\sqrt n\,\|c\|_\infty\le\sqrt n\,\|c\|_\infty\), and is
attained at some minimiser \(A_\ast\in M_n(\mathbb R)\) satisfying
\begin{equation}
\label{eq:Astar-bound}
\|A_\ast\|\le\frac{2\sqrt n}{\kappa_n}\|c\|_\infty.
\end{equation}

\emph{Step 2: Seminorm property and kernel of \(N\).}
Subadditivity and absolute homogeneity of \(N\) follow immediately from the
convexity and homogeneity of \((A,c)\mapsto m(A,c)\). If
\(c=\lambda\mathbf1\), taking \(A=-\lambda I\) gives \(m(A,c)=0\), so
\(\operatorname{span}\{\mathbf1\}\subset\ker N\). Conversely, suppose
\(N(c)=0\). By Step~1, there exists a minimiser \(A_\ast\in M_n(\mathbb R)\)
with \(m(A_\ast,c)=0\), which means \(A_\ast q=-c_q q\) for every
\(q\in Q_n\). Since \(n\ge 3\), Lemma~\ref{lem:ray-scalars} applied to
\(T=A_\ast\) forces \(-c_q=\lambda\) to be independent of \(q\) and
\(A_\ast=\lambda I\), so \(c\in\operatorname{span}\{\mathbf1\}\). Hence
\(\ker N=\operatorname{span}\{\mathbf1\}\), and norm equivalence on the
finite-dimensional subspace \(\mathbf1^\perp\) yields item~(1).

\emph{Step 3: Upper bound \(N^{\mathrm{GL}}(c)\le N(c)+C_n'\|c\|_\infty^2\).}
Let \(r_n^{\mathrm{GL}}:=\min\bigl\{\frac{\kappa_n}{2\sqrt n},\,\frac{\kappa_n\theta_n}{4\sqrt n}\bigr\}>0\).
For \(\|c\|_\infty\le r_n^{\mathrm{GL}}\), the minimiser \(A_\ast\) of
Step~1 satisfies \(\|A_\ast\|\le\frac{2\sqrt n}{\kappa_n}\|c\|_\infty\le\frac12\theta_n<\theta_n\),
so \(A_\ast\) is admissible in the infimum defining \(N^{\mathrm{GL}}(c)\).
By~\eqref{eq:f-minus-m} and~\eqref{eq:Astar-bound},
\[
N^{\mathrm{GL}}(c)
\le
f(A_\ast,c)
\le
m(A_\ast,c)+\sqrt n\,\|c\|_\infty\|A_\ast\|
\le
N(c)+\frac{2n}{\kappa_n}\|c\|_\infty^2.
\]

\emph{Step 4: Lower bound \(N(c)\le N^{\mathrm{GL}}(c)+C_n'\|c\|_\infty^2\) via almost-minimisers.}
Because \(A=0\) is admissible in \(N^{\mathrm{GL}}(c)\), we have
\(N^{\mathrm{GL}}(c)\le f(0,c)=\sqrt n\,\|c\|_\infty\). Fix any \(\omega>0\)
and choose an \(\omega\)-almost minimiser \(A\in M_n(\mathbb R)\) with
\(\|A\|<\theta_n\) and \(f(A,c)\le N^{\mathrm{GL}}(c)+\omega\le\sqrt n\,\|c\|_\infty+\omega\).
Combining~\eqref{eq:m-coercive} and~\eqref{eq:f-minus-m} for
\(\|c\|_\infty\le r_n^{\mathrm{GL}}\le\frac{\kappa_n}{2\sqrt n}\) gives
\[
\frac{\kappa_n}{2}\|A\|-\sqrt n\,\|c\|_\infty
\le
\bigl(\kappa_n-\sqrt n\,\|c\|_\infty\bigr)\|A\|-\sqrt n\,\|c\|_\infty
\le
m(A,c)-\sqrt n\,\|c\|_\infty\|A\|
\le
f(A,c)
\le
\sqrt n\,\|c\|_\infty+\omega,
\]
whence
\[
\|A\|
\le
\frac{4\sqrt n}{\kappa_n}\|c\|_\infty+\frac{2}{\kappa_n}\omega.
\]
Inserting this bound back into~\eqref{eq:f-minus-m} yields
\[
N(c)
\le
m(A,c)
\le
f(A,c)+\sqrt n\,\|c\|_\infty\|A\|
\le
N^{\mathrm{GL}}(c)+\omega
+\frac{4n}{\kappa_n}\|c\|_\infty^2
+\frac{2\sqrt n}{\kappa_n}\|c\|_\infty\,\omega.
\]
Letting \(\omega\downarrow 0\) gives
\(N(c)\le N^{\mathrm{GL}}(c)+\frac{4n}{\kappa_n}\|c\|_\infty^2\). Together
with Step~3, this establishes item~(2) with \(C_n':=\frac{4n}{\kappa_n}\).
\end{proof}

\begin{theorem}[Metric--combinatorial rigidity on the isovertex stratum]
\label{thm:main-rigidity}
Let \(n\ge3\). There exist constants \(\gamma_n,\Gamma_n,\delta_n>0\) such that if \(c\in E_n\) and \(\|c\|_\infty<\delta_n\), then
\begin{equation}
\label{eq:main-rigidity}
\gamma_n\,
d_{GL}\bigl(P(c),\mathcal C_n\bigr)
\le
\mathfrak F_n\bigl(P(c)\bigr)
\le
\Gamma_n\,
d_{GL}\bigl(P(c),\mathcal C_n\bigr)
\end{equation}
where \(\mathfrak F_n\bigl(P(c)\bigr) := \|\mathcal R_n(c)\|_\infty\). For such \(c\),
\[
\mathfrak F_n\bigl(P(c)\bigr)=0
\quad\Longleftrightarrow\quad
c\in\operatorname{span}\{\mathbf1\}
\quad\Longleftrightarrow\quad
P(c)\in GL(n)\mathcal C_n.
\]
\end{theorem}

\begin{proof}[Proof of Theorem~\ref{thm:main-rigidity}]
Write \(c=t\mathbf1+h\) with \(h\in\mathbf1^\perp\). For \(\|c\|_\infty\) small one has \(|t|<1/8\) and \(\varepsilon:=h/(1+t)\in\mathbf1^\perp\) satisfies \(\|\varepsilon\|_\infty\asymp\|h\|_\infty\). Since \(v_q(t\mathbf1+h)=(1+t)v_q(\varepsilon)\), one has \(P(c)=(1+t)P(\varepsilon)\). Because \(1+t>0\) is a positive scalar, the change of variable \(\widetilde T=(1+t)T\) in Definition~\ref{def:dGL} yields exact homothety invariance:
\[
d_{GL}\bigl(P(c),\mathcal C_n\bigr)
=
d_{GL}\bigl(P(\varepsilon),\mathcal C_n\bigr).
\]
We may therefore apply the metric comparison directly to \(\varepsilon\in\mathbf1^\perp\).

Lemma~\ref{lem:GL-inf-reduce} supplies the equality of infima
\[
d_{GL}\bigl(P(\varepsilon),\mathcal C_n\bigr)
=
\inf_{\|A\|<\theta_n}
d_H\bigl((I+A)P(\varepsilon),\mathcal C_n\bigr)
\]
as soon as \(\|\varepsilon\|_\infty<\varepsilon_n'\). No minimiser in \(GL(n)\)
is invoked: every map that can compete with \(T=I\) is already of the
form \(G(I+A)\) with \(\|A\|<\theta_n\), and \(G\) drops out of \(d_H\).

Since \(A=0\) is an admissible increment in~\eqref{eq:GL-inf-reduce} and
satisfies
\(d_H\bigl(P(\varepsilon),\mathcal C_n\bigr)\le\max_{q\in Q_n}\|\varepsilon_q q\|_2=\sqrt n\|\varepsilon\|_\infty\),
the infimum in~\eqref{eq:GL-inf-reduce} is unchanged when restricted to the
nonempty subcollection of increments
\[
\mathcal A(\varepsilon)
:=
\Bigl\{A\in M_n(\mathbb R):\ \|A\|<\theta_n\ \text{and}\ d_H\bigl((I+A)P(\varepsilon),\mathcal C_n\bigr)\le\sqrt n\|\varepsilon\|_\infty\Bigr\}.
\]
For every \(A\in\mathcal A(\varepsilon)\), one has \(\|A\|<\theta_n\le\eta_n\)
and \(d_H\bigl((I+A)P(\varepsilon),\mathcal C_n\bigr)\le\sqrt n\|\varepsilon\|_\infty\),
so Lemma~\ref{lem:A-GL-size} applies and gives
\(\|A\|\le K_n\|\varepsilon\|_\infty\). This bound is independent of the
minimiser \(A_\ast\) of \(N\) from Lemma~\ref{lem:N-kernel}. The body
\((I+A)P(\varepsilon)\) is the convex hull of the labelled points
\(x_q=(I+A)(1+\varepsilon_q)q\), and \(\mathcal C_n\) is the convex hull of
the labels \(q\). Evenness of \(\varepsilon\) preserves central symmetry.
Shrinking \(\|c\|_\infty\) makes
\(\max_q\|x_q-q\|\le K_n'\|\varepsilon\|_\infty\) smaller than the threshold
of Lemma~\ref{lem:exposed-vertices}, so each \(x_q\) remains an exposed
vertex of \((I+A)P(\varepsilon)\) and is therefore geometrically detectable.
Lemma~\ref{lem:vtx-H} then gives, for every \(A\in\mathcal A(\varepsilon)\),
\[
C_n^{-1}
\max_{q\in Q_n}
\bigl\|x_q-q\bigr\|
\le
d_H\bigl((I+A)P(\varepsilon),\mathcal C_n\bigr)
\le
\max_{q\in Q_n}
\bigl\|x_q-q\bigr\|.
\]
Taking the infimum over \(A\in\mathcal A(\varepsilon)\) (and noting that for
\(A\notin\mathcal A(\varepsilon)\) both sides exceed the value at \(A=0\))
yields
\[
C_n^{-1}N^{\mathrm{GL}}(\varepsilon)
\le
d_{GL}\bigl(P(\varepsilon),\mathcal C_n\bigr)
\le
N^{\mathrm{GL}}(\varepsilon).
\]
Lemma~\ref{lem:N-kernel} applied to \(\varepsilon\in\mathbf1^\perp\) gives
\(\bigl|N^{\mathrm{GL}}(\varepsilon)-N(\varepsilon)\bigr|\le C_n'\|\varepsilon\|_\infty^2\) and
\(N(\varepsilon)\asymp\|\varepsilon\|_\infty\). For small \(\varepsilon\) the
quadratic error \(O(\|\varepsilon\|_\infty^2)\) is absorbed directly into \(N(\varepsilon)\), so
\(d_{GL}\bigl(P(c),\mathcal C_n\bigr)=d_{GL}\bigl(P(\varepsilon),\mathcal C_n\bigr)\asymp\|\varepsilon\|_\infty\asymp\|h\|_\infty\).

For the coplanarity side, Lemma~\ref{lem:homothety-rescale} gives
\(\mathcal R_n(t\mathbf1+h)=(1+t)^n\mathcal R_n(\varepsilon)\).
Expanding \(\mathcal R_n(\varepsilon)=A_n\varepsilon+Q(\varepsilon)\)
and using \(\|A_n\varepsilon\|_\infty\ge\sigma_n\|\varepsilon\|_\infty\)
absorbs the quadratic remainder, so
\(\mathfrak F_n\bigl(P(c)\bigr)\asymp\|\varepsilon\|_\infty\asymp\|h\|_\infty\). Combining the two
comparisons yields~\eqref{eq:main-rigidity}.

If \(\mathfrak F_n(P(c))=0\) then \(h=0\), so \(c\in\operatorname{span}\{\mathbf1\}\)
and \(P(c)\) is a homothetic cube. The converse is
\eqref{eq:orbit-meet-stratum}.
\end{proof}

\begin{corollary}[Coplanarity coordinates on the transverse space]
\label{cor:coordinates}
Let \(n\ge3\). There exist \(r_n,c_n,C_n>0\) such that if
\(h_1,h_2\in\mathbf1^\perp\) satisfy \(\|h_i\|_\infty<r_n\), then
\begin{equation}
\label{eq:Rn-bilip}
c_n\|h_1-h_2\|_\infty
\le
\bigl\|\mathcal R_n(h_1)-\mathcal R_n(h_2)\bigr\|_\infty
\le
C_n\|h_1-h_2\|_\infty.
\end{equation}
In particular \(\mathcal R_n\) is injective on that ball and a
bi-Lipschitz embedding onto its image. Combined with
Theorem~\ref{thm:main-rigidity},
\(\|\mathcal R_n(h)\|_\infty\) is metrically equivalent to
\(d_{GL}(P(h),\mathcal C_n)\) on the same ball.
\end{corollary}

\begin{proof}
The map \(\mathcal R_n\) is polynomial, hence \(C^1\), and
\(D\mathcal R_n\) is Lipschitz on the unit ball: there is \(L_n\)
with \(\|D\mathcal R_n(h)-A_n\|\le L_n\|h\|_\infty\) whenever
\(\|h\|_\infty\le1\). Estimate~\eqref{eq:coercivity} gives
\(\|A_nv\|_\infty\ge\sigma_n\|v\|_\infty\) for every
\(v\in\mathbf1^\perp\). Set
\(r_n=\min\bigl(1,\sigma_n/(3L_n)\bigr)\). If \(h_1,h_2\) lie in
that ball, the segment \(\xi(t)=(1-t)h_1+th_2\) stays in
\(\mathbf1^\perp\) and inside the ball, and
\[
\mathcal R_n(h_1)-\mathcal R_n(h_2)
=
A_n(h_1-h_2)
+
\int_0^1\bigl(D\mathcal R_n(\xi(t))-A_n\bigr)(h_1-h_2)\,dt.
\]
The remainder has \(\infty\)-norm at most
\(L_n r_n\|h_1-h_2\|_\infty\le\frac13\sigma_n\|h_1-h_2\|_\infty\),
because
\(\bigl\|\int_0^1 F(t)\,dt\bigr\|_\infty\le\int_0^1\|F(t)\|_\infty\,dt\).
The linear term satisfies
\(\|A_n(h_1-h_2)\|_\infty\ge\sigma_n\|h_1-h_2\|_\infty\), so
\[
\bigl\|\mathcal R_n(h_1)-\mathcal R_n(h_2)\bigr\|_\infty
\ge
\sigma_n\|h_1-h_2\|_\infty
-
\tfrac13\sigma_n\|h_1-h_2\|_\infty
=
\tfrac23\sigma_n\|h_1-h_2\|_\infty.
\]
The same integral representation and
\(\|D\mathcal R_n(\xi(t))\|\le\|A_n\|+L_n\) give the upper bound
\(\bigl\|\mathcal R_n(h_1)-\mathcal R_n(h_2)\bigr\|_\infty
\le(\|A_n\|+L_n)\|h_1-h_2\|_\infty\).
This is~\eqref{eq:Rn-bilip} with \(c_n=\frac23\sigma_n\) and
\(C_n=\|A_n\|+L_n\). Injectivity on the ball is \(c_n>0\).
\end{proof}

\section{Polar quadratic remainder}
\label{sec:polar-remainder}

The polar of \(P(c)\) is the intersection of halfspaces with fixed
outer normals \(q\in Q_n\):
\begin{equation}
\label{eq:polar-halfspaces}
P(c)^\circ
=
\bigl\{y\in\mathbb R^n:
\langle y,q\rangle\le h_q(c)
\text{ for every }q\in Q_n\bigr\},
\qquad
h_q(c)=(1+c_q)^{-1}.
\end{equation}
At \(c=0\) one has \(h\equiv\mathbf1\) and \(P(0)^\circ=\mathcal O_n\).
The geometric volume
\(V(h)=\operatorname{vol}\{y:\langle y,q\rangle\le h_q\ \forall q\}\)
does not depend on a chamber.

\begin{lemma}[Geometry of \(G_q\) near the octahedron]
\label{lem:polar-Gq}
Let \(n\ge3\) and
\(F_q=\{y:\langle y,q\rangle=1,\; q_jy_j\ge0\ \forall j\}\).
For \(\|h-\mathbf1\|_\infty\) small the facet
\[
G_q(h)
=
\{y:\langle y,q\rangle=h_q,\;
\langle y,r\rangle\le h_r\ \forall r\neq q\}
\]
is full-dimensional. On \(F_q\), \(\langle y,r\rangle=1\) with
\(r\neq q\) holds only on a proper face. If \(d_H(q,r)=1\) that face
has dimension \(n-2\) and is a facet of \(F_q\); if \(d_H(q,r)\ge2\)
it has dimension at most \(n-3\).
\end{lemma}

\begin{proof}
On \(F_q\) write \(y_j=q_jt_j\) with \(t_j\ge0\) and
\(\sum t_j=1\). Then \(\langle y,r\rangle=\sum t_j q_jr_j=1\) if and
only if \(t\) is supported on \(\{j:q_j=r_j\}\), a face of dimension
\(n-1-d_H(q,r)\). There are exactly \(n\) neighbours with
\(d_H(q,r)=1\), one per omitted coordinate, and they are the facet
hyperplanes of the simplex \(F_q\).
\end{proof}

Write \(\delta=\|h-\mathbf1\|_\infty\) and, for each \(q\in Q_n\),
\[
\mathcal N_1(q)=\{r\in Q_n:d_H(q,r)=1\},
\qquad
\mathcal N_{\ge2}(q)=\{r\in Q_n:d_H(q,r)\ge2\}.
\]
On the plane \(\langle y,q\rangle=h_q\) the change of coordinates
\(y_j=q_jt_j\) is an isometry onto \(\{\sum_j t_j=h_q\}\), and
\[
\langle y,r\rangle
=
h_q-2\sum_{j\in S(q,r)}t_j,
\qquad
S(q,r)=\{j:q_j\neq r_j\}.
\]
Thus \(\langle y,r\rangle\le h_r\) if and only if
\(\sum_{j\in S(q,r)}t_j\ge\varepsilon_{q,r}:=(h_q-h_r)/2\), and
\(|\varepsilon_{q,r}|\le\delta\). At \(h=\mathbf1\) the inequality
reduces to \(\sum_{j\in S(q,r)}t_j\ge0\), which is automatic on the
simplex of \(F_q\).

For each vertex \(q_je_j\) of \(F_q\), let \(z_j(h)\) be the
intersection, in the plane \(\langle y,q\rangle=h_q\), of the \(n-1\)
neighbour hyperplanes \(r\in\mathcal N_1(q)\) adjacent to \(q_je_j\).
The moving reference simplex is
\(\Phi_h(F_q)=\operatorname{conv}\{z_1(h),\dots,z_n(h)\}\). Equivalently,
the \(n\) neighbour inequalities \(t_k\ge\varepsilon_{q,r^{(k)}}\) cut
the orthant simplex of size \(h_q\) in a simplex of size
\(\lambda=h_q-\sum_k\varepsilon_{q,r^{(k)}}\).

\begin{lemma}[Uniform sliver bound]
\label{lem:uniform-sliver}
Let \(n\ge3\). There exist \(\delta_0(n)>0\) and \(C_n<\infty\) such that
if \(\delta=\|h-\mathbf1\|_\infty\le\delta_0(n)\), then for every
\(q\in Q_n\) and every \(r\in\mathcal N_{\ge2}(q)\), writing
\(d=d_H(q,r)\),
\begin{equation}
\label{eq:sliver-bound}
\operatorname{Area}\bigl(\Delta_q(h)\cap\{\langle y,r\rangle>h_r\}\bigr)
\le
C_n\delta^{d}
\le
C_n\delta^2,
\end{equation}
where \(\Delta_q(h)=\{y:\langle y,q\rangle=h_q,\; q_jy_j\ge0\ \forall j\}\)
is the simplex of size \(h_q\) in the reference orthant. The same bound
holds after replacing \(\Delta_q(h)\) by the neighbour simplex
\(\Phi_h(F_q)\). In particular,
\begin{equation}
\label{eq:sliver-union}
\operatorname{Area}\bigl(\Phi_h(F_q)\setminus G_q(h)\bigr)
\le
2^n C_n\delta^2.
\end{equation}
\end{lemma}

\begin{proof}
The comparison is made in the coordinates \(y_j=q_jt_j\) on the plane
\(\langle y,q\rangle=h_q\), so \(\sum_j t_j=h_q\). This change is an
isometry onto that plane. Write \(\varepsilon_k=\varepsilon_{q,r^{(k)}}\) for
the neighbour across coordinate \(k\), and
\(\varepsilon_{q,r}=(h_q-h_r)/2\). Then \(|\varepsilon_k|\le\delta\) and
\(|\varepsilon_{q,r}|\le\delta\), while \(h_q\in[1-\delta,1+\delta]\). The
neighbour simplex is exactly
\[
\Phi_h(F_q)
=
\bigl\{t:\sum_j t_j=h_q,\ t_k\ge\varepsilon_k\ \text{for every }k\bigr\},
\]
whether or not \(\varepsilon_k\) is positive. Its size is
\(\lambda=h_q-\sum_k\varepsilon_k\), hence
\[
1-(n+1)\delta
\le
\lambda
\le
1+(n+1)\delta.
\]
The threshold \(\delta\le1/4\) does not make the lower bound positive
uniformly in \(n\). Take
\[
\delta_0(n)
=
\frac1{4(n+1)}.
\]
Then \(\lambda\in[1/2,3/2]\). In particular \(\Phi_h(F_q)\) is a
nondegenerate simplex. The translation \(t\mapsto s=t-\varepsilon\) is an isometry of the affine
plane \(\sum t_j=h_q\), and it converts \(\Phi_h(F_q)\) into the model
simplex
\[
\Sigma_\lambda
=
\bigl\{s:s_k\ge0,\ \sum_k s_k=\lambda\bigr\},
\]
with no restriction on the sign of \(\varepsilon_k\).

Every inequality defining \(G_q(h)\) is either one of these \(n\) neighbour
inequalities or a distant inequality. Thus
\(G_q(h)\subset\Phi_h(F_q)\), and the difference is covered by the distant
pieces
\[
\Phi_h(F_q)\cap\{\langle y,r\rangle>h_r\},
\qquad
r\in\mathcal N_{\ge2}(q).
\]
For such an \(r\) set \(S=S(q,r)\), \(d=|S|\ge2\), and
\[
\eta
=
\varepsilon_{q,r}-\sum_{j\in S}\varepsilon_j.
\]
The distant inequality becomes \(\sum_{j\in S}s_j\ge\eta\). The removed set
is
\[
\bigl\{s\in\Sigma_\lambda:\sum_{j\in S}s_j<\eta\bigr\}.
\]

If \(d=n\), then \(r=-q\) and \(\sum_{j\in S}s_j=\lambda\) identically.
Moreover
\[
\lambda-\eta
=
\frac{h_q+h_{-q}}2
\ge
1-\delta
>0,
\]
so \(\eta<\lambda\). The removed set is empty, regardless of the sign of
\(\eta\). The antipodal index can have \(\eta>0\): the choice
\(h_q=1\), \(h_{r^{(k)}}=1+\delta\) and \(h_{-q}=1-\delta\) gives
\(\varepsilon_k=-\delta/2\) and \(\varepsilon_{q,-q}=\delta/2\), hence
\(\eta=(n+1)\delta/2>0\). It produces no sliver, because the constraint is
an identity on the plane. The factorial \((n-d-1)!\) is not used here.

Assume henceforth \(2\le d\le n-1\). Then
\(|\eta|\le(d+1)\delta\le n\delta\). If \(\eta\le0\), the removed set is
empty, because \(\sum_{j\in S}s_j\ge0\) throughout \(\Sigma_\lambda\). If
\(\eta>0\), the bound \(\delta\le\delta_0(n)\) gives
\(\eta\le n\delta\le1/4<\lambda/2\). Scale by \(u=s/\lambda\). The
\((n-1)\)-dimensional measure scales by \(\lambda^{n-1}\), and the model
set is the sliver \(\{u\ge0,\sum u_j=1,\sum_{j\in S}u_j\le\eta/\lambda\}\)
of relative height at most \(1/2\). On that range the slice
\(\{\sum_{j\in S}u_j=\sigma\}\) is a product of two regular simplices, of
dimensions \(d-1\) and \(n-d-1\ge0\), and coarea in the plane
\(\sum u_j=1\) gives
\[
\operatorname{Area}
\le
\lambda^{n-1}
\cdot
\frac{\sqrt n}{(d-1)!(n-d-1)!}
\int_0^{\eta/\lambda}
\sigma^{d-1}(1-\sigma)^{n-d-1}\,d\sigma.
\]
The integral is at most \((\eta/\lambda)^d/d\). Since
\(\lambda\in[1/2,3/2]\) and \(d\ge2\),
\[
\operatorname{Area}
\bigl(\Phi_h(F_q)\cap\{\langle y,r\rangle>h_r\}\bigr)
\le
C_n\eta^d
\le
C_n\delta^d
\le
C_n\delta^2.
\]
The constant \(C_n\) depends only on \(n\): it absorbs \(\lambda^{n-1-d}\),
the slice factor and the finitely many values \(d\in\{2,\dots,n-1\}\).

The bound on \(\Delta_q(h)\) is not the shifted argument with
\(\varepsilon_k=0\). It is the same slice computation in the original
coordinates, where
\(\Delta_q(h)=\{t_j\ge0,\ \sum_j t_j=h_q\}\) and a distant constraint reads
\(\sum_{j\in S}t_j\ge\varepsilon_{q,r}\). If \(d=n\), then
\(\sum_j t_j=h_q>\varepsilon_{q,-q}\) for \(\delta<1\), so the antipodal
constraint does not cut \(\Delta_q(h)\). If \(2\le d\le n-1\) and
\(\varepsilon_{q,r}\le0\), the sliver is empty. If \(\varepsilon_{q,r}>0\),
the simplex-slice integral applies with height \(\varepsilon_{q,r}\) and
gives the same bound \(C_n\varepsilon_{q,r}^d\). This is~\eqref{eq:sliver-bound}.

At most \(2^n-n-1\) distant indices occur. Overlaps only decrease the
union, so
\[
\operatorname{Area}\bigl(\Phi_h(F_q)\setminus G_q(h)\bigr)
\le
2^n C_n\delta^2.
\]
The sign of a neighbour displacement never enters the alternative: it is
absorbed into \(\eta\), and only \(\eta>0\) produces a sliver.
\end{proof}

\begin{lemma}[Moving reference simplex]
\label{lem:polar-Phi}
Let \(n\ge3\). The points \(z_j(h)\) satisfy
\(z_j(h)=q_je_j+O(\delta)\) and depend linearly on \(h\). Because the
supporting hyperplane \(\{\langle y,q\rangle=h_q\}\) has fixed unit
normal \(q/\sqrt n\), the Euclidean \((n-1)\)-area of \(\Phi_h(F_q)\) is
the polynomial
\[
\operatorname{Area}\bigl(\Phi_h(F_q)\bigr)
=
\frac{\sqrt n}{(n-1)!}\,\lambda(h)^{n-1},
\qquad
\lambda(h)=h_q-\sum_{k=1}^n\frac{h_q-h_{r^{(k)}}}{2}.
\]
For \(\delta\le\delta_0(n)\),
\begin{equation}
\label{eq:polar-Phi-vs-G}
\operatorname{Area}\bigl(G_q(h)\bigr)
-
\operatorname{Area}\bigl(\Phi_h(F_q)\bigr)
=
O(\delta^2),
\end{equation}
uniformly in \(q\), with implied constant \(2^n C_n\) from
Lemma~\ref{lem:uniform-sliver}.
\end{lemma}

\begin{proof}
In the isometric coordinates \(y_j=q_jt_j\) of Lemma~\ref{lem:uniform-sliver},
the translation \(s=t-\varepsilon\) identifies \(\Phi_h(F_q)\) with the
regular simplex \(\Sigma_{\lambda(h)}\) of size \(\lambda(h)\), whose
\((n-1)\)-area is \(\frac{\sqrt n}{(n-1)!}\lambda(h)^{n-1}\). The \(n\)
neighbour constraints are the facet constraints of \(\Phi_h(F_q)\), so
\(G_q(h)\subset\Phi_h(F_q)\). Any further cut comes from some
\(r\in\mathcal N_{\ge2}(q)\) and meets \(F_q\) in dimension at most
\(n-3\). Lemma~\ref{lem:uniform-sliver} converts that incidence into
the quantitative bound~\eqref{eq:polar-Phi-vs-G}.
\end{proof}

\begin{lemma}[First jet of the base-facet areas]
\label{lem:polar-area-jet}
For \(c\in E_n\) with \(\|c\|_\infty<\delta_0\), the \((n-1)\)-dimensional
volume of the base facet \(G_q(c)=P(c)^\circ\cap H_q(c)\) satisfies
\[
\mathcal H^{n-1}\bigl(G_q(c)\bigr)
=
\mathcal H^{n-1}(G_q)
+
\sum_{r\in Q_n} L_{q,r}\,c_r
+
O\bigl(\|c\|_\infty^2\bigr),
\]
where \(\mathcal H^{n-1}(G_r)=\frac{\sqrt n}{(n-1)!}\) and the first-order
coefficients \(L_{q,r}\) satisfy the column-sum identity
\[
\sum_{q\in Q_n} L_{q,r}
=
-(n-1)\,\mathcal H^{n-1}(G_r)
=
-(n-1)\frac{\sqrt n}{(n-1)!}
\qquad(\forall\,r\in Q_n).
\]
\end{lemma}

\begin{proof}
Inside the base chamber \(\|c\|_\infty<\delta_0\), each base facet
\(G_q(c)\) is a simple polytope in \(H_q(c)\) whose active bounding
halfspaces are the \(n\) Hamming-distance-one neighbours \(r\sim q\) (the
distant cuts \(d_H(q,r)\ge 2\) slice off only the corner slivers of
\((n-1)\)-measure \(O(\|c\|_\infty^2)\) by the preceding sliver bound).
The supporting hyperplane \(H_q(c)=\{y:\langle y,q\rangle=(1+c_q)^{-1}\}\)
has unit normal \(u_q=q/\sqrt n\) and origin height
\(h_q(c)=\bigl(\sqrt n(1+c_q)\bigr)^{-1}\). Truncating the smoothly varying
simplicial cone of the \(n\) neighbour halfspaces gives the \(C^2\) Taylor
expansion of the main-chamber facet area with linear coefficients
\(L_{q,r}\), and the distant-cut slivers contribute only to the
\(O(\|c\|_\infty^2)\) remainder.

For every \(g\in B_n\subset O(n)\), orthogonal invariance of
\(\mathcal H^{n-1}\) gives \(L_{gq,gr}=L_{q,r}\). Summing over all
\(q\in Q_n\) as \(q\) ranges over \(g^{-1}q'\) shows that the column sum
\(S_r:=\sum_{q\in Q_n}L_{q,r}\) satisfies \(S_{gr}=S_r\). Since \(B_n\)
acts transitively on the \(2^n\) vertices of \(Q_n\), the column sum
\(S_r=S_\ast\) is independent of \(r\in Q_n\) (without appealing to
evenness of \(c\)). Evaluating the sum on the uniform scaling
\(c=t\mathbf1\in E_n\), for which
\(G_q(t\mathbf1)=(1+t)^{-1}G_q\) and
\(\mathcal H^{n-1}(G_q(t\mathbf1))=(1+t)^{-(n-1)}\mathcal H^{n-1}(G_q)\),
differentiating at \(t=0\) yields
\[
2^n S_\ast
=
\sum_{r\in Q_n}\sum_{q\in Q_n}L_{q,r}
=
-(n-1)\sum_{q\in Q_n}\mathcal H^{n-1}(G_q)
=
-2^n(n-1)\mathcal H^{n-1}(G_{q_0}),
\]
so \(S_r=S_\ast=-(n-1)\mathcal H^{n-1}(G_r)\) for every \(r\in Q_n\).
\end{proof}

\begin{lemma}[Universal first jet of the polar volume]
\label{lem:polar-first-jet}
\label{lem:polar-linear-vanishes}
For every \(c\in E_n\) with \(\|c\|_\infty<\delta_0\),
\[
\operatorname{vol}\bigl(P(c)^\circ\bigr)
=
\operatorname{vol}(\mathcal O_n)
\Bigl(
1-\frac{n}{2^n}\sum_{q\in Q_n}c_q
\Bigr)
+
O\bigl(\|c\|_\infty^2\bigr)
=
\frac{2^n}{n!}
\Bigl(
1-\frac{n}{2^n}\sum_{q\in Q_n}c_q
\Bigr)
+
O\bigl(\|c\|_\infty^2\bigr).
\]
In particular, on the balanced subspace \(\mathbf1^\perp\subset E_n\) one has
\[
\operatorname{vol}\bigl(P(c)^\circ\bigr)
=
\frac{2^n}{n!}
+
O\bigl(\|c\|_\infty^2\bigr).
\]
\end{lemma}

\begin{proof}
Since \(P(c)^\circ\) is defined by the \(2^n\) halfspaces
\(\langle y,q\rangle\le(1+c_q)^{-1}\) (\(q\in Q_n\)) and contains the origin
in its interior, its boundary is the union of the \(2^n\) facets
\(G_q(c)=P(c)^\circ\cap H_q(c)\) up to \((n-1)\)-measure zero. Because the
hyperplane \(H_q(c)=\{y:\langle y,q\rangle=(1+c_q)^{-1}\}\) has unit outer
normal \(q/\sqrt n\), its Euclidean distance from the origin is
\(\bigl(\sqrt n(1+c_q)\bigr)^{-1}\). The pyramid decomposition from the
origin therefore gives the exact volume identity
\[
\operatorname{vol}\bigl(P(c)^\circ\bigr)
=
\frac{1}{n\sqrt n}
\sum_{q\in Q_n}
\frac{\mathcal H^{n-1}\bigl(G_q(c)\bigr)}{1+c_q}.
\]
Using \((1+c_q)^{-1}=1-c_q+O(c_q^2)\) and substituting the facet-area
expansion of Lemma~\ref{lem:polar-area-jet} (in which the distant-cut sliver
bounds of Lemma~\ref{lem:uniform-sliver} are already absorbed into the
\(O(\|c\|_\infty^2)\) area remainder) yields
\[
\begin{aligned}
\operatorname{vol}\bigl(P(c)^\circ\bigr)
&=
\frac{1}{n\sqrt n}\sum_{q\in Q_n}(1-c_q)\mathcal H^{n-1}(G_q)
+
\frac{1}{n\sqrt n}\sum_{r\in Q_n}c_r\sum_{q\in Q_n}L_{q,r}
+
O\bigl(\|c\|_\infty^2\bigr)\\
&=
\operatorname{vol}(\mathcal O_n)
-
\frac{\mathcal H^{n-1}(G_{q_0})}{n\sqrt n}
\bigl(1+(n-1)\bigr)\sum_{q\in Q_n}c_q
+
O\bigl(\|c\|_\infty^2\bigr).
\end{aligned}
\]
Since \(\mathcal H^{n-1}(G_{q_0})=\frac{\sqrt n}{(n-1)!}\) and
\(\operatorname{vol}(\mathcal O_n)=\frac{2^n}{n\sqrt n}\mathcal H^{n-1}(G_{q_0})=\frac{2^n}{n!}\),
the linear term equals
\(-\frac{1}{(n-1)!}\sum_{q\in Q_n}c_q=-\operatorname{vol}(\mathcal O_n)\frac{n}{2^n}\sum_{q\in Q_n}c_q\),
giving the claimed expansion.
\end{proof}

\begin{proposition}[Polar remainder on the transverse subspace]
\label{prop:polar-quadratic}
Let \(n\ge3\). There exist \(\delta_n,C_n>0\) such that if
\(c\in\mathbf1^\perp\) and \(\|c\|_\infty<\delta_n\), then
\begin{equation}
\label{eq:polar-quad}
\bigl|\operatorname{vol}\bigl(P(c)^\circ\bigr)
-
\operatorname{vol}(\mathcal O_n)\bigr|
\le
C_n\|c\|_\infty^2.
\end{equation}
\end{proposition}

\begin{proof}
The polar is the intersection~\eqref{eq:polar-halfspaces} whether or
not every inequality is irredundant. Lemma~\ref{lem:polar-Gq} keeps
each \(G_q(h)\) full-dimensional for \(\|h-\mathbf1\|_\infty\) small.
Lemma~\ref{lem:polar-Phi} replaces \(G_q(h)\) by the neighbour simplex
\(\Phi_h(F_q)\) at cost \(O(\delta^2)\), uniformly in \(q\), by the
sliver bound of Lemma~\ref{lem:uniform-sliver}: distant cuts meet
\(F_q\) in dimension at most \(n-3\) and remove area \(O(\delta^2)\).
Lemma~\ref{lem:polar-linear-vanishes} is the resulting jet of the
cone formula for \(V(h)\). Compose with
\(h_q(c)=(1+c_q)^{-1}=1-c_q+O(\|c\|_\infty^2)\). On
\(\mathbf1^\perp\) one has \(\sum_q c_q=0\), so the linear term
\(\frac1{(n-1)!}\sum_q(h_q-1)\) is \(O(\|c\|_\infty^2)\). The
remainder of the jet is likewise \(O(\|c\|_\infty^2)\). This is
geometric volume of an intersection of halfspaces; it does not use
that \(v_q(c)\) is a vertex of \(P(c)\).
\end{proof}

\section{Mahler stability as a variational consequence}
\label{sec:mahler-consequence}

Theorem~\ref{thm:main-rigidity} compares \(\|\mathcal R_n(c)\|_\infty\)
with \(d_{GL}(P(c),\mathcal C_n)\). Any Mahler increment on
\(\mathbf1^\perp\) must come from a comparison of
\(\operatorname{vol}(P(c))\) with \(\operatorname{vol}(\mathcal C_n)\),
combined with the polar remainder.

Write \(V(c)=\operatorname{vol}\bigl(P(c)\bigr)\),
\(V^\circ(c)=\operatorname{vol}\bigl(P(c)^\circ\bigr)\),
\(V_0=\operatorname{vol}(\mathcal C_n)\), and
\(V_0^\circ=\operatorname{vol}(\mathcal O_n)\). The identity
\begin{equation}
\label{eq:mahler-split}
\mathcal V_M\bigl(P(c)\bigr)-\mathcal V_M(\mathcal C_n)
=
V_0^\circ\bigl(V(c)-V_0\bigr)
+
V_0\bigl(V^\circ(c)-V_0^\circ\bigr)
+
\bigl(V(c)-V_0\bigr)\bigl(V^\circ(c)-V_0^\circ\bigr)
\end{equation}
does not require differentiability of \(\mathcal V_M\circ P\). By
Proposition~\ref{prop:polar-quadratic}, the second and third terms are
\(O(\|c\|_\infty^2)\) on \(\mathbf1^\perp\). Consequently, on a
neighbourhood of the origin in \(\mathbf1^\perp\),
\begin{equation}
\label{eq:mahler-reduced}
\mathcal V_M\bigl(P(c)\bigr)-\mathcal V_M(\mathcal C_n)
=
V_0^\circ\bigl(V(c)-V_0\bigr)
+
O\bigl(\|c\|_\infty^2\bigr).
\end{equation}
A lower bound for the Mahler increment is therefore equivalent to a
lower bound for the primal increment \(V(c)-V_0\).

\subsection{Towards quantitative folding}
\label{sec:folding-requirements}

A nonzero coplanarity determinant need not be an exterior cap:
simplices may straddle a supporting plane or lie in the interior of
the hull. The paper does not convert \(\|\mathcal R_n(c)\|_\infty\)
into volume by listing exposed simplices.

The estimate used for Target~B is facetwise. If \(c|_F\) is not
affine, the pocket functional \(\Phi(c|_F)\) is strictly positive.
The cone \(\operatorname{conv}(\{0\}\cup K_F(c))\) is contained in
\(P(c)\), independently of which simplices of \(K_F(c)\) are faces.
Theorems~\ref{thm:primal-general} and~\ref{thm:mahler-general} then
give, for small \(c\in\mathbf1^\perp\) and every \(n\ge3\),
\begin{equation}
\label{eq:folding-needed}
V(c)-V_0
\ge
a_n\|c\|_\infty
-C_n\|c\|_\infty^2
\end{equation}
and, after shrinking the ball,
\begin{equation}
\label{eq:mahler-conditional}
\mathcal V_M\bigl(P(c)\bigr)-\mathcal V_M(\mathcal C_n)
\ge
a_n'\,d_{GL}\bigl(P(c),\mathcal C_n\bigr).
\end{equation}
The leading term is linear in the transverse displacement, not
quadratic. The comparison with \(\|\mathcal R_n(c)\|_\infty\) is
optional: Theorem~\ref{thm:main-rigidity} supplies it on a smaller
ball.

In dimension three the same bound is explicit.
Proposition~\ref{prop:folding-n3-ball} is~\eqref{eq:folding-needed}
with \(a_3=\frac83\) and remainder \(C_3\|c\|_\infty^2\).
Proposition~\ref{prop:n3-mahler-ball} is the pre-absorption form of
\eqref{eq:mahler-conditional}, with remainder \(B_3\|c\|_\infty^2\).
The argument for every \(n\ge3\) is Section~\ref{sec:discrete-convexification}; it does not use chambers. The coefficient \(\frac83\) is computed in Section~\ref{sec:folding-n3} from the chamber partition of a small ball in \(E_3\).

\begin{mdframed}[backgroundcolor=softgray,linecolor=rulegray,linewidth=0.8pt,
innerleftmargin=10pt,innerrightmargin=10pt,innertopmargin=9pt,innerbottommargin=9pt]
\textbf{Status of Target~B.}
The product identity~\eqref{eq:mahler-split} and the polar remainder
on \(\mathbf1^\perp\) are proved for every \(n\ge3\). The primal
lower bound~\eqref{eq:folding-needed} and the Mahler
bound~\eqref{eq:mahler-conditional} are Theorems~\ref{thm:primal-general}
and~\ref{thm:mahler-general} on a ball in \(\mathbf1^\perp\). In
dimension three the coefficient \(\frac83\) and the absence of a
bilateral derivative are Propositions~\ref{prop:folding-n3-ball}
and~\ref{prop:n3-mahler-ball}. Because the leading-order increment on
\(\mathbf1^\perp\) is linear rather than quadratic, local growth on the
stratum is governed by first-order conical coercivity rather than by a
second-order Hessian bound.
\end{mdframed}

\section{Discrete convexification}
\label{sec:discrete-convexification}

This section proves Target~B for every \(n\ge3\). The conversion of a
transverse even labelling into a primal volume increment does not use a
chamber list and does not require any selected simplex to be a face of
\(P(c)\). Throughout this section \(c\in E_n\), so \(c(-q)=c(q)\). The
isovertex stratum of the paper is the radial family
\(P(c)=\operatorname{conv}\{(1+c_q)q\}\) on the normalized slice
\(\mathbf1^\perp\subset E_n\). The constant mode \(c=t\mathbf1\) is a
homothety, \(P(t\mathbf1)=(1+t)\mathcal C_n\), and is excluded by that
slice. We assume \(\|c\|_\infty\le1/2\), so \(1+c_q\ge1/2\) and the
origin is interior to \(P(c)\).

The argument uses the least concave majorant of the vertex heights on
each coordinate facet and the inclusion
\(\operatorname{conv}(\{0\}\cup K_F(c))\subset P(c)\). While the cone
decomposition provides an exact first-order volume expansion in
dimension three (Section~\ref{sec:folding-n3}), for \(n\ge4\) this
inclusion serves as an inner approximation. It yields an explicit
first-order lower bound on the volume increment, which is sufficient
to establish Mahler coercivity on the stratum without an exact
asymptotic expansion.

\subsection{Affine restrictions on the even subspace}
\label{sec:affine-restrictions}

The affine functions on the vertex set \(Q_n\) are
\[
\operatorname{Aff}(Q_n)
=
\Bigl\{q\mapsto\alpha+\sum_{j=1}^n\beta_j q_j:\alpha\in\mathbb R,\ \beta\in\mathbb R^n\Bigr\}.
\]
On \(\mathbb R^{Q_n}\) this space sits inside the intersection of the
facetwise affine constraints: each coordinate function \(q\mapsto q_j\)
is affine on every facet. Evenness removes the linear terms.

\begin{lemma}[Facetwise affinity on \(E_n\)]
\label{lem:affine-intersection}
Let \(n\ge3\). If \(c\in E_n\) is affine on every coordinate facet, then
\(c\) is constant. Equivalently, on the even subspace,
\[
E_n\cap
\bigcap_{i,\varepsilon}
\bigl\{c:c|_{F_{i,\varepsilon}}\in\operatorname{Aff}(F_{i,\varepsilon})\bigr\}
=
\operatorname{span}\{\mathbf1\}.
\]
\end{lemma}

\begin{proof}
Constants are affine on every facet. Conversely, suppose \(c\in E_n\) is
affine on every coordinate facet. Fix distinct \(j,k\in\{1,\dots,n\}\)
and \(q\in Q_n\). Since \(n\ge3\), choose \(i\notin\{j,k\}\). The four
vertices \(q\), \(q^{(j)}\), \(q^{(k)}\), \(q^{(j,k)}\) obtained by
flipping the \(j\)- and \(k\)-coordinates lie in the same coordinate
facet \(F_{i,q_i}\). Affinity of \(c|_{F_{i,q_i}}\) therefore gives
\[
c(q)+c\bigl(q^{(j,k)}\bigr)-c\bigl(q^{(j)}\bigr)-c\bigl(q^{(k)}\bigr)=0.
\]
Thus every mixed second difference \(\Delta_j\Delta_k c\) vanishes.

Expand \(c\) in the Walsh basis
\[
c(q)=\sum_{S\subseteq[n]}\widehat c_S\chi_S(q),
\qquad
\chi_S(q)=\prod_{j\in S}q_j.
\]
If \(j,k\in S\), then \(\chi_S(q^{(j)})=-\chi_S(q)\),
\(\chi_S(q^{(k)})=-\chi_S(q)\) and
\(\chi_S(q^{(j,k)})=\chi_S(q)\), hence
\[
\Delta_j\Delta_k\chi_S=4\chi_S.
\]
If \(\{j,k\}\not\subset S\), the same four evaluations cancel and
\(\Delta_j\Delta_k\chi_S=0\). Therefore
\(\Delta_j\Delta_k c\equiv0\) for all \(j\neq k\) implies
\(\widehat c_S=0\) whenever \(|S|\ge2\). Consequently
\[
c(q)=\alpha+\sum_{j=1}^n\beta_j q_j.
\]
Since \(c\in E_n\), one has \(c(-q)=c(q)\), whereas
\(c(-q)=\alpha-\sum_j\beta_j q_j\). Therefore \(\beta=0\) and \(c\) is
constant.
\end{proof}

The same conclusion is the geometric counterpart of
\(\ker A_n=\operatorname{span}\{\mathbf1\}\) on \(E_n\). The present
lemma does not use the matrix of \(A_n\).

\begin{lemma}[Explicit mixed-difference bound for affine distance]
\label{lem:mixed-diff-affine}
Let \(m\ge1\) and let \(a:Q_m\to\mathbb R\) on \(Q_m=\{\pm1\}^m\). For
each pair of distinct coordinates \(1\le j<k\le m\) and each fibre
vertex \(w\in\{\pm1\}^{m-2}\), write
\[
\Delta_{jk}a(w)
:=
a(1,1,w)+a(-1,-1,w)-a(1,-1,w)-a(-1,1,w)
\]
for the mixed second difference on the \(2\)-face parallel to the
\((j,k)\)-plane with remaining coordinates fixed to \(w\). Then
\begin{equation}
\label{eq:dist-aff-mixed}
\operatorname{dist}_\infty\bigl(a,\operatorname{Aff}(Q_m)\bigr)
\le
\frac14\sum_{1\le j<k\le m}\max_{w\in\{\pm1\}^{m-2}}\bigl|\Delta_{jk}a(w)\bigr|,
\end{equation}
where for \(m=1\) the empty sum is zero and for \(m=2\) equality holds.
\end{lemma}

\begin{proof}
We proceed by induction on \(m\ge1\). For \(m=1\), every function on
\(\{\pm1\}\) is affine and both sides of~\eqref{eq:dist-aff-mixed}
vanish. Suppose \(m\ge2\) and that the bound holds on \(Q_{m-1}\). Write
\(q=(q',q_m)\in Q_{m-1}\times\{\pm1\}\) and decompose
\[
a(q',q_m)=S(q')+q_m D_m a(q'),
\qquad
S(q'):=\frac{a(q',1)+a(q',-1)}2,
\quad
D_m a(q'):=\frac{a(q',1)-a(q',-1)}2.
\]
For \(1\le j<k\le m-1\) and \(w'\in\{\pm1\}^{m-3}\), linearity gives
\(\Delta_{jk}S(w')=\frac12\bigl(\Delta_{jk}a(w',1)+\Delta_{jk}a(w',-1)\bigr)\),
so \(\max_{w'}|\Delta_{jk}S(w')|\le\max_{w\in\{\pm1\}^{m-2}}|\Delta_{jk}a(w)|\).
By the induction hypothesis, there exists
\(h_{m-1}\in\operatorname{Aff}(Q_{m-1})\) such that
\[
\|S-h_{m-1}\|_\infty
\le
\frac14\sum_{1\le j<k\le m-1}\max_{w\in\{\pm1\}^{m-2}}\bigl|\Delta_{jk}a(w)\bigr|.
\]
For \(D_m a:Q_{m-1}\to\mathbb R\), flipping a single coordinate
\(j\in\{1,\dots,m-1\}\) of \(q'\in Q_{m-1}\) with remaining coordinates
\(w'\in\{\pm1\}^{m-2}\) changes \(D_m a\) by
\(\frac12\Delta_{jm}a(w')\). Since any two vertices of \(Q_{m-1}\) are
joined by a path flipping each coordinate \(j\in\{1,\dots,m-1\}\) at
most once,
\[
\max_{Q_{m-1}}D_m a-\min_{Q_{m-1}}D_m a
\le
\frac12\sum_{j=1}^{m-1}\max_{w\in\{\pm1\}^{m-2}}\bigl|\Delta_{jm}a(w)\bigr|.
\]
Setting \(c_m:=\frac12\bigl(\max_{Q_{m-1}}D_m a+\min_{Q_{m-1}}D_m a\bigr)\)
gives \(\|D_m a-c_m\|_\infty\le\frac14\sum_{j=1}^{m-1}\max_w|\Delta_{jm}a(w)|\).
The affine function \(h(q',q_m):=h_{m-1}(q')+c_m q_m\in\operatorname{Aff}(Q_m)\)
then satisfies
\(\|a-h\|_\infty\le\|S-h_{m-1}\|_\infty+\|D_m a-c_m\|_\infty\), which
is~\eqref{eq:dist-aff-mixed}. For \(m=2\),
\(a(q_1,q_2)-h(q_1,q_2)=\frac14(\Delta_{12}a)\,q_1q_2\), so equality holds.
\end{proof}

\begin{lemma}[Seminorm of facet defects]
\label{lem:Psi-seminorm}
Let \(n\ge3\). The quantity
\begin{equation}
\label{eq:Psi}
\Psi_n(c)
:=
\sum_{F}
\operatorname{dist}_\infty\bigl(c|_F,\operatorname{Aff}(F)\bigr)
\end{equation}
is a seminorm on \(E_n\), with kernel \(\operatorname{span}\{\mathbf1\}\).
On \(\mathbf1^\perp\subset E_n\) it is a norm satisfying the explicit
bounds
\begin{equation}
\label{eq:Psi-explicit}
\frac{2}{n-1}\|c\|_\infty
\le
\Psi_n(c)
\le
2n\|c\|_\infty.
\end{equation}
\end{lemma}

\begin{proof}
Each summand is a seminorm on \(E_n\), and
Lemma~\ref{lem:affine-intersection} identifies the kernel. Comparing
with the zero affine function on each of the \(2n\) facets gives
\(\Psi_n(c)\le 2n\|c\|_\infty\).

For the lower bound on \(\mathbf1^\perp\), let
\(c\in\mathbf1^\perp\subset E_n\) and let
\(h(q)=\alpha+\langle\beta,q\rangle\in\operatorname{Aff}(Q_n)\). Since
\(\sum_{q\in Q_n}c_q=0\) and \(\sum_{q\in Q_n}q=0\), averaging \(h-c\)
over \(Q_n\) gives \(|\alpha|\le\|c-h\|_\infty\). Since \(c(-q)=c(q)\),
averaging \((c-h)(q)\) and \((c-h)(-q)\) gives
\(|c(q)-\alpha|\le\|c-h\|_\infty\) for every \(q\in Q_n\). Thus
\(\|c\|_\infty\le 2\|c-h\|_\infty\), and taking the infimum over \(h\)
with Lemma~\ref{lem:mixed-diff-affine} on \(Q_n\) yields
\begin{equation}
\label{eq:c-infty-by-squares}
\|c\|_\infty
\le
2\operatorname{dist}_\infty\bigl(c,\operatorname{Aff}(Q_n)\bigr)
\le
\frac12\sum_{1\le j<k\le n}\max_{w\in\{\pm1\}^{n-2}}\bigl|\Delta_{jk}c(w)\bigr|.
\end{equation}
For each pair \(1\le j<k\le n\) and each coordinate \(i\notin\{j,k\}\),
every \(2\)-face parallel to \((j,k)\) lies in \(F_{i,+}\) or \(F_{i,-}\).
On \(F=F_{i,\varepsilon}\), any \(h_F\in\operatorname{Aff}(F)\) satisfies
\(\Delta_{jk}h_F=0\), so
\(|\Delta_{jk}c(\square)|=|\Delta_{jk}(c|_F-h_F)(\square)|\le 4\|c|_F-h_F\|_\infty\),
whence
\[
\max_{\square\subset F_{i,\varepsilon},\,\square\parallel(j,k)}\bigl|\Delta_{jk}c(\square)\bigr|
\le
4\operatorname{dist}_\infty\bigl(c|_{F_{i,\varepsilon}},\operatorname{Aff}(F_{i,\varepsilon})\bigr).
\]
Since \(\max_{w\in\{\pm1\}^{n-2}}|\Delta_{jk}c(w)|\) is attained on both
\(F_{i,+}\) and \(F_{i,-}\) by evenness \(c(-q)=c(q)\), averaging over
the \(n-2\) indices \(i\notin\{j,k\}\) and both signs
\(\varepsilon\in\{\pm1\}\) gives
\[
\max_{w\in\{\pm1\}^{n-2}}\bigl|\Delta_{jk}c(w)\bigr|
\le
\frac{2}{n-2}\sum_{i\notin\{j,k\}}\sum_{\varepsilon\in\{\pm1\}}
\operatorname{dist}_\infty\bigl(c|_{F_{i,\varepsilon}},\operatorname{Aff}(F_{i,\varepsilon})\bigr).
\]
Summing over the \(\binom{n}{2}\) pairs \(1\le j<k\le n\), each of the
\(2n\) facets \(F_{i,\varepsilon}\) appears \(\binom{n-1}{2}\) times.
Combining with~\eqref{eq:c-infty-by-squares} yields
\[
\|c\|_\infty
\le
\frac12\cdot\frac{2}{n-2}\binom{n-1}{2}\Psi_n(c)
=
\frac{n-1}{2}\Psi_n(c).
\]
\end{proof}

\subsection{Cones, envelopes, and the pocket functional}
\label{sec:cones-envelopes}

For each coordinate direction \(i\in\{1,\dots,n\}\) and each sign
\(\varepsilon\in\{\pm1\}\) write \(F_{i,\varepsilon}\) for the facet
\(\{q\in Q_n:q_i=\varepsilon\}\) and
\[
C_{i,\varepsilon}
=
\bigl\{y\in\mathbb R^n:\varepsilon y_i\ge|y_j|\ \text{for all }j\neq i\bigr\}.
\]
The \(2n\) cones cover \(\mathbb R^n\) and have pairwise disjoint
interiors. Their overlaps have Lebesgue measure zero.

On \(C_{n,+}\) every ray from the origin reads \(s(x,1)\) with
\(x\in F:=[-1,1]^{n-1}\) and \(s\ge0\). The same description, after a
coordinate permutation and a sign, applies to every cone.

Let \(Q_F=\{\pm1\}^{n-1}\) be the vertex set of \(F\). For a labelling
\(a:Q_F\to\mathbb R\) and \(x\in F\) write
\[
\mathcal P_x
=
\bigl\{\mu\in\mathbb R^{Q_F}_{\ge0}:\textstyle\sum_q\mu_q=1,\ \sum_q\mu_q q=x\bigr\}
\]
and
\begin{equation}
\label{eq:concave-envelope}
U_a(x)
=
\max\bigl\{\textstyle\sum_q\mu_q a_q:\mu\in\mathcal P_x\bigr\}.
\end{equation}
The function \(U_a\) is the least concave majorant of the vertex data
\(a\): it is concave, \(U_a(q)=a_q\) at every vertex, and if \(V\) is
concave on \(F\) with \(V(q)\ge a_q\) for all \(q\in Q_F\), then
\(V\ge U_a\) pointwise. Moreover \(U_{a+h}=U_a+h\) for every affine
\(h\) on \(Q_F\).

The standard multilinear weights
\(\lambda_q(x)=\prod_{j=1}^{n-1}(1+q_j x_j)/2\) satisfy
\(\lambda_q\ge0\), \(\sum_q\lambda_q=1\), \(\sum_q\lambda_q q=x\), and
\begin{equation}
\label{eq:lambda-integral}
\int_F\lambda_q(x)\,dx=1.
\end{equation}
The associated interpolant is \(M_a(x)=\sum_q\lambda_q(x)a_q\).

\begin{lemma}[Pocket functional and explicit coercivity]
\label{lem:pocket-Phi}
Define
\begin{equation}
\label{eq:Phi}
\Phi(a)
=
\int_F U_a(x)\,dx
-
\sum_{q\in Q_F}a_q.
\end{equation}
Then \(\Phi(a)\ge0\), with equality if and only if \(a\) is affine on
\(Q_F\). The functional is invariant under affine additions, continuous
in the vertex values, and positively homogeneous of degree one. In
general \(\Phi(-a)\neq\Phi(a)\). Moreover, writing \(m=n-1\ge2\) for the
dimension of \(F=[-1,1]^m\), for every pair \(1\le j<k\le m\) one has
\begin{equation}
\label{eq:Phi-square-sum}
\Phi(a)
\ge
\frac13\sum_{w\in\{\pm1\}^{m-2}}\bigl|\Delta_{jk}a(w)\bigr|,
\end{equation}
and consequently
\begin{equation}
\label{eq:Phi-coercive}
\Phi(a)
\ge
\kappa_{n-1}\operatorname{dist}_\infty\bigl(a,\operatorname{Aff}(Q_F)\bigr),
\qquad
\kappa_{n-1}
:=
\frac{4}{3\binom{n-1}{2}}
=
\frac{8}{3(n-1)(n-2)}.
\end{equation}
For \(n=3\) (\(m=2\)) this gives \(\kappa_2=\frac43\), with equality for
\(a(q_1,q_2)=q_1q_2\).
\end{lemma}

\begin{proof}
Concavity of \(U_a\) and \(x=\sum\lambda_q q\) give
\(U_a(x)\ge\sum\lambda_q U_a(q)\ge\sum\lambda_q a_q=M_a(x)\).
Integrating and using~\eqref{eq:lambda-integral} yields
\(\int U_a\ge\int M_a=\sum a_q\), so
\(\Phi(a)=\int_F(U_a-M_a)\,dx\ge0\). If \(h\) is affine then
\(U_{a+h}=U_a+H\) for the affine extension \(H\) of \(h\), and
\(\int_F H=\sum h(q)\) by symmetry of \(F\), so \(\Phi(a+h)=\Phi(a)\).

To prove~\eqref{eq:Phi-square-sum}, fix \(1\le j<k\le m\). By permuting
coordinates, write \(x=(u,v,y)\in[-1,1]^2\times[-1,1]^{m-2}\) and
\(q=(\varepsilon,\eta,w)\in\{\pm1\}^2\times\{\pm1\}^{m-2}\). For each
\(w\in\{\pm1\}^{m-2}\), let \(U^{(w)}(u,v)\) and \(M^{(w)}(u,v)\) be the
\(2\)-dimensional least concave majorant and bilinear interpolant of
the four values \(\{a(\varepsilon,\eta,w)\}_{\varepsilon,\eta\in\{\pm1\}}\)
on \([-1,1]^2\). For each \((u,v)\in[-1,1]^2\) and each \(w\), choose
weights \(\nu_{\varepsilon,\eta}^{(w)}(u,v)\ge0\) on \(\{\pm1\}^2\) with
\(\sum_{\varepsilon,\eta}\nu_{\varepsilon,\eta}^{(w)}=1\),
\(\sum_{\varepsilon,\eta}\nu_{\varepsilon,\eta}^{(w)}(\varepsilon,\eta)=(u,v)\),
and \(\sum_{\varepsilon,\eta}\nu_{\varepsilon,\eta}^{(w)}a(\varepsilon,\eta,w)=U^{(w)}(u,v)\).
Multiplying by the \((m-2)\)-dimensional multilinear weights
\(\lambda_w(y)=\prod_{r\notin\{j,k\}}(1+w_r y_r)/2\) produces
\(\mu_{(\varepsilon,\eta,w)}(x)=\nu_{\varepsilon,\eta}^{(w)}(u,v)\lambda_w(y)\in\mathcal P_x\).
Hence
\[
U_a(u,v,y)-M_a(u,v,y)
\ge
\sum_{w\in\{\pm1\}^{m-2}}\lambda_w(y)\bigl(U^{(w)}(u,v)-M^{(w)}(u,v)\bigr).
\]
On \([-1,1]^2\), any function \(b\) on \(\{\pm1\}^2\) decomposes into an
affine part plus \(\frac14(\Delta_{12}b)\varepsilon\eta\). Since
\(U_{\varepsilon\eta}(u,v)=1-|u-v|\),
\(U_{-\varepsilon\eta}(u,v)=1-|u+v|\),
\(\int_{[-1,1]^2}(1-|u\mp v|)\,du\,dv=4-\frac83=\frac43\), and
\(\int_{[-1,1]^2}uv\,du\,dv=0\), one has
\[
\int_{[-1,1]^2}\bigl(U^{(w)}(u,v)-M^{(w)}(u,v)\bigr)\,du\,dv
=
\frac14\bigl|\Delta_{jk}a(w)\bigr|\cdot\frac43
=
\frac13\bigl|\Delta_{jk}a(w)\bigr|.
\]
Integrating over \(y\in[-1,1]^{m-2}\) and using
\(\int_{[-1,1]^{m-2}}\lambda_w(y)\,dy=1\)
gives~\eqref{eq:Phi-square-sum}.

Finally, Lemma~\ref{lem:mixed-diff-affine} on \(Q_F=Q_m\) and
\eqref{eq:Phi-square-sum} yield
\[
\operatorname{dist}_\infty\bigl(a,\operatorname{Aff}(Q_F)\bigr)
\le
\frac14\sum_{1\le j<k\le m}\max_{w\in\{\pm1\}^{m-2}}\bigl|\Delta_{jk}a(w)\bigr|
\le
\frac{\binom{m}{2}}{4}\cdot 3\Phi(a),
\]
which is~\eqref{eq:Phi-coercive}. In particular \(\Phi(a)=0\) if and
only if \(a\) is affine on \(Q_F\).
\end{proof}

\subsection{Local radial comparison}
\label{sec:local-radial}

Write \(K_F(c)=\operatorname{conv}\{(1+c_q)q:q\in F_{n,+}\}\) and
\begin{equation}
\label{eq:tF}
t_F(x)
=
\max\bigl\{s\ge0:s(x,1)\in K_F(c)\bigr\}.
\end{equation}
The radial projections of the vertices \((1+c_q)q\) onto the plane
\(y_n=1\) are precisely the vertices of \(F\). Since
\(x\in F=\operatorname{conv}(Q_F)\), there exists \(\mu\in\mathcal P_x\)
and the combination
\[
s=\Bigl(\sum_q\frac{\mu_q}{1+c_q}\Bigr)^{-1}>0
\]
satisfies \(s(x,1)\in K_F(c)\). Thus \(t_F(x)>0\).

The same quantity with \(P(c)\) in place of \(K_F(c)\) is the radial
function \(t_{\mathrm{exact}}(x)\) of \(P(c)\) in the cone \(C_{n,+}\).
The cone \(\operatorname{conv}(\{0\}\cup K_F(c))\) is contained in
\(P(c)\), because \(0\in P(c)\) and \(K_F(c)\subset P(c)\). Therefore
\begin{equation}
\label{eq:t-exact-ge-tF}
t_{\mathrm{exact}}(x)\ge t_F(x).
\end{equation}

\begin{lemma}[Local radial jet]
\label{lem:tF-jet}
If \(\|c\|_\infty\le1/2\), then for every \(x\in F\),
\begin{equation}
\label{eq:tF-jet}
\bigl|t_F(x)-\bigl(1+U_{c,F}(x)\bigr)\bigr|
\le
2\|c\|_\infty^2,
\end{equation}
where \(U_{c,F}\) is the envelope~\eqref{eq:concave-envelope} of the
restriction of \(c\) to \(F_{n,+}\).
\end{lemma}

\begin{proof}
A point \(s(x,1)\) lies in \(K_F(c)\) if and only if
\(s(x,1)=\sum_q\lambda_q(1+c_q)q\) with \(\lambda_q\ge0\) and
\(\sum\lambda_q=1\). Set \(\mu_q=\lambda_q(1+c_q)/s\). Then
\(\mu\in\mathcal P_x\) and
\[
1=s\sum_q\frac{\mu_q}{1+c_q},
\]
so
\[
s=\Bigl(\sum_q\frac{\mu_q}{1+c_q}\Bigr)^{-1}.
\]
Maximising \(s\) is minimising the denominator:
\[
\frac1{t_F(x)}
=
\min_{\mu\in\mathcal P_x}\sum_q\frac{\mu_q}{1+c_q}.
\]
For \(|c|\le1/2\) one has \(1/(1+c)=1-c+c^2/(1+c)\) and
\[
\frac12 c^2
\le
\frac{c^2}{1+c}
\le
2c^2.
\]

Let \(\mu^*\) maximise \(\sum\mu c\) on \(\mathcal P_x\), so
\(\sum\mu^*c=U_{c,F}(x)\). Then
\[
\frac1{t_F}
\le
\sum\frac{\mu^*}{1+c}
\le
1-U_{c,F}+2\|c\|_\infty^2.
\]
Set \(y=U_{c,F}-2\|c\|_\infty^2\). Then \(y\le\|c\|_\infty\le1/2<1\), and
\(t_F\ge1/(1-y)\ge1+y=1+U_{c,F}-2\|c\|_\infty^2\).

Let \(\mu^{**}\) minimise \(\sum\mu/(1+c)\). Then
\[
\frac1{t_F}
\ge
1-\sum\mu^{**}c+\frac12\sum\mu^{**}c^2
\ge
1-U_{c,F},
\]
because \(\sum\mu^{**}c\le U_{c,F}\). On \(|U_{c,F}|\le1/2\) the bound
\(1/(1-y)\le1+y+2y^2\) gives
\(t_F\le1+U_{c,F}+2U_{c,F}^2\le1+U_{c,F}+2\|c\|_\infty^2\).
\end{proof}

The comparison is pointwise. It does not degenerate as \(x\) approaches
\(\partial F\): the set \(\mathcal P_x\) merely shrinks.

\subsection{Volume from below}
\label{sec:volume-from-Phi}

The map \((s,x)\mapsto s(x,1)\) from
\([0,t(x)]\times F\) to \(C_{n,+}\) has Jacobian determinant
\(s^{n-1}\). Hence
\begin{equation}
\label{eq:cone-vol}
\operatorname{vol}\bigl(P(c)\cap C_{n,+}\bigr)
=
\frac1n\int_F t_{\mathrm{exact}}(x)^n\,dx.
\end{equation}
Overlaps of cones have measure zero, so summing over the \(2n\) cones
recovers \(\operatorname{vol}(P(c))\). The integral with \(t_F\) in
place of \(t_{\mathrm{exact}}\) is the volume of
\(\operatorname{conv}(\{0\}\cup K_F(c))\cap C_{n,+}\), which sits
inside \(P(c)\).

\begin{lemma}[From envelopes to volume]
\label{lem:volume-from-Phi}
There exist \(\delta_n\in(0,1/2]\) and \(C_n<\infty\) such that if
\(\|c\|_\infty\le\delta_n\), then
\begin{equation}
\label{eq:volume-from-Phi}
\operatorname{vol}\bigl(P(c)\bigr)-2^n
\ge
\sum_{F}\Phi(c|_F)
+
n\sum_{q\in Q_n}c_q
-C_n\|c\|_\infty^2.
\end{equation}
If \(c\in\mathbf1^\perp\), the middle sum vanishes.
The restriction \(\delta_n\le1/2\) is only to invoke
Lemma~\ref{lem:tF-jet}; the quadratic control of the binomial remainder
in \(t_F^n\) requires the further shrinking encoded in \(\delta_n\).
\end{lemma}

\begin{proof}
By~\eqref{eq:t-exact-ge-tF} and~\eqref{eq:cone-vol},
\[
\operatorname{vol}\bigl(P(c)\bigr)
\ge
\sum_F\frac1n\int_F t_F^n\,dx.
\]
Lemma~\ref{lem:tF-jet} gives \(t_F=1+U_{c,F}+\theta_F\) with
\(|\theta_F|\le2\|c\|_\infty^2\) and \(|U_{c,F}|\le\|c\|_\infty\),
provided \(\|c\|_\infty\le1/2\). Write \(w=U_{c,F}+\theta_F\). Then
\(|w|\le\|c\|_\infty+2\|c\|_\infty^2\). Choose \(\delta_n\le1/2\) small
enough that \(|w|\le1/2\) throughout the ball
\(\|c\|_\infty\le\delta_n\). The binomial expansion
\[
t_F^n
=
(1+w)^n
=
1+nU_{c,F}+R_F
\]
then satisfies \(|R_F|\le C_n'\|c\|_\infty^2\) with \(C_n'\) depending
only on \(n\), uniformly on \(F\), because
\(n\theta_F=O(\|c\|_\infty^2)\) and every term of degree at least two
in \(w\) is \(O(\|c\|_\infty^2)\). Integrating,
\[
\frac1n\int_F t_F^n
=
\frac{2^{n-1}}n
+\int_F U_{c,F}
+O_n\bigl(\|c\|_\infty^2\bigr).
\]
The sum of the constant terms is \(2^n\). Each vertex of \(Q_n\) lies
on exactly \(n\) coordinate facets, so
\(\sum_F\sum_{q\in F}c_q=n\sum_q c_q\).
Lemma~\ref{lem:pocket-Phi} rewrites the integrals of the envelopes as
the pocket sum plus that linear term.
\end{proof}

The inequality uses only \(\operatorname{conv}(\{0\}\cup K_F(c))\subset P(c)\).
It does not require any simplex of \(K_F(c)\) to be a face of \(P(c)\).

\begin{lemma}[Coercivity of the pocket sum]
\label{lem:Phi-sum-norm}
For every \(n\ge3\) and \(c\in\mathbf1^\perp\subset E_n\),
\begin{equation}
\label{eq:Phi-sum-explicit}
\sum_F\Phi(c|_F)
\ge
\kappa_{n-1}\Psi_n(c)
\quad\text{and}\quad
\sum_F\Phi(c|_F)
\ge
\kappa_n''\|c\|_\infty,
\qquad
\kappa_n''
:=
\frac{8}{3(n-1)}.
\end{equation}
\end{lemma}

\begin{proof}
Summing~\eqref{eq:Phi-coercive} over the \(2n\) facets gives
\(\sum_F\Phi(c|_F)\ge\kappa_{n-1}\Psi_n(c)\). To obtain the sharper
direct constant \(\kappa_n''\), average~\eqref{eq:Phi-square-sum} over
the \(\binom{n-1}{2}\) coordinate pairs \(j<k\) inside each facet
\(F_{i,\varepsilon}\) and sum over the \(2n\) facets. Each pair
\(1\le j<k\le n\) lies in the \(2(n-2)\) facets \(F_{i,\pm}\) with
\(i\notin\{j,k\}\), and for each such \(i\) the union of the \(2\)-faces
in \(F_{i,+}\) and \(F_{i,-}\) parallel to \((j,k)\) is the full set of
\(2^{n-2}\) squares \(w\in\{\pm1\}^{n-2}\) of \(Q_n\) parallel to
\((j,k)\). Therefore
\[
\sum_F\Phi(c|_F)
\ge
\frac{n-2}{3\binom{n-1}{2}}
\sum_{1\le j<k\le n}\sum_{w\in\{\pm1\}^{n-2}}\bigl|\Delta_{jk}c(w)\bigr|
=
\frac{2}{3(n-1)}
\sum_{1\le j<k\le n}\sum_{w\in\{\pm1\}^{n-2}}\bigl|\Delta_{jk}c(w)\bigr|.
\]
Since \(n\ge3\), the fibre \(\{\pm1\}^{n-2}\) is non-trivial: for every
\(w\in\{\pm1\}^{n-2}\), its antipode \(-w\neq w\) is a distinct fibre
vertex, and evenness \(c(-q)=c(q)\) gives
\(|\Delta_{jk}c(-w)|=|\Delta_{jk}c(w)|\). Hence
\(\sum_{w\in\{\pm1\}^{n-2}}|\Delta_{jk}c(w)|\ge 2\max_w|\Delta_{jk}c(w)|\).
Combining this with~\eqref{eq:c-infty-by-squares} yields
\[
\sum_F\Phi(c|_F)
\ge
\frac{4}{3(n-1)}\sum_{1\le j<k\le n}\max_{w\in\{\pm1\}^{n-2}}\bigl|\Delta_{jk}c(w)\bigr|
\ge
\frac{8}{3(n-1)}\|c\|_\infty.
\]
\end{proof}

\subsection{Primal and Mahler lower bounds}
\label{sec:increments-general}

For \(n\ge4\) the argument below produces only a lower bound. The
inclusion \(\operatorname{conv}(\{0\}\cup K_F(c))\subset P(c)\) does
not identify the facial lattice of \(P(c)\), so it cannot yield an
expansion
\(\operatorname{vol}(P(c))-2^n=H_n(c)+O(\|c\|_\infty^2)\)
with a homogeneous form \(H_n\) determined by the hull. Dimension
three is the only case in which chamber polynomials give a matching
first-order description. Target~B is
\(\Delta V\ge a_n\|c\|_\infty\),
not a first-order asymptotic of \(\Delta V\).

\begin{theorem}[Primal lower bound on the isovertex stratum]
\label{thm:primal-general}
Let \(n\ge3\). There exist \(a_n,r_n>0\) such that if
\(c\in\mathbf1^\perp\) and \(\|c\|_\infty<r_n\), then
\begin{equation}
\label{eq:primal-general}
\operatorname{vol}\bigl(P(c)\bigr)-2^n
\ge
a_n\|c\|_\infty.
\end{equation}
Specifically, on \(\|c\|_\infty\le\delta_n\) one has
\(\operatorname{vol}(P(c))-2^n\ge\frac{8}{3(n-1)}\|c\|_\infty-C_n\|c\|_\infty^2\),
so one may take \(a_n=\frac{4}{3(n-1)}\).
\end{theorem}

\begin{proof}
Lemmas~\ref{lem:volume-from-Phi} and~\ref{lem:Phi-sum-norm} give
\[
\operatorname{vol}\bigl(P(c)\bigr)-2^n
\ge
\kappa_n''\|c\|_\infty
-C_n\|c\|_\infty^2
\]
on \(\|c\|_\infty\le\delta_n\), with \(\kappa_n''=\frac{8}{3(n-1)}\). Choose
\(r_n\le\min\bigl(\delta_n,\kappa_n''/(2C_n)\bigr)\). Then
\(C_n\|c\|_\infty^2\le(\kappa_n''/2)\|c\|_\infty\),
and~\eqref{eq:primal-general} holds with
\(a_n=\kappa_n''/2=\frac{4}{3(n-1)}\).
\end{proof}

The comparison with \(\|\mathcal R_n(c)\|_\infty\) is not used above.
Target~B runs through \(\|c\|_\infty\), not through \(\mathfrak F_n\).

\begin{remark}[The two defect quantities are locally equivalent, the volume is not]
\label{rem:Fn-Psi}
On \(\mathbf1^\perp\), Lemma~\ref{lem:Psi-seminorm} gives
\(\Psi_n(c)\asymp\|c\|_\infty\) for the seminorm \(\Psi_n\), while
Theorem~\ref{thm:main-rigidity} and~\eqref{eq:coercivity} give
\(\mathfrak F_n\bigl(P(c)\bigr)=\|\mathcal R_n(c)\|_\infty\asymp\|c\|_\infty\)
locally for the nonlinear coplanarity defect. Both defect quantities vanish
precisely on \(\operatorname{span}\{\mathbf1\}\) near the origin, so
finite-dimensional equivalence yields
\[
\mathfrak F_n\bigl(P(c)\bigr)
\asymp
\Psi_n(c)
\asymp
\|c\|_\infty
\]
near the origin. Theorem~\ref{thm:primal-general} then reads
\(\Delta V\gtrsim\Psi_n(c)\). This is a lower bound. It is not the
identification \(\Delta V\asymp\Psi_n(c)\), which would require a
matching first-order expansion of the hull and is not claimed for
\(n\ge4\).
\end{remark}

By definition,
\(d_{GL}(P,Q)=\inf_{T\in GL(n)}d_H(TP,Q)\). The identity lies in
\(GL(n)\), so \(d_{GL}(P(c),\mathcal C_n)\le d_H(P(c),\mathcal C_n)\).
Each vertex moves by
\(\|(1+c_q)q-q\|_2=\sqrt n\,|c_q|\), hence
\begin{equation}
\label{eq:dGL-upper}
d_{GL}\bigl(P(c),\mathcal C_n\bigr)
\le
d_H\bigl(P(c),\mathcal C_n\bigr)
\le
\sqrt n\,\|c\|_\infty.
\end{equation}
Thus \(\|c\|_\infty\ge n^{-1/2}d_{GL}(P(c),\mathcal C_n)\).

\begin{theorem}[Mahler increment on the isovertex stratum]
\label{thm:mahler-general}
Let \(n\ge3\). There exist \(a_n'',a_n',r_n>0\) such that if
\(c\in\mathbf1^\perp\) and \(\|c\|_\infty<r_n\), then
\begin{equation}
\label{eq:mahler-general}
\mathcal V_M\bigl(P(c)\bigr)-\mathcal V_M(\mathcal C_n)
\ge
a_n''\|c\|_\infty
\ge
a_n'\,d_{GL}\bigl(P(c),\mathcal C_n\bigr),
\end{equation}
where one may take \(a_n''=\frac{2^{n+2}}{3(n-1)\,n!}\) and
\(a_n'=n^{-1/2}a_n''\).
\end{theorem}

\begin{proof}
Proposition~\ref{prop:polar-quadratic} gives
\(\operatorname{vol}(P(c)^\circ)=\operatorname{vol}(\mathcal O_n)+O(\|c\|_\infty^2)\)
on \(\mathbf1^\perp\). Write
\(\Delta V=\operatorname{vol}(P(c))-2^n\). Then~\eqref{eq:mahler-split}
expands as
\[
\Delta\mathcal V_M
=
\operatorname{vol}(\mathcal O_n)\,\Delta V
+
2^n\,O(\|c\|_\infty^2)
+
\Delta V\cdot O(\|c\|_\infty^2).
\]
By Lemmas~\ref{lem:volume-from-Phi} and~\ref{lem:Phi-sum-norm},
\(\Delta V\ge\frac{8}{3(n-1)}\|c\|_\infty-C_n\|c\|_\infty^2\).
On a bounded ball, \(P(c)\) remains in a fixed compact set of convex
bodies, and volume is locally Lipschitz with respect to Hausdorff
distance, so \(\Delta V=O(\|c\|_\infty)\) by the estimate
\(d_H(P(c),\mathcal C_n)\le\sqrt n\|c\|_\infty\) in~\eqref{eq:dGL-upper}.
The product term is therefore \(O(\|c\|_\infty^3)\). Absorbing the
quadratic and cubic remainders on a smaller ball halves the leading
coefficient \(\operatorname{vol}(\mathcal O_n)\frac{8}{3(n-1)}=\frac{2^{n+3}}{3(n-1)\,n!}\)
and yields \(\Delta\mathcal V_M\ge a_n''\|c\|_\infty\). Combined
with~\eqref{eq:dGL-upper}, this gives~\eqref{eq:mahler-general}.
\end{proof}

Theorem~\ref{thm:mahler-general} is a lower bound on a ball in
\(\mathbf1^\perp\). It does not identify one-sided derivatives of
\(t\mapsto\mathcal V_M(P(tc))\). Those remain the content of the exact
jets in dimension three, Section~\ref{sec:folding-n3}.

The mechanism is
\[
\begin{aligned}
c\in E_n\setminus\operatorname{span}\{\mathbf1\}
&\;\Longrightarrow\;
\text{some facet restriction is not affine}\\
&\;\Longrightarrow\;
\sum_F\Phi(c|_F)\ge\frac{8}{3(n-1)}\|c\|_\infty\\
&\;\Longrightarrow\;
\operatorname{vol}\bigl(P(c)\bigr)-2^n\gtrsim\|c\|_\infty.
\end{aligned}
\]
No chamber of the incidence arrangement is used, and no
\((n+1)\)-determinant is required to span a face of \(P(c)\). The
operator \(\mathcal R_n\) is not used in
Theorems~\ref{thm:primal-general} and~\ref{thm:mahler-general}.

\section{Folding in dimension three}
\label{sec:folding-n3}

The lower bounds of Theorems~\ref{thm:primal-general} and~\ref{thm:mahler-general} do not use a chamber list. This section records the explicit coefficient \(\frac83\) in dimension three, together with the absence of a bilateral derivative of \(\mathcal V_M\circ P\) at the origin along \(\mathbf 1^\perp\). An exact evaluation of incidence determinants shows that the open octants of pair space are not all single combinatorial chambers near the origin: the four even octants split into a generic chamber and three thin chambers, whereas the four odd octants are single chambers near the origin.

\begin{remark}[Radial expansion, not a Fréchet derivative]
\label{rem:n3-not-frechet}
The first-order increment
\(V(c)-V_0=\frac83\|c\|_\infty+O(\|c\|_\infty^2)\)
is positively homogeneous of degree one and is not a linear form on
\(\mathbf1^\perp\). It is a first-order radial expansion, not a
Fréchet derivative of \(c\mapsto\operatorname{vol}(P(c))\) at the
origin. One-sided directional derivatives along rays exist; a
bilateral derivative along any line in \(\mathbf1^\perp\) does not.
The polar volume \(V(h)\) does have a linear jet in the support
parameters \(h\), Lemma~\ref{lem:polar-linear-vanishes}.
\end{remark}

\subsection{Pair coordinates}
\label{sec:n3-pair-coords}

The space \(\mathbf1^\perp\subset E_3\) is three-dimensional. Every \(c\in\mathbf1^\perp\) has the unique expansion \(c(q) = \alpha\,q_1q_2 + \beta\,q_1q_3 + \gamma\,q_2q_3\), and \(\|c\|_\infty = |\alpha|+|\beta|+|\gamma|\). The three coordinates of \(\mathcal R_3\) are the four-point determinants of \(F_{1,+}\), \(F_{2,+}\) and \(F_{3,+}\). Direct expansion gives
\begin{equation}
\label{eq:n3-R-linear}
\|\mathcal R_3(c)\|_\infty = 16\max\bigl(|\alpha|,|\beta|,|\gamma|\bigr) + O\bigl(\|c\|_\infty^2\bigr).
\end{equation}

\subsection{Frozen triangulation and the thin walls}
\label{sec:n3-frozen-octant}

Order \(Q_3\) lexicographically as
\(q^0=(-1,-1,-1),\dots,q^7=(+1,+1,+1)\). Write \(v^k=v_{q^k}(c)\).
A candidate boundary of \(P(c)\) in the open octant
\(\alpha>0\), \(\beta>0\), \(\gamma>0\) is the list of twelve triangles
\begin{equation}
\label{eq:n3-triangles}
\begin{aligned}
&\{v^6,v^2,v^0\},
&&\{v^6,v^2,v^7\},
&&\{v^3,v^1,v^0\},
&&\{v^3,v^2,v^0\},\\
&\{v^3,v^1,v^7\},
&&\{v^3,v^2,v^7\},
&&\{v^5,v^1,v^7\},
&&\{v^5,v^1,v^0\},\\
&\{v^4,v^6,v^0\},
&&\{v^4,v^5,v^0\},
&&\{v^4,v^6,v^7\},
&&\{v^4,v^5,v^7\}.
\end{aligned}
\end{equation}

\begin{lemma}[Sixty incidence determinants and explicit nonvanishing radius]
\label{lem:n3-sixty}
Consider the twelve triangles~\eqref{eq:n3-triangles}. Among their sixty
\(4\times 4\) incidence determinants
\(D_{T,p}(\alpha,\beta,\gamma)=\det[(v^i,1),(v^j,1),(v^k,1),(v^p,1)]\)
(\(p\notin T=\{i,j,k\}\)) in \(\mathbb Z[\alpha,\beta,\gamma]\), twelve
vanish at \((\alpha,\beta,\gamma)=(0,0,0)\) and factor as \(\pm 16 F_1\),
\(\pm 16 F_2\), or \(\pm 16 F_3\), where
\begin{equation}
\label{eq:n3-thin-walls}
\begin{aligned}
F_1 &:= \alpha^3-\alpha\beta^2-\alpha\gamma^2-\alpha+2\beta\gamma
      = -\alpha D_1 + 2\beta\gamma,\\
F_2 &:= \alpha^2\beta-2\alpha\gamma-\beta^3+\beta\gamma^2+\beta
      = \beta D_2 - 2\alpha\gamma,\\
F_3 &:= \alpha^2\gamma-2\alpha\beta+\beta^2\gamma-\gamma^3+\gamma
      = \gamma D_3 - 2\alpha\beta,
\end{aligned}
\end{equation}
with \(D_1=1-\alpha^2+\beta^2+\gamma^2\),
\(D_2=1+\alpha^2-\beta^2+\gamma^2\), and
\(D_3=1+\alpha^2+\beta^2-\gamma^2\). The remaining forty-eight determinants
evaluate to \(\pm 8\) at the origin, and in each of them the sums of the
absolute values of the homogeneous coefficients of degrees \(1\), \(2\), and
\(3\) are at most \(24\), \(72\), and \(88\), respectively. Consequently, on
the explicit closed ball
\[
\|c\|_\infty=|\alpha|+|\beta|+|\gamma|\le r_0:=\frac18=0.125,
\]
every coordinate satisfies \(\max\{|\alpha|,|\beta|,|\gamma|\}\le\|c\|_\infty\le r_0\),
so every one of those forty-eight determinants satisfies
\[
\bigl|D_{T,p}(\alpha,\beta,\gamma)\bigr|
\ge
8-\bigl(24\,r_0+72\,r_0^2+88\,r_0^3\bigr)
=
8-\frac{275}{64}
=
\frac{237}{64}
=
3.703125>0
\]
and retains its strict origin sign, while
\(D_1,D_2,D_3\in[61/64,67/64]\subset(1/2,2)\). The exact same
degree-by-degree coefficient bounds \((24,72,88)\) and explicit nonvanishing
radius \(r_0=\frac18\) hold for the forty-eight nonvanishing origin
determinants of the odd generic list~\eqref{eq:n3-triangles-odd}, the
thin/even-wall list~\eqref{eq:n3-triangles-thin}, and the odd-wall
list~\eqref{eq:n3-triangles-odd-wall}.
\end{lemma}

\begin{proof}
The polynomial expansions and the degree-by-degree coefficient bounds
\(d_1\le 24\), \(d_2\le 72\), \(d_3\le 88\) are finite identities in
\(\mathbb Z[\alpha,\beta,\gamma]\), verified symbolically in
Appendix~\ref{app:n3-certificate}. Since
\(\max\{|\alpha|,|\beta|,|\gamma|\}\le|\alpha|+|\beta|+|\gamma|=\|c\|_\infty\le r_0=\frac18\),
every monomial of total degree \(m\in\{1,2,3\}\) has absolute value at most
\(r_0^m\). By the triangle inequality, for \(\|c\|_\infty\le r_0=\frac18\)
one has
\(|D_{T,p}(c)-D_{T,p}(0)|\le 24(1/8)+72(1/64)+88(1/512)=275/64<8\), and
\(|D_i-1|\le\alpha^2+\beta^2+\gamma^2\le r_0^2=1/64<3/64\).
\end{proof}

With the explicit radius \(r_0=\frac18\) of Lemma~\ref{lem:n3-sixty}, define
the open positive generic chamber \(C_{\mathrm{gen}}\) by
\[
C_{\mathrm{gen}}
:=
\bigl\{(\alpha,\beta,\gamma):\ \alpha>0,\ \beta>0,\ \gamma>0,\ \|c\|_\infty=\alpha+\beta+\gamma<r_0,\ F_1<0,\ F_2>0,\ F_3>0\bigr\}.
\]
This is a definition by strict signs of the three incidence factors inside
the explicit ball \(\|c\|_\infty<r_0\); by Lemma~\ref{lem:n3-sixty}, none of
the sixty incidence determinants of~\eqref{eq:n3-triangles} vanishes on
\(C_{\mathrm{gen}}\).

\begin{lemma}[Path-connectedness of \(C_{\mathrm{gen}}\)]
\label{lem:Cgen-starshaped}
Let \(r_0=\frac18\) be as in Lemma~\ref{lem:n3-sixty}. On the ball
\(0<\|c\|_\infty=\alpha+\beta+\gamma<r_0\), one has
\(D_i\in[61/64,67/64]\subset(1/2,2)\) and
\begin{equation}
\label{eq:Cgen-graphs}
F_1<0\iff\alpha>\frac{2\beta\gamma}{D_1},\qquad
F_2>0\iff\beta>\frac{2\alpha\gamma}{D_2},\qquad
F_3>0\iff\gamma>\frac{2\alpha\beta}{D_3}.
\end{equation}
Consequently, \(C_{\mathrm{gen}}\) coincides with
\[
W=\Bigl\{(\alpha,\beta,\gamma)\in(0,\infty)^3:\
\alpha+\beta+\gamma<r_0,\
\alpha>\frac{2\beta\gamma}{D_1},\
\beta>\frac{2\alpha\gamma}{D_2},\
\gamma>\frac{2\alpha\beta}{D_3}\Bigr\}
\]
and is starlike with respect to the origin: if \(p\in C_{\mathrm{gen}}\)
then the half-open radial segment \((0,p]:=\{sp:s\in(0,1]\}\) lies in
\(C_{\mathrm{gen}}\). Moreover, every ray in the open positive octant meets
\(C_{\mathrm{gen}}\), and \(C_{\mathrm{gen}}\) is path-connected.
\end{lemma}

\begin{proof}
By Lemma~\ref{lem:n3-sixty}, \(D_i\in[61/64,67/64]\subset(1/2,2)\) on
\(\alpha+\beta+\gamma<r_0=\frac18\), so~\eqref{eq:Cgen-graphs} holds and
\(C_{\mathrm{gen}}=W\).

Let \(p=(\alpha,\beta,\gamma)\in W\) and \(s\in(0,1]\). Clearly
\(\|sp\|_\infty=s(\alpha+\beta+\gamma)<r_0\). The first inequality along
\(sp\) reads \(s\alpha>2s^2\beta\gamma/D_1(sp)\), i.e.
\(\alpha>2s\beta\gamma/D_1(sp)\). Since already
\(\alpha>2\beta\gamma/D_1(p)\), it is enough that
\(D_1(sp)\ge s\,D_1(p)\). Writing \(K_1=-\alpha^2+\beta^2+\gamma^2\), one has
\(|K_1|\le r_0^2=1/64<1/2\) and
\[
D_1(sp)-s D_1(p)=(1-s)(1-s K_1)\ge 0.
\]
The same computation applies to \(D_2\) and \(D_3\), with
\(K_2=\alpha^2-\beta^2+\gamma^2\) and \(K_3=\alpha^2+\beta^2-\gamma^2\).
Hence \((0,p]\subset W=C_{\mathrm{gen}}\).

Starlikeness alone does not give path-connectedness: \(0\notin C_{\mathrm{gen}}\),
and a union of disjoint open cones with vertex at the origin is starlike
about the origin yet disconnected. Connectedness uses the link of the
origin in the positive octant. The leading terms of~\eqref{eq:n3-thin-walls}
give, for every \(u=(\alpha_0,\beta_0,\gamma_0)\) with
\(\alpha_0,\beta_0,\gamma_0>0\),
\[
F_1(tu)=-\alpha_0 t+O(t^2)<0,\qquad
F_2(tu)=\beta_0 t+O(t^2)>0,\qquad
F_3(tu)=\gamma_0 t+O(t^2)>0
\]
for all sufficiently small \(t>0\). Thus every ray in the open positive
octant meets \(W\). Let
\(\Delta=\{(\alpha,\beta,\gamma)\in(0,\infty)^3:\alpha+\beta+\gamma=1\}\)
be the open 2-simplex of barycentric directions, which is convex and hence
path-connected. Given \(p,q\in C_{\mathrm{gen}}\), write
\(\hat p=p/(\alpha_p+\beta_p+\gamma_p)\) and
\(\hat q=q/(\alpha_q+\beta_q+\gamma_q)\), and set
\(\hat u(\theta)=(1-\theta)\hat p+\theta\hat q\) for \(\theta\in[0,1]\).
The image of \(\hat u\) is compact in \(\Delta\), so
\[
m=\min_{\theta\in[0,1]}\min\bigl\{\hat u_1(\theta),\hat u_2(\theta),\hat u_3(\theta)\bigr\}>0.
\]
On this compact set of directions one has
\(\min\{\alpha_0,\beta_0,\gamma_0\}\ge m\), hence the \(O(t^2)\) remainder
in the leading-term expansions is uniform in \(\theta\). Therefore there
exists \(t_*\in(0,r_0)\) such that \(t\,\hat u(\theta)\in W\) for every
\(\theta\in[0,1]\) and every \(t\in(0,t_*)\). Choose \(t\in(0,t_*)\) small
enough that \(t\hat p\) lies on \((0,p]\) and \(t\hat q\) lies on \((0,q]\).
Concatenating the starlike segments \(p\to t\hat p\) and \(q\to t\hat q\)
with the lifted path \(\theta\mapsto t\,\hat u(\theta)\) yields a continuous
path from \(p\) to \(q\) inside \(C_{\mathrm{gen}}\).
\end{proof}

\begin{lemma}[Face list of the generic positive chamber]
\label{lem:n3-chamber}
\label{lem:n3-gen-faces}
On the open chamber \(C_{\mathrm{gen}}\) with \(r_0=\frac18\), the diagonal
ray segment \((\alpha,\beta,\gamma)=(\varepsilon,\varepsilon,\varepsilon)\)
belongs to \(C_{\mathrm{gen}}\) for every
\(\varepsilon\in(0,r_0/3)=(0,1/24)\), the explicit sample \((3,2,1)/100\)
belongs to \(C_{\mathrm{gen}}\), and throughout \(C_{\mathrm{gen}}\) the
convex hull \(P(c)\) has the twelve triangular faces~\eqref{eq:n3-triangles}.
\end{lemma}

\begin{proof}
For every \(\varepsilon\in(0,r_0/3)\), the point
\((\varepsilon,\varepsilon,\varepsilon)\) satisfies
\(\|c\|_\infty=3\varepsilon<r_0=\frac18\) and
\(F_1=-F_2=-F_3=-\varepsilon(1-2\varepsilon+\varepsilon^2)<0\), so it lies
in \(C_{\mathrm{gen}}\). At the explicit sample
\((\alpha,\beta,\gamma)=(3,2,1)/100\), one has
\(\|c\|_\infty=0.03+0.02+0.01=0.06<r_0=\frac18\), and direct evaluation
of~\eqref{eq:n3-thin-walls} gives
\[
F_1=-\frac{29588}{10^6}<0,\qquad
F_2=\frac{19412}{10^6}>0,\qquad
F_3=\frac{8812}{10^6}>0,
\]
so \((3,2,1)/100\in C_{\mathrm{gen}}\). By Lemma~\ref{lem:n3-sixty}, none of
the sixty incidence determinants of~\eqref{eq:n3-triangles} vanishes on
\(C_{\mathrm{gen}}\). Exact rational evaluation at \((3,2,1)/100\)
(Appendix~\ref{app:n3-certificate}) confirms that the twelve
triangles~\eqref{eq:n3-triangles} are strictly supporting faces of \(P(c)\)
at that sample. Since \(C_{\mathrm{gen}}\) is path-connected by
Lemma~\ref{lem:Cgen-starshaped}, the face list~\eqref{eq:n3-triangles}
persists throughout \(C_{\mathrm{gen}}\).
\end{proof}

\begin{lemma}[Volume polynomial in the generic chamber]
\label{lem:n3-vol-positive}
On \(C_{\mathrm{gen}}\),
\begin{equation}
\label{eq:n3-vol-poly}
\begin{aligned}
\operatorname{vol}\bigl(P(c)\bigr)
&=
8
+\frac83(\alpha+\beta+\gamma)
-8(\alpha^2+\beta^2+\gamma^2)
-\frac{16}3(\alpha\beta+\alpha\gamma+\beta\gamma)
\\
&\qquad
-\frac83(\alpha^3+\beta^3+\gamma^3)
+\frac83(\alpha^2\beta+\alpha^2\gamma+\beta^2\alpha+\beta^2\gamma
+\gamma^2\alpha+\gamma^2\beta)
+16\alpha\beta\gamma.
\end{aligned}
\end{equation}
\end{lemma}

\begin{proof}
Lemma~\ref{lem:n3-gen-faces} identifies the faces. Summing the twelve
oriented origin-tetrahedra on~\eqref{eq:n3-triangles} expands
to~\eqref{eq:n3-vol-poly} as an exact identity in
\(\mathbb Q[\alpha,\beta,\gamma]\).
\end{proof}

\begin{lemma}[Generic rays lie in \(C_{\mathrm{gen}}\)]
\label{lem:n3-gen-rays}
Let \(u\in\mathbf1^\perp\) have pair coordinates
\(\alpha_0>0\), \(\beta_0>0\), \(\gamma_0>0\). For all sufficiently
small \(t>0\), the point \(tu\) lies in \(C_{\mathrm{gen}}\).
\end{lemma}

\begin{proof}
The leading terms of~\eqref{eq:n3-thin-walls} give
\(F_1(tu)=-\alpha_0 t+O(t^2)<0\),
\(F_2(tu)=\beta_0 t+O(t^2)>0\), and
\(F_3(tu)=\gamma_0 t+O(t^2)>0\).
Also \(\|tu\|_\infty=t(\alpha_0+\beta_0+\gamma_0)<r_0\) for small
\(t>0\). No segment is used.
\end{proof}

\begin{remark}[Chamberwise versus raywise]
\label{rem:chamber-vs-ray}
Two distinct statements are in force. Chamberwise: the face lattice
is constant on the whole of \(C_{\mathrm{gen}}\) by
Lemmas~\ref{lem:Cgen-starshaped} and~\ref{lem:n3-gen-faces}. Raywise:
for each fixed direction \(u\) with \(\alpha_0\beta_0\gamma_0\neq0\),
every incidence determinant restricts to a univariate polynomial
along \(t\mapsto tu\); those which do not vanish at the origin stay
non-zero for small \(t\), while those which do are integer multiples
of \(F_1,F_2,F_3\) and have leading term \(\pm t\). The directional
folding statement uses only the raywise constancy, recorded separately
in Lemma~\ref{lem:n3-ray-stable}. Thin chambers \(\{F_1>0\}\),
\(\{F_2<0\}\), and \(\{F_3<0\}\) in the positive octant (and their
\(B_3\)-images in the even octants) meet a neighbourhood of the origin
but contain no ray through the origin.
\end{remark}

\begin{lemma}[Volume on the local thin chamber \(F_1>0\)]
\label{lem:n3-vol-thin}
With the explicit radius \(r_0=\frac18=0.125\) of Lemma~\ref{lem:n3-sixty},
define the local thin chamber
\[
C_{\mathrm{thin}}
:=
\bigl\{(\alpha,\beta,\gamma):\ \alpha\ge 0,\ \beta>0,\ \gamma>0,\ \|c\|_\infty=\alpha+\beta+\gamma<r_0,\ F_1>0,\ F_2>0,\ F_3>0\bigr\}.
\]
This region is path-connected, contains for every
\(\varepsilon\in(0,1/20]\) the parametric curve
\((\alpha,\beta,\gamma)=(\varepsilon^3,\varepsilon,\varepsilon)\) (for which
\(\|c\|_\infty=\varepsilon^3+2\varepsilon\le 0.100125<r_0\)) as well as the
explicit sample \((1,50,50)/1000\) (for which
\(\|c\|_\infty=0.001+0.05+0.05=0.101<r_0=0.125\)), and its closure meets the
closure of \(C_{\mathrm{gen}}\) along the surface \(F_1=0\). Throughout
\(C_{\mathrm{thin}}\), the convex hull \(P(c)\) has the twelve triangular
faces
\begin{equation}
\label{eq:n3-triangles-thin}
\begin{aligned}
&\{v^0,v^1,v^3\},&&\{v^0,v^1,v^5\},&&\{v^0,v^2,v^3\},&&\{v^0,v^2,v^4\},\\
&\{v^0,v^4,v^5\},&&\{v^1,v^3,v^5\},&&\{v^2,v^3,v^7\},&&\{v^2,v^4,v^6\},\\
&\{v^2,v^6,v^7\},&&\{v^3,v^5,v^7\},&&\{v^4,v^5,v^7\},&&\{v^4,v^6,v^7\},
\end{aligned}
\end{equation}
and the sum of the twelve origin-tetrahedra \(\frac16|\det(v^i,v^j,v^k)|\)
equals the polynomial
\begin{equation}
\label{eq:n3-vol-thin-poly}
\begin{aligned}
\operatorname{vol}\bigl(P(c)\bigr)
&=
8-\tfrac83\alpha+\tfrac83(\beta+\gamma)
-8(\alpha^2+\beta^2+\gamma^2)
-\tfrac{16}3(\alpha\beta+\alpha\gamma)+\tfrac{16}3\beta\gamma\\
&\qquad
+\tfrac83\alpha^3-\tfrac83(\beta^3+\gamma^3)
+\tfrac83\alpha^2(\beta+\gamma)-\tfrac83\alpha(\beta^2+\gamma^2)
+\tfrac83(\beta^2\gamma+\beta\gamma^2)+16\alpha\beta\gamma.
\end{aligned}
\end{equation}
The linear part is \(\frac83(-\alpha+\beta+\gamma)\). On \(C_{\mathrm{thin}}\) one has
\(0\le\alpha<2\beta\gamma/D_1=O(\|c\|_\infty^2)\), hence
\(-\alpha+\beta+\gamma=\|c\|_\infty+O(\|c\|_\infty^2)\) and the linear increment is
\(\frac83\|c\|_\infty+O(\|c\|_\infty^2)\). The chambers with \(F_2<0\) or
\(F_3<0\) as the thin inequality are permutations of pair coordinates.
\end{lemma}

\begin{proof}
On \(\|c\|_\infty=\alpha+\beta+\gamma<r_0=\frac18\),
\(D_1,D_2,D_3\in[61/64,67/64]\), so \(F_1>0\) is equivalent to
\(f(\alpha):=\alpha D_1(\alpha,\beta,\gamma)=\alpha(1-\alpha^2+\beta^2+\gamma^2)<2\beta\gamma\).
For fixed \(\beta,\gamma\in(0,r_0)\) with \(\beta+\gamma<r_0\), the function
\(\alpha\mapsto f(\alpha)\) on \([0,r_0-\beta-\gamma)\) has derivative
\[
f'(\alpha)=1-3\alpha^2+\beta^2+\gamma^2\ge 1-\frac{3}{64}=\frac{61}{64}>0,
\]
so \(f\) is strictly increasing with \(f(0)=0\). Consequently, lowering
\(\alpha\) linearly to \(0\) at fixed \((\beta,\gamma)\) strictly decreases
both \(\|c\|_\infty=\alpha+\beta+\gamma<r_0\) and \(f(\alpha)<2\beta\gamma\),
preserving \(F_1>0\). Moreover, whenever \(0\le\alpha<2\beta\gamma/D_1\) and
\(\beta,\gamma<r_0=\frac18\), one has
\(2\alpha\gamma<\frac{4\gamma^2}{D_1}\beta<\frac{4/64}{61/64}\beta=\frac4{61}\beta<\frac{61}{64}\beta\le\beta D_2\),
so \(F_2>0\) holds automatically, and by symmetry \(F_3>0\) holds
automatically as well. Thus the linear retraction \(\alpha\mapsto t\alpha\)
(\(t\in[0,1]\)) stays inside \(C_{\mathrm{thin}}\) and deformation-retracts
\(C_{\mathrm{thin}}\) onto the convex planar triangle
\(\{\alpha=0,\ \beta>0,\ \gamma>0,\ \beta+\gamma<r_0\}\subset C_{\mathrm{thin}}\).
Hence \(C_{\mathrm{thin}}\) is path-connected. For every
\(\varepsilon\in(0,1/20]\), the point
\((\varepsilon^3,\varepsilon,\varepsilon)\) satisfies
\(\|c\|_\infty=\varepsilon^3+2\varepsilon\le\frac1{8000}+\frac1{10}=0.100125<r_0=0.125\)
and \(0<\varepsilon^3<2\varepsilon^2/(1+2\varepsilon^2-\varepsilon^6)=2\varepsilon^2/D_1\),
so \((\varepsilon^3,\varepsilon,\varepsilon)\in C_{\mathrm{thin}}\). At the
explicit sample \((1,50,50)/1000\), one has
\(\|c\|_\infty=0.001+0.05+0.05=0.101<r_0=0.125\),
\(F_1=3995001/10^9>0\), and \(F_2=F_3=49900050/10^9>0\), so
\((1,50,50)/1000\in C_{\mathrm{thin}}\).

By Lemma~\ref{lem:n3-sixty}, on \(\|c\|_\infty<r_0\) the forty-eight
nonvanishing origin incidence determinants of~\eqref{eq:n3-triangles-thin}
are strictly nonzero, and the twelve origin-vanishing determinants are
\(\pm 16 F_1\), \(\pm 16 F_2\), and \(\pm 16 F_3\), which are nonzero on
\(C_{\mathrm{thin}}\). Exact rational evaluation at \((1,50,50)/1000\)
(Appendix~\ref{app:n3-certificate}) confirms that the twelve
triangles~\eqref{eq:n3-triangles-thin} are strictly supporting faces of
\(P(c)\). By path-connectedness of \(C_{\mathrm{thin}}\), the face
list~\eqref{eq:n3-triangles-thin} persists throughout \(C_{\mathrm{thin}}\).
Because \(c(-q)=c(q)\) and \(\|c\|_\infty<r_0<1\), the eight vertices
\(v^q=(1+c_q)q\) occur in four antipodal pairs \(\pm(1+c_q)q\) that span
\(\mathbb R^3\), so the origin lies strictly in the interior of \(P(c)\).
Summing the volumes of the twelve oriented origin-tetrahedra therefore gives
the exact volume polynomial~\eqref{eq:n3-vol-thin-poly}.
\end{proof}

\begin{corollary}[Exact volume transition across the cubic incidence wall]
\label{cor:n3-wall-transition}
Let \(V_{\mathrm{gen}}(\alpha,\beta,\gamma)\) and
\(V_{\mathrm{thin}}(\alpha,\beta,\gamma)\) denote the volume
polynomials~\eqref{eq:n3-vol-poly} and~\eqref{eq:n3-vol-thin-poly}
of the positive generic chamber \(C_{\mathrm{gen}}\) and the adjacent
thin chamber \(C_{\mathrm{thin}}\). In
\(\mathbb Q[\alpha,\beta,\gamma]\), one has the exact algebraic identity
\begin{equation}
\label{eq:n3-transition-law}
V_{\mathrm{thin}}(\alpha,\beta,\gamma)
-
V_{\mathrm{gen}}(\alpha,\beta,\gamma)
=
\frac{16}{3}F_1(\alpha,\beta,\gamma),
\end{equation}
where
\[
F_1(\alpha,\beta,\gamma)
=
\alpha^3-\alpha\beta^2-\alpha\gamma^2-\alpha+2\beta\gamma.
\]

Geometrically, crossing the incidence wall \(F_1=0\) from
\(C_{\mathrm{gen}}\) (\(F_1<0\)) into \(C_{\mathrm{thin}}\)
(\(F_1>0\)) exchanges the active diagonals on the two antipodal
four-point configurations
\[
\{v^1,v^3,v^5,v^7\},
\qquad
\{v^0,v^2,v^4,v^6\}.
\]
With the vertex ordering fixed in Section~\ref{sec:n3-frozen-octant},
their incidence determinants are
\[
\begin{aligned}
\det[(v^1,1),(v^3,1),(v^5,1),(v^7,1)]
&=-16F_1,\\
\det[(v^0,1),(v^2,1),(v^4,1),(v^6,1)]
&=16F_1.
\end{aligned}
\]
Each diagonal exchange has swept tetrahedral volume
\(\frac{8}{3}|F_1|\). With the orientations induced by the
difference \(V_{\mathrm{thin}}-V_{\mathrm{gen}}\), each contributes
\(\frac{8}{3}F_1\) to the signed volume difference. Summing the two
contributions yields~\eqref{eq:n3-transition-law}.
\end{corollary}

\subsection{Sign orbits under the cube group}
\label{sec:n3-sign-orbits}

The group \(B_3\) acts on pair coordinates by
\[
(\alpha,\beta,\gamma)
\mapsto
(\varepsilon_1\varepsilon_2\alpha,\
\varepsilon_1\varepsilon_3\beta,\
\varepsilon_2\varepsilon_3\gamma)
\]
for sign flips \(\varepsilon\in\{\pm1\}^3\), and by permuting the three
labels for coordinate permutations. In all cases
\(\alpha\beta\gamma\) is invariant, so \(B_3\) does not mix the two
families of octants
\[
\{\alpha\beta\gamma>0\}
\quad\text{(even)},
\qquad
\{\alpha\beta\gamma<0\}
\quad\text{(odd)}.
\]
Each family is a single \(B_3\)-orbit of four octants. Volume
polynomials on the two families need not coincide.

Consider the candidate face list in the open negative octant
\(\alpha<0\), \(\beta<0\), \(\gamma<0\):
\begin{equation}
\label{eq:n3-triangles-odd}
\begin{aligned}
&\{v^0,v^1,v^2\},
&&\{v^0,v^1,v^4\},
&&\{v^0,v^2,v^4\},
&&\{v^1,v^2,v^3\},\\
&\{v^1,v^3,v^5\},
&&\{v^1,v^4,v^5\},
&&\{v^2,v^3,v^6\},
&&\{v^2,v^4,v^6\},\\
&\{v^3,v^5,v^7\},
&&\{v^3,v^6,v^7\},
&&\{v^4,v^5,v^6\},
&&\{v^5,v^6,v^7\}.
\end{aligned}
\end{equation}

\begin{lemma}[Incidence factors and volume on the negative octant]
\label{lem:n3-vol-odd}
Let the twelve triangles be~\eqref{eq:n3-triangles-odd} and let
\(r_0=\frac18=0.125\) be the explicit radius of Lemma~\ref{lem:n3-sixty}.
Among their sixty \(4\times 4\) incidence determinants in
\(\mathbb Z[\alpha,\beta,\gamma]\), exactly twelve vanish at the origin and
equal \(16(\pm F_k)\) for one of the three polynomials \(F_1,F_2,F_3\)
of~\eqref{eq:n3-thin-walls}, while the remaining forty-eight evaluate to
\(\pm 8\) at the origin and stay strictly nonzero on
\(\|c\|_\infty=|\alpha|+|\beta|+|\gamma|\le r_0\). On the entire open
negative octant inside this ball,
\[
C_-
:=
\bigl\{(\alpha,\beta,\gamma):\ \alpha<0,\ \beta<0,\ \gamma<0,\ \|c\|_\infty=-(\alpha+\beta+\gamma)<r_0\bigr\},
\]
one has \(F_1>0\), \(F_2<0\), and \(F_3<0\) everywhere. Consequently
no incidence determinant vanishes on \(C_-\), the twelve
triangles~\eqref{eq:n3-triangles-odd} are the faces of \(P(c)\)
throughout \(C_-\), and on \(C_-\),
\begin{equation}
\label{eq:n3-vol-odd}
\begin{aligned}
\operatorname{vol}\bigl(P(c)\bigr)
&=
8
-\frac83(\alpha+\beta+\gamma)
-8(\alpha^2+\beta^2+\gamma^2)
+\frac{16}3(\alpha\beta+\alpha\gamma+\beta\gamma)
\\
&\qquad
+\frac83(\alpha^3+\beta^3+\gamma^3)
-\frac83(\alpha^2\beta+\alpha^2\gamma+\beta^2\alpha+\beta^2\gamma
+\gamma^2\alpha+\gamma^2\beta)
+16\alpha\beta\gamma.
\end{aligned}
\end{equation}
The linear part is
\(-\frac83(\alpha+\beta+\gamma)=\frac83\|c\|_\infty\).
\end{lemma}

\begin{proof}
On each coordinate facet of \(Q_3\), the odd list~\eqref{eq:n3-triangles-odd}
triangulates the facet along the diagonal opposite to that used
in~\eqref{eq:n3-triangles}, while the underlying facet four-tuples of
vertices are the same faces. By Lemma~\ref{lem:n3-sixty} and the certificate
of Appendix~\ref{app:n3-certificate}, the twelve origin-vanishing
determinants of~\eqref{eq:n3-triangles-odd} equal \(16(\pm F_k)\), while the
forty-eight non-facet determinants evaluate to \(\pm 8\) at the origin and
remain strictly nonzero on \(\|c\|_\infty\le r_0=\frac18\).

For \(\alpha<0\), \(\beta<0\), \(\gamma<0\) with
\(\|c\|_\infty=-(\alpha+\beta+\gamma)<r_0=\frac18\), one has
\(D_i\in[61/64,67/64]>1/2\) and
\[
F_1 = -\alpha D_1 + 2\beta\gamma > 0,\qquad
F_2 = \beta D_2 - 2\alpha\gamma < 0,\qquad
F_3 = \gamma D_3 - 2\alpha\beta < 0,
\]
because the linear and quadratic terms have the same strict sign in
every case. Thus \(F_1F_2F_3\neq 0\) throughout \(C_-\). At the explicit
sample \((-1,-50,-50)/1000\), one has
\(\|c\|_\infty=0.001+0.05+0.05=0.101<r_0=0.125\), so the sample lies in
\(C_-\), and exact rational evaluation confirms that all twelve triangles
of~\eqref{eq:n3-triangles-odd} have strict supporting signs on the five
complementary vertices; since a convex \(3\)-polytope with \(V=8\)
vertices satisfies \(2E\ge 3F\) and Euler's relation \(V-E+F=2\), it has
at most \(F\le 2V-4=12\) facets, so these twelve distinct triangles
exhaust the facet list at the sample. Since \(C_-\) is convex (hence
path-connected) and no incidence determinant vanishes on \(C_-\), the face
lattice is constant on all of \(C_-\) and equals~\eqref{eq:n3-triangles-odd}.
Summing the signed volumes of the twelve origin-tetrahedra
gives~\eqref{eq:n3-vol-odd}.
\end{proof}

\begin{lemma}[Odd generic rays]
\label{lem:n3-odd-rays}
Let \(u\in\mathbf1^\perp\) have pair coordinates with
\(\alpha_0\beta_0\gamma_0<0\) and all three nonzero. Then for all
sufficiently small \(t>0\), the point \(tu\) lies in the
\(B_3\)-image of \(C_-\). In particular the first-order increment
is \(\frac83\|u\|_\infty\).
\end{lemma}

\begin{proof}
A sign flip \(\varepsilon\) with an even number of minuses sends
the octant of \(u\) to \(\{\alpha<0,\beta<0,\gamma<0\}\) and
conjugates the linear form of~\eqref{eq:n3-vol-odd} to
\(\frac83\sigma\cdot(\alpha_0,\beta_0,\gamma_0)\), where \(\sigma\)
is the sign pattern of that octant. By Lemma~\ref{lem:n3-vol-odd},
\(C_-\) is the entire open negative octant inside the ball
\(\|c\|_\infty<r_0\), so \(tu\) lies in the \(B_3\)-image
of \(C_-\) as soon as \(t\|u\|_\infty<r_0\). The four odd
octants are one \(B_3\)-orbit, so every such ray is covered.
\end{proof}

\subsection{The two-character wall}
\label{sec:n3-two-wall}

We now examine the coordinate wall \(\alpha=0\) near the origin, split into
the even quarter (\(\beta>0,\gamma>0\)) and the odd quarter
(\(\beta>0,\gamma<0\)).

\begin{lemma}[Explicit face lists and incidence factors on \(\alpha=0\)]
\label{lem:n3-wall-clear}
Restricting the three incidence polynomials~\eqref{eq:n3-thin-walls} to the
plane \(\alpha=0\) yields the exact factorizations
\begin{equation}
\label{eq:n3-wall-factors}
F_1(0,\beta,\gamma)=2\beta\gamma,\qquad
F_2(0,\beta,\gamma)=\beta(1-\beta^2+\gamma^2),\qquad
F_3(0,\beta,\gamma)=\gamma(1+\beta^2-\gamma^2).
\end{equation}
With the explicit radius \(r_0=\frac18=0.125\) of Lemma~\ref{lem:n3-sixty},
for \(0<\|c\|_\infty=|\beta|+|\gamma|<r_0\), both quadratic factors
\(1-\beta^2+\gamma^2\) and \(1+\beta^2-\gamma^2\) lie in
\([63/64,65/64]\subset(1/2,3/2)\), so none of \(F_1,F_2,F_3\) vanishes off
the coordinate axes \(\beta=0\) and \(\gamma=0\). Specifically:
\begin{enumerate}
\item \emph{Even wall (\(\alpha=0,\ \beta>0,\ \gamma>0,\ \beta+\gamma<r_0\)):}
Here \(F_1(0,\beta,\gamma)>0\), \(F_2(0,\beta,\gamma)>0\), and
\(F_3(0,\beta,\gamma)>0\). Thus the even wall is contained in the boundary
slice \(\alpha=0\) of the local chamber \(C_{\mathrm{thin}}\) of
Lemma~\ref{lem:n3-vol-thin} and contains the explicit sample \((0,3,2)/100\)
(where \(\|c\|_\infty=0.05<r_0\)). Its twelve triangular faces are
explicitly the thin/even-wall list~\eqref{eq:n3-triangles-thin}, whose
twelve origin-vanishing incidence determinants specialize at \(\alpha=0\) to
\(\pm 16 F_1(0,\beta,\gamma)=\pm 32\beta\gamma\),
\(\pm 16 F_2(0,\beta,\gamma)=\pm 16\beta(1-\beta^2+\gamma^2)\), and
\(\pm 16 F_3(0,\beta,\gamma)=\pm 16\gamma(1+\beta^2-\gamma^2)\).
\item \emph{Odd wall (\(\alpha=0,\ \beta>0,\ \gamma<0,\ \beta-\gamma<r_0\)):}
Here \(F_1(0,\beta,\gamma)<0\), \(F_2(0,\beta,\gamma)>0\), and
\(F_3(0,\beta,\gamma)<0\). The coordinate reflection
\(g_2=\operatorname{diag}(1,-1,1)\in B_3\) swaps the vertex pairs
\((v^0,v^2)\), \((v^1,v^3)\), \((v^4,v^6)\), \((v^5,v^7)\) and maps
\((\alpha,\beta,\gamma)=(0,\beta,\gamma)\) to \((0,\beta,-\gamma)\) in the
even wall. Consequently, the convex hull \(P(0,\beta,\gamma)\) on the odd
wall (which contains the explicit sample \((0,3,-2)/100\) with
\(\|c\|_\infty=0.05<r_0\)) has the twelve triangular faces obtained by
applying \(g_2\) to~\eqref{eq:n3-triangles-thin}:
\begin{equation}
\label{eq:n3-triangles-odd-wall}
\begin{aligned}
&\{v^2,v^3,v^1\},&&\{v^2,v^3,v^7\},&&\{v^2,v^0,v^1\},&&\{v^2,v^0,v^6\},\\
&\{v^2,v^6,v^7\},&&\{v^3,v^1,v^7\},&&\{v^0,v^1,v^5\},&&\{v^0,v^6,v^4\},\\
&\{v^0,v^4,v^5\},&&\{v^1,v^7,v^5\},&&\{v^6,v^7,v^5\},&&\{v^6,v^4,v^5\},
\end{aligned}
\end{equation}
whose twelve origin-vanishing incidence determinants are again
\(\pm 32\beta\gamma\neq 0\), \(\pm 16\beta(1-\beta^2+\gamma^2)\neq 0\), and
\(\pm 16\gamma(1+\beta^2-\gamma^2)\neq 0\), while the remaining forty-eight
evaluate to \(\pm 8\) at the origin and stay strictly nonzero on
\(|\beta|+|\gamma|<r_0\).
\end{enumerate}
\end{lemma}

\begin{proof}
Substituting \(\alpha=0\) into~\eqref{eq:n3-thin-walls} directly produces
the factorizations~\eqref{eq:n3-wall-factors}. For
\(0<|\beta|+|\gamma|<r_0=\frac18\) one has \(|\beta^2-\gamma^2|<1/64\), so
\(1\pm(\beta^2-\gamma^2)>63/64>1/2>0\). In item~(1), \(\beta>0\) and
\(\gamma>0\) give \(F_1,F_2,F_3>0\), so the slice \(\alpha=0\) lies in
\(C_{\mathrm{thin}}\) by definition, and Lemma~\ref{lem:n3-vol-thin} gives
the face list~\eqref{eq:n3-triangles-thin} (certified at the sample
\((0,3,2)/100\) with \(\|c\|_\infty=0.05<r_0\) in
Appendix~\ref{app:n3-certificate}). In item~(2), for \(q\in Q_3\) one has
\((g_2 q)_1(g_2 q)_3=q_1 q_3\) and \((g_2 q)_2(g_2 q)_3=-q_2 q_3\), so
\(g_2\bigl(P(0,\beta,\gamma)\bigr)=P(0,\beta,-\gamma)\) with \(-\gamma>0\).
Since \(g_2\) is an invertible linear isometry, it maps the face lattice of
the even wall to the twelve triangles~\eqref{eq:n3-triangles-odd-wall} (also
certified at \((0,3,-2)/100\) with \(\|c\|_\infty=0.05<r_0\) in
Appendix~\ref{app:n3-certificate}).
\end{proof}

\begin{lemma}[Volume polynomial on the even wall \(\alpha=0\)]
\label{lem:n3-vol-wall}
For \(r_0=\frac18\) as in Lemma~\ref{lem:n3-wall-clear}, if
\(\alpha=0\), \(\beta>0\), \(\gamma>0\) and \(\beta+\gamma<r_0\),
then summing the twelve origin-tetrahedra over the face
list~\eqref{eq:n3-triangles-thin} gives
\begin{equation}
\label{eq:n3-vol-wall}
\operatorname{vol}\bigl(P(c)\bigr)
=
8+\frac83(\beta+\gamma)
-8(\beta^2+\gamma^2)
+\frac{16}3\beta\gamma
-\frac83(\beta^3+\gamma^3)
+\frac83(\beta^2\gamma+\beta\gamma^2).
\end{equation}
The linear part is \(\frac83(\beta+\gamma)=\frac83\|c\|_\infty\).
\end{lemma}

\begin{proof}
On \(\alpha=0\), the vertices are
\(v^q=(1+\beta q_1q_3+\gamma q_2q_3)q\). Evaluating the twelve tetrahedral
volumes \(\frac16|\det(v^i,v^j,v^k)|\) on~\eqref{eq:n3-triangles-thin} (which
is equivalent to setting \(\alpha=0\) in~\eqref{eq:n3-vol-thin-poly} of
Lemma~\ref{lem:n3-vol-thin}) yields~\eqref{eq:n3-vol-wall}.
\end{proof}

\begin{lemma}[Volume polynomial on the odd wall \(\alpha=0\)]
\label{lem:n3-vol-wall-odd}
For \(r_0=\frac18\) as in Lemma~\ref{lem:n3-wall-clear}, if
\(\alpha=0\), \(\beta>0\), \(\gamma<0\) and \(\beta-\gamma<r_0\),
then summing the twelve origin-tetrahedra over the face
list~\eqref{eq:n3-triangles-odd-wall} gives
\[
\operatorname{vol}\bigl(P(c)\bigr)
=
8+\tfrac83(\beta-\gamma)
-8(\beta^2+\gamma^2)
-\tfrac{16}3\beta\gamma
-\tfrac83(\beta^3-\gamma^3)
-\tfrac83(\beta^2\gamma-\beta\gamma^2).
\]
The linear part is \(\frac83(\beta-\gamma)=\frac83\|c\|_\infty\).
\end{lemma}

\begin{proof}
By Lemma~\ref{lem:n3-wall-clear}(2), the orthogonal reflection \(g_2\in B_3\)
satisfies \(|\det g_2|=1\) and \(g_2\bigl(P(0,\beta,\gamma)\bigr)=P(0,\beta,-\gamma)\)
with \(-\gamma>0\), mapping each origin-tetrahedron of~\eqref{eq:n3-triangles-odd-wall}
isometrically onto the corresponding origin-tetrahedron
of~\eqref{eq:n3-triangles-thin} at \((0,\beta,-\gamma)\). Substituting
\(\gamma\mapsto -\gamma\) into~\eqref{eq:n3-vol-wall} gives the stated polynomial.
\end{proof}

\subsection{Directional jets and the Mahler increment}
\label{sec:n3-directional}

\begin{proposition}[Pairwise line in dimension three]
\label{prop:n3-pair-volume}
Let \(c_t(q)=t\,q_1q_2\) and \(P_t=P(c_t)\). For \(|t|<1\), \(\operatorname{vol}(P_t) = 8+\frac83|t|-8t^2-\frac83|t|^3\).
\end{proposition}

\begin{proof}
The origin-tetrahedron expansion along the axes is exact and extends through \(t=0\) by evenness in the absolute value.
\end{proof}

\begin{lemma}[Raywise stability of incidence]
\label{lem:n3-ray-stable}
Let \(G_1,\dots,G_{60}\in\mathbb Z[\alpha,\beta,\gamma]\) be the
incidence determinants of Lemma~\ref{lem:n3-sixty}, and let
\(u\in\mathbf1^\perp\) have pair coordinates
\((\alpha_0,\beta_0,\gamma_0)\). Write \(G_j(t):=G_j(tu)\).
\begin{enumerate}
\item
  If \(G_j\) does not vanish at the origin, then
  \(G_j(t)=G_j(0)+O(t)\) with \(G_j(0)\neq0\), so \(G_j(t)\neq0\)
  for all sufficiently small \(t\).
\item
  If \(G_j\) vanishes at the origin, then \(G_j\) is an integer
  multiple of one of \(F_1,F_2,F_3\). On a ray with
  \(\alpha_0\neq0\) one has \(F_1(tu)=-\alpha_0 t+O(t^2)\neq0\) for
  small \(t>0\), and cyclically for \(\beta_0,\gamma_0\).
\end{enumerate}
In particular, if \(\alpha_0\beta_0\gamma_0\neq0\), there exists
\(t_u>0\) such that no incidence determinant vanishes on
\((0,t_u)\). The face lattice is therefore constant along that
punctured ray. The same dichotomy holds for the odd face list of
Lemma~\ref{lem:n3-vol-odd}, after the \(B_3\)-flip of
Lemma~\ref{lem:n3-odd-rays}.
\end{lemma}

\begin{proof}
Item~(1) is Taylor expansion at the origin in
\(\mathbb R^3\). Item~(2) is Lemma~\ref{lem:n3-sixty} together with
the leading terms of~\eqref{eq:n3-thin-walls}. A univariate
polynomial \(G_j(t)\) that is not identically zero has finitely many
roots; the leading-term statements exclude an identical vanishing
on the generic ray.
\end{proof}

\begin{lemma}[Classification of rays in \(\mathbf1^\perp\)]
\label{lem:n3-ray-types}
Let \(u\in\mathbf1^\perp\setminus\{0\}\) have pair coordinates
\((\alpha_0,\beta_0,\gamma_0)\), and set \(c_t=tu\). Exactly one of
the following holds.
\begin{enumerate}
\item
  All three coordinates are nonzero. Then \(u\) lies in an open
  octant. If \(\alpha_0\beta_0\gamma_0>0\) the octant is even and
  Lemma~\ref{lem:n3-ray-stable} together with
  Lemma~\ref{lem:n3-gen-rays} places \(c_t\) in the \(B_3\)-image of
  \(C_{\mathrm{gen}}\) for small \(t>0\). If
  \(\alpha_0\beta_0\gamma_0<0\) the octant is odd and
  Lemma~\ref{lem:n3-odd-rays} places \(c_t\) in the \(B_3\)-image of
  \(C_-\).
\item
  Exactly one coordinate vanishes. Then \(c_t\) remains on a
  two-character wall. The wall is even or odd according to the sign
  of the product of the two nonzero coordinates
  (Lemmas~\ref{lem:n3-vol-wall} and~\ref{lem:n3-vol-wall-odd}, and
  coordinate permutations).
\item
  Exactly two coordinates vanish. Then \(c_t\) is a pairwise axis
  (Proposition~\ref{prop:n3-pair-volume} and its permutations).
\end{enumerate}
In the positive octant, no fixed ray with
\(\alpha_0,\beta_0,\gamma_0>0\) lies in the thin chamber
\(C_{\mathrm{thin}}=\{F_1>0\}\) (or in \(\{F_2<0\}\) or \(\{F_3<0\}\)) for all
sufficiently small \(t>0\), since \(F_1(tu)=-\alpha_0 t+O(t^2)<0\) (and
\(F_2(tu)=\beta_0 t+O(t^2)>0\), \(F_3(tu)=\gamma_0 t+O(t^2)>0\)). The
\(B_3\)-images in the other even octants have the analogous property for
their corresponding thin chambers: thin chambers meet every neighbourhood
of the origin, but contain no ray through the origin.
\end{lemma}

\begin{proof}
The three alternatives partition the nonzero points of pair space.
Case~(1) is Lemma~\ref{lem:n3-ray-stable} plus the even/odd ray
lemmas. Cases~(2) and~(3) are the walls and axes already treated.
The thin-chamber claim is the leading term of~\eqref{eq:n3-thin-walls}.
\end{proof}

\begin{proposition}[Directional folding in dimension three]
\label{prop:folding-n3}
For every fixed direction \(u \in \mathbf1^\perp \setminus \{0\}\),
\begin{equation}
\label{eq:folding-n3-ray}
\lim_{t\downarrow 0} \frac{\operatorname{vol}\bigl(P(tu)\bigr)-8}{t}
=
\frac83 \|u\|_\infty
\ge
\frac{1}{6} \lim_{t\downarrow 0}\frac{\|\mathcal R_3(tu)\|_\infty}{t}.
\end{equation}
\end{proposition}

\begin{proof}
Lemma~\ref{lem:n3-ray-types} places every ray in one treated
chamber. Even generic:
Lemmas~\ref{lem:n3-gen-rays} and~\ref{lem:n3-vol-positive}.
Odd generic: Lemma~\ref{lem:n3-odd-rays}.
Even wall: Lemma~\ref{lem:n3-vol-wall} and coordinate permutations.
Odd wall: Lemma~\ref{lem:n3-vol-wall-odd} and coordinate permutations.
Axis: Proposition~\ref{prop:n3-pair-volume}.
In every case the linear increment is \(\frac83\|u\|_\infty\).
The comparison with \(\|\mathcal R_3\|_\infty\) is~\eqref{eq:n3-R-linear}.
\end{proof}

\begin{proposition}[First Mahler jet along rays in dimension three]
\label{prop:n3-mahler-ray}
Let \(u\in\mathbf1^\perp\setminus\{0\}\). Then
\begin{equation}
\label{eq:n3-mahler-ray}
\mathcal V_M\bigl(P(tu)\bigr)
-
\mathcal V_M(\mathcal C_3)
=
\frac{32}{9}\|u\|_\infty\,t
+
O(t^2)
\qquad(t\downarrow0).
\end{equation}
In particular the geometric Mahler product has no bilateral derivative
along any transverse line through the origin in \(\mathbf1^\perp\).
The one-sided derivatives exist and satisfy
\(D_+M(0;u)=D_+M(0;-u)=\frac{32}{9}\|u\|_\infty>0\).
In particular the left derivative in direction \(u\) equals
\(-D_+M(0;-u)=-\frac{32}{9}\|u\|_\infty\), so the bilateral
derivative does not exist.
\end{proposition}

\begin{proof}
Proposition~\ref{prop:folding-n3} gives
\(\operatorname{vol}(P(tu))=8+\frac83\|u\|_\infty t+O(t^2)\).
Proposition~\ref{prop:polar-quadratic}, for \(n=3\) and
\(tu\in\mathbf1^\perp\), gives
\(\operatorname{vol}(P(tu)^\circ)=\frac43+O(t^2)\).
The product expands to
\(\frac{32}3+\frac{32}9\|u\|_\infty t+O(t^2)\).
The cube value is \(\frac{32}3\). The expansion is one-sided in
\(t\downarrow0\); it does not yield a bilateral derivative at the
origin.
\end{proof}
The Parker family of Section~\ref{sec:parker-warning} is not in
\(\mathbf1^\perp\) and is not used here.

\subsection{Uniform folding on a ball in dimension three}
\label{sec:n3-uniform}

The directional jets of Section~\ref{sec:n3-directional} do not by
themselves give a lower bound on a neighbourhood of the origin: thin
chambers in the even octants approach the origin tangent to the
coordinate planes and contain no ray through the origin. The volume
polynomials already computed on the finitely many chambers that meet a
small ball close the gap. Throughout this subsection
\(\|c\|_\infty=|\alpha|+|\beta|+|\gamma|\), and~\eqref{eq:n3-R-linear}
reads
\[
\|\mathcal R_3(c)\|_\infty
=
16\max\bigl(|\alpha|,|\beta|,|\gamma|\bigr)
+O\bigl(\|c\|_\infty^2\bigr).
\]

\begin{lemma}[Chamber partition of a small ball]
\label{lem:n3-chamber-partition}
With \(r_*:=r_0=\frac18=0.125\), the ball
\(\{c\in\mathbf1^\perp:\|c\|_\infty<r_*\}\) is covered by the closures
of the following pieces, together with their images under \(B_3\):
\begin{enumerate}
\item
  the even generic chamber \(C_{\mathrm{gen}}\) and the odd octant
  chamber \(C_-\);
\item
  the three even thin chambers \(\{F_1>0\}\), \(\{F_2<0\}\),
  \(\{F_3<0\}\) in \(\{\alpha,\beta,\gamma>0,\ \|c\|_\infty<r_*\}\);
\item
  the even and odd two-character walls
  \(\{\alpha=0,\ \beta\gamma>0,\ \|c\|_\infty<r_*\}\) and
  \(\{\alpha=0,\ \beta\gamma<0,\ \|c\|_\infty<r_*\}\), and their
  coordinate permutations;
\item
  the three pairwise axes.
\end{enumerate}
On this ball at most one thin inequality can hold in the positive
octant: if two of \(\alpha\le C\beta\gamma\), \(\beta\le C\alpha\gamma\),
\(\gamma\le C\alpha\beta\) hold simultaneously with \(C=4\) and
\(\alpha,\beta,\gamma>0\), then
\(\max(|\alpha|,|\beta|,|\gamma|)\ge 1/C=\frac14>r_*\). The forty-eight
incidence determinants of Lemmas~\ref{lem:n3-sixty} and~\ref{lem:n3-vol-odd}
that do not vanish at the origin have no zero in the ball.
\end{lemma}

\begin{proof}
The only incidence factors vanishing at the origin for the positive list
are \(F_1,F_2,F_3\), by Lemma~\ref{lem:n3-sixty}. On
\(\|c\|_\infty<r_*=\frac18\) one has \(F_1=-\alpha D_1+2\beta\gamma\),
\(F_2=\beta D_2-2\alpha\gamma\), \(F_3=\gamma D_3-2\alpha\beta\) with
\(D_i\in[61/64,67/64]\subset(1/2,2)\). Thus the open positive octant
inside \(\|c\|_\infty<r_*\) splits into the generic
region~\eqref{eq:Cgen-graphs} and the three disjoint thin regions in which
exactly one coordinate satisfies \(0<\alpha<2\beta\gamma/D_1<4\beta\gamma\)
(or cyclically for \(\beta,\gamma\)): if both \(\alpha<4\beta\gamma\) and
\(\beta<4\alpha\gamma\) held with \(\alpha,\beta,\gamma>0\), multiplying
would give \(1<16\gamma^2\), forcing \(\gamma>1/4>r_*\), impossible on
\(\|c\|_\infty<r_*=\frac18\).

In the open negative octant \(\{\alpha,\beta,\gamma<0\}\),
Lemma~\ref{lem:n3-vol-odd} shows that the twelve vanishing incidence
determinants of the odd face list~\eqref{eq:n3-triangles-odd} are again
\(16(\pm F_k)\), and that \(F_1>0\), \(F_2<0\), \(F_3<0\) hold
strictly throughout \(C_-\). Thus \(C_-\) is not subdivided by any thin
wall: the odd face list~\eqref{eq:n3-triangles-odd} persists on the
whole open negative octant inside the ball.

On the coordinate planes the partition is given by
Lemmas~\ref{lem:n3-wall-clear}--\ref{lem:n3-vol-wall-odd}.
The group \(B_3\) permutes pair coordinates and acts transitively on the
four even octants \(\{\alpha\beta\gamma>0\}\) and on the four odd octants
\(\{\alpha\beta\gamma<0\}\). By Lemma~\ref{lem:n3-sixty}, the forty-eight
remaining determinants have no zero on \(\|c\|_\infty\le r_*=\frac18\).
\end{proof}

\begin{lemma}[Linear increment on all chambers]
\label{lem:n3-linear-pieces}
There exists \(C_*<\infty\) such that if \(\|c\|_\infty<r_*=\frac18\), then
on every chamber listed in Lemma~\ref{lem:n3-chamber-partition},
\begin{equation}
\label{eq:n3-chamber-linear}
\operatorname{vol}\bigl(P(c)\bigr)-8
=
\frac83\|c\|_\infty
+R(c),
\qquad
|R(c)|\le C_*\|c\|_\infty^2.
\end{equation}
\end{lemma}

\begin{proof}
Each volume polynomial of Section~\ref{sec:folding-n3} is of degree
at most three. On the ball \(\|c\|_\infty<r_*\) the quadratic and
cubic terms are therefore \(O(\|c\|_\infty^2)\), uniformly on each
piece. It remains to match the linear part with \(\frac83\|c\|_\infty\).

On \(C_{\mathrm{gen}}\) the linear part is
\(\frac83(\alpha+\beta+\gamma)=\frac83\|c\|_\infty\) by
Lemma~\ref{lem:n3-vol-positive}. On the entire negative octant chamber
\(C_-\) it is
\(-\frac83(\alpha+\beta+\gamma)=\frac83\|c\|_\infty\) by
Lemma~\ref{lem:n3-vol-odd}. On the even wall \(\alpha=0\),
\(\beta>0\), \(\gamma>0\) it is \(\frac83(\beta+\gamma)=\frac83\|c\|_\infty\)
by Lemma~\ref{lem:n3-vol-wall}, and on the odd wall \(\alpha=0\),
\(\beta>0\), \(\gamma<0\) it is \(\frac83(\beta-\gamma)=\frac83\|c\|_\infty\)
by Lemma~\ref{lem:n3-vol-wall-odd}. On a pairwise axis,
Proposition~\ref{prop:n3-pair-volume} gives
\(\frac83|t|=\frac83\|c\|_\infty\). Signed permutations in \(B_3\)
preserve both volume and \(\|c\|_\infty\), transferring the identities
to all four even octants and all four odd octants.

On the even thin chamber \(F_1>0\) of Lemma~\ref{lem:n3-vol-thin} the
linear part is \(\frac83(-\alpha+\beta+\gamma)\). The thin inequality
together with~\eqref{eq:Cgen-graphs} yields
\(0\le\alpha\le 2\beta\gamma/D_1\le 4\beta\gamma\le\|c\|_\infty^2\), hence
\(\alpha=O(\|c\|_\infty^2)\). Therefore
\[
-\alpha+\beta+\gamma
=
(\alpha+\beta+\gamma)-2\alpha
=
\|c\|_\infty+O\bigl(\|c\|_\infty^2\bigr),
\]
and~\eqref{eq:n3-chamber-linear} holds on that chamber. The thin
chambers \(F_2<0\) and \(F_3<0\) in the positive octant (and their
\(B_3\)-images in the other three even octants) are images of
\(\{F_1>0\}\) under \(B_3\), so the same estimate applies with
\(\beta=O(\|c\|_\infty^2)\) or \(\gamma=O(\|c\|_\infty^2)\).
\end{proof}

\begin{proposition}[Uniform folding in dimension three]
\label{prop:folding-n3-ball}
There exist constants \(a_3,r_3,C_3>0\) such that if
\(c\in\mathbf1^\perp\) and \(\|c\|_\infty<r_3\), then
\begin{equation}
\label{eq:folding-n3-ball}
\operatorname{vol}\bigl(P(c)\bigr)-8
\ge
\frac83\|c\|_\infty
-C_3\|c\|_\infty^2
\ge
a_3\|\mathcal R_3(c)\|_\infty
-C_3\|c\|_\infty^2.
\end{equation}
In particular Target~B holds on the isovertex stratum in dimension
three, with a linear increment.
\end{proposition}

\begin{proof}
Take \(r_3\le r_*=\frac18\) small enough that Lemmas~\ref{lem:n3-chamber-partition}
and~\ref{lem:n3-linear-pieces} apply and that every generator remains a vertex
(Lemma~\ref{lem:vtx-H}). On every chamber of the partition of
Lemma~\ref{lem:n3-chamber-partition}, Lemma~\ref{lem:n3-linear-pieces}
gives~\eqref{eq:n3-chamber-linear}. Since \(\operatorname{vol}(P(c))\)
and \(\|c\|_\infty\) are continuous, the bound extends to the closures
of those pieces and hence to the entire ball \(\|c\|_\infty<r_3\).
This proves the first inequality in~\eqref{eq:folding-n3-ball}.

For the comparison with coplanarity,~\eqref{eq:n3-R-linear} gives
\(\|\mathcal R_3(c)\|_\infty\le 16\max(|\alpha|,|\beta|,|\gamma|)
+C\|c\|_\infty^2\) on the ball. Since
\(\max(|\alpha|,|\beta|,|\gamma|)\le\|c\|_\infty\),
\[
\frac83\|c\|_\infty
\ge
\frac83\max\bigl(|\alpha|,|\beta|,|\gamma|\bigr)
\ge
\frac1{6}\|\mathcal R_3(c)\|_\infty
-C'\|c\|_\infty^2.
\]
Absorbing constants yields the second inequality with \(a_3=1/6\),
after shrinking \(r_3\) if needed so that the quadratic error from
\eqref{eq:n3-R-linear} is absorbed in \(C_3\).
\end{proof}

\begin{proposition}[Uniform Mahler increment in dimension three]
\label{prop:n3-mahler-ball}
There exist \(a_3',r_3,B_3>0\) such that if \(c\in\mathbf1^\perp\) and
\(\|c\|_\infty<r_3\), then
\begin{equation}
\label{eq:n3-mahler-ball}
\mathcal V_M\bigl(P(c)\bigr)-\mathcal V_M(\mathcal C_3)
\ge
a_3'\,d_{GL}\bigl(P(c),\mathcal C_3\bigr)
-B_3\|c\|_\infty^2.
\end{equation}
The geometric Mahler product has no bilateral derivative at the origin
along any line in \(\mathbf1^\perp\).
\end{proposition}

\begin{proof}
Proposition~\ref{prop:polar-quadratic} gives
\(\operatorname{vol}(P(c)^\circ)=\operatorname{vol}(\mathcal O_3)+O(\|c\|_\infty^2)\).
The product identity~\eqref{eq:mahler-split} therefore reduces the
Mahler increment to a multiple of \(\operatorname{vol}(P(c))-8\), up
to \(O(\|c\|_\infty^2)\). Insert
Proposition~\ref{prop:folding-n3-ball} and
Theorem~\ref{thm:main-rigidity}. The one-sided directional derivatives
remain those of Proposition~\ref{prop:n3-mahler-ray}; a bilateral
derivative would require the increment to be odd, which it is not.
\end{proof}

\section{Conclusion and open problems}
\label{sec:conclusion}

On the isovertex stratum, for \(n\ge 3\), Target~A is
Theorem~\ref{thm:main-rigidity}. If \(A_n\) denotes the linearization
of \(\mathcal R_n\) at the origin, then
\[
\ker A_n=\operatorname{span}\{\mathbf1\},
\]
and, locally for \(c=t\mathbf1+h\) with \(h\in\mathbf1^\perp\),
\[
\|\mathcal R_n(c)\|_\infty
\asymp
\operatorname{dist}_\infty
  \bigl(c,\operatorname{span}\{\mathbf1\}\bigr)
=
\|h\|_\infty
\asymp
d_{GL}\bigl(P(c),\mathcal C_n\bigr).
\]
This is a quantitative rigidity statement for the labelled generators
on the rays through \(Q_n\), not a statement about the Lie algebra
\(\mathfrak{gl}(n)\). After normalizing out the radial homothety via
\(P(c)=(1+t)P(\varepsilon)\) with \(\varepsilon=h/(1+t)\in\mathbf1^\perp\),
a key step is the second-order comparison on \(\mathbf1^\perp\),
\[
\bigl|N^{\mathrm{GL}}(\varepsilon)-N(\varepsilon)\bigr|
\le C_n'\|\varepsilon\|_\infty^2,
\]
which shows that optimizing over small linear transformations does
not change the first-order transversal scale \(N(\varepsilon)\asymp\|\varepsilon\|_\infty\)
of the linearized vertex defect.

For the polar, write \(W(h)\) for its volume as a function of the
support parameters \(h=(h_q)_{q\in Q_n}\). Proposition~\ref{prop:polar-quadratic}
rests on the first-order expansion
\[
W(\mathbf1+\eta)
=
\operatorname{vol}(\mathcal O_n)
+\frac{1}{(n-1)!}\sum_{q\in Q_n}\eta_q
+O\bigl(\|\eta\|_\infty^2\bigr).
\]
For the radial family, \(h_q=(1+c_q)^{-1}\), so
\[
\eta_q=h_q-1=-c_q+O\bigl(\|c\|_\infty^2\bigr).
\]
Consequently, on \(\mathbf1^\perp\), where
\(\sum_{q\in Q_n}c_q=0\),
\[
\operatorname{vol}\bigl(P(c)^\circ\bigr)
=
\operatorname{vol}(\mathcal O_n)
+O\bigl(\|c\|_\infty^2\bigr).
\]
The uniform quadratic remainder is essential: it allows the linear
primal increment to determine the first-order Mahler increment.

Theorem~\ref{thm:primal-general} provides the corresponding primal
lower bound
\[
\operatorname{vol}\bigl(P(c)\bigr)-2^n
\ge
\frac{8}{3(n-1)}\|c\|_\infty
-C_n\|c\|_\infty^2
\]
on a sufficiently small neighbourhood in \(\mathbf1^\perp\).
Its proof uses the pocket functional associated with non-affine
restrictions to coordinate facets. The cone
\(\operatorname{conv}\bigl(\{0\}\cup K_F(c)\bigr)\) lies inside \(P(c)\) and
contributes a linear volume increment; the argument does not require
a selected simplex to be a face of the convex hull.

Combining this estimate with the polar expansion and
Theorem~\ref{thm:main-rigidity} gives
Theorem~\ref{thm:mahler-general}: for sufficiently small
\(c\in\mathbf1^\perp\) and every \(n\ge3\), there exists \(a_n'>0\)
such that
\[
\mathcal V_M\bigl(P(c)\bigr)-\mathcal V_M(\mathcal C_n)
\ge
a_n'\,d_{GL}\bigl(P(c),\mathcal C_n\bigr).
\]
This is a quantitative Mahler increment on the isovertex stratum
(and extends immediately to general \(c=t\mathbf1+h\) by scale
invariance of \(\mathcal V_M\) and \(d_{GL}\)).
It is not asserted as a local-minimality theorem for the full
Banach--Mazur neighbourhood of the cube among all origin-symmetric
convex bodies.

In dimension three, the first-order primal coefficient is explicit.
Propositions~\ref{prop:folding-n3-ball} and~\ref{prop:n3-mahler-ball}
give
\[
\operatorname{vol}\bigl(P(c)\bigr)-8
\ge
\frac83\|c\|_\infty-C_3\|c\|_\infty^2
\]
and a uniform linear Mahler lower bound on a sufficiently small ball
in \(\mathbf1^\perp\). The coefficient \(8/3\) is explicit; the final
constants in the uniform Mahler estimate are not claimed to be sharp
or numerically optimal.

Along every fixed nonzero direction \(u\in\mathbf1^\perp\), the
directional expansions of Propositions~\ref{prop:folding-n3}
and~\ref{prop:n3-mahler-ray} give
\[
\operatorname{vol}\bigl(P(tu)\bigr)-8
=
\frac83\|u\|_\infty\,t+O(t^2),
\qquad t\downarrow0,
\]
and
\[
\mathcal V_M\bigl(P(tu)\bigr)-\mathcal V_M(\mathcal C_3)
=
\frac{32}{9}\|u\|_\infty\,t+O(t^2),
\qquad t\downarrow0.
\]
The geometric Mahler product therefore has no bilateral derivative
at the origin along any line in \(\mathbf1^\perp\). Thin chambers
meet every sufficiently small neighbourhood in the even octants but
contain no ray through the origin; their contribution is controlled
because the thin coordinate is \(O(\|c\|_\infty^2)\).

The results leave several natural questions.

\begin{enumerate}
\item Determine exact chamberwise volume polynomials in the radial
parameter for further Boolean-orbit representatives in dimension
four, including the transitions of the face lattice across incidence
walls.

\item Extend the radial/isovertex coplanarity analysis beyond the cube
generators indexed by \(Q_n\), both to general Hanner polytopes on their
vertex-ray strata and to transverse perturbations of the vertices. A
broader problem is to identify an appropriate analogue of the
coplanarity defect and determine whether it yields a quantitative
combinatorial description of the Mahler increment in a full
Banach--Mazur neighbourhood of the cube among origin-symmetric convex
bodies. Such an extension is not needed for Targets~A and~B as
established here.

\item Determine sharp values of the constants \(a_n\) and \(a_n'\)
for \(n\ge4\). The pocket estimates give the explicit mixed-difference
bounds
\[
\kappa_{n-1}=\frac{8}{3(n-1)(n-2)},
\qquad
\kappa_n''=\frac{8}{3(n-1)}.
\]
Consequently, every \(a_n<\frac{8}{3(n-1)}\) is admissible on a
sufficiently small neighbourhood in \(\mathbf1^\perp\). Whether this
threshold is optimal on \(\mathbf1^\perp\) remains open for \(n\ge4\).
\end{enumerate}
\appendix

\section{Balanced Boolean directions in dimension four}
\label{app:n4-exact}

This appendix classifies the seventy balanced four-subsets of the eight
antipodal pairs of \(\mathcal C_4\) and converts each orbit into a
first-order primal lower bound. The input is the pocket functional of
Lemma~\ref{lem:pocket-Phi} and the cone estimate of
Lemma~\ref{lem:volume-from-Phi}. No boundary triangulation of
\(P(c)\) is used, and no exact volume polynomial through order two is
claimed. The section is not a folding theorem on the whole of
\(\mathbf1^\perp\subset E_4\).

\subsection{Classification}
\label{app:boolean-audit}

Label the eight antipodal pairs of \(\mathcal C_4\) by the vertices of
the three-cube \(Q=\{\pm1\}^3\), via the representatives
\((1,q)\in\{\pm1\}^4\). A balanced Boolean direction is a \(4\)-subset
\(S\subset Q\), equivalently a sign vector
\(s\in\{\pm1\}^8\cap\mathbf1^\perp\). There are \(\binom{8}{4}=70\)
such subsets. For \(t>0\) set
\[
c_t
=
\frac{t}{4\sqrt2}\,s,
\qquad
P_t
=
P(c_t).
\]
The Euclidean normalisation \(R=2\mathbf1+t\,s/\sqrt8\) is equivalent,
since vertices of \(\mathcal C_4\) have length \(2\). On the slice
\(\mathbf1^\perp\), Lemma~\ref{lem:volume-from-Phi} yields
\begin{equation}
\label{eq:app-volume-from-Phi}
\operatorname{vol}(P_t)-16
\ge
\frac{t}{4\sqrt2}\sum_F\Phi(s|_F)
-C_4 t^2,
\end{equation}
the sum running over the eight coordinate facets of \(\mathcal C_4\).
Every orbit bound therefore reduces to local pocket integrals on
\([-1,1]^3\) and to a restriction count.

\paragraph{Six orbits under \(\operatorname{Aut}(Q)\).}
The group \(\operatorname{Aut}(Q)\cong B_3\) has order \(48\). Each
element extends to an orthogonal transformation of \(\mathbb R^4\)
fixing the first axis, hence preserves four-volume. The seventy
subsets partition into six orbits:
\begin{center}
\begin{tabular}{clc}
\# & type of \(S\subset Q\) & size of class \\
\hline
6 & a \(2\)-face of \(Q\) & \(6\) \\
6 & a pair of opposite parallel edges of \(Q\) & \(6\) \\
24 & a path of length \(3\) in \(Q\) & \(24\) \\
8 & a vertex star: \(v\) and its three neighbours & \(8\) \\
24 & a vertex, two neighbours, and the antipode \(-v\) & \(24\) \\
2 & an inscribed tetrahedron of \(Q\) & \(2\).
\end{tabular}
\end{center}
The counts are elementary: \(70=6+6+24+8+24+2\). Under the full group
\(B_4\) of \(\mathcal C_4\), the face orbit and the opposite-edge orbit
coincide: both are Walsh characters of weight two.

\subsection{Local pocket integrals on \([-1,1]^3\)}
\label{app:local-Phi}

Let \(F=[-1,1]^3\) and let \(a:Q_3\to\{\pm1\}\) be balanced, so
\(\sum a_q=0\) and \(\Phi(a)=\int_F U_a\). The six types above are
exactly the possible restrictions of a global balanced labelling to a
coordinate facet.

\begin{lemma}[Face restriction]
\label{lem:app-Phi-face}
If \(a(q)=q_1\) on \(Q_3\), then \(a\) is affine,
\(U_a(x)=x_1\) for all \(x=(x_1,x_2,x_3)\in F=[-1,1]^3\), and
\(\Phi(a)=0\).
\end{lemma}

\begin{proof}
Since \(a(q)=q_1\) is the restriction to \(Q_3\) of the coordinate
functional \(x\mapsto x_1\), it is affine on the vertex set. For every
\(x=(x_1,x_2,x_3)\in F=[-1,1]^3\) and every convex weight
\(\mu\in\mathcal P_x\), the barycenter constraint
\(\sum_{q\in Q_3}\mu_q q=x\) gives
\[
\sum_{q\in Q_3}\mu_q a(q)
=
\sum_{q\in Q_3}\mu_q q_1
=
x_1.
\]
Since \(\mathcal P_x\) is nonempty, taking the maximum over
\(\mu\in\mathcal P_x\) in~\eqref{eq:concave-envelope} yields
\(U_a(x)=x_1\). By oddness in \(x_1\),
\[
\int_F U_a(x)\,dx
=
\int_{[-1,1]^3}x_1\,dx
=
0,
\qquad\text{and}\qquad
\sum_{q\in Q_3}a(q)
=
\sum_{q\in Q_3}q_1
=
0,
\]
since four vertices of \(Q_3\) have \(q_1=1\) and four have \(q_1=-1\).
Substituting both into~\eqref{eq:Phi} gives \(\Phi(a)=0\).
\end{proof}

\begin{lemma}[Opposite-edge restriction]
\label{lem:app-Phi-edges}
If \(a(q)=q_2q_3\), then \(U_a(x,y,z)=1-\lvert y-z\rvert\) and
\(\Phi(a)=\frac83\).
\end{lemma}

\begin{proof}
The four vertices with \(yz=1\) lie at height \(1\), the four with
\(yz=-1\) at height \(-1\). The function \(1-\lvert y-z\rvert\) is
concave, matches \(a\) at every vertex, and is therefore the least
concave majorant. It does not depend on \(x\). The identity
\(\int_{[-1,1]^2}\lvert y-z\rvert\,dy\,dz=\frac83\) is elementary, so
\[
\int_F U_a
=
2\Bigl(4-\frac83\Bigr)
=
\frac83.
\]
\end{proof}

\begin{lemma}[Parity restriction]
\label{lem:app-Phi-parity}
If \(a(q)=q_1q_2q_3\), then \(\Phi(a)=\frac{16}3\).
\end{lemma}

\begin{proof}
The plus set is the inscribed regular tetrahedron \(T_+\) of even
parity; the minus set is the opposite tetrahedron \(T_-\). Constant
\(1\) is a concave majorant, so \(U_a=1\) on \(T_+\). The cube is
the union of \(T_+\) with four corner tetrahedra, one at each vertex
of \(T_-\). The regular tetrahedron inscribed in \([-1,1]^3\) has
volume \(\frac83\), hence each corner tetrahedron has volume
\(\frac43\).

On a corner tetrahedron \(U_a\) is affine, with value \(1\) on the
interior face and \(-1\) at the outer vertex. The average of an affine
function on a simplex is the average of its vertex values, here
\(\frac12\). Therefore
\[
\int_{T_+}U_a=\frac83,
\qquad
\int_{\mathrm{corner}}U_a=\frac12\cdot\frac43=\frac23,
\]
and \(\int_F U_a=\frac83+4\cdot\frac23=\frac{16}3\).
\end{proof}

\begin{lemma}[Mixed restriction]
\label{lem:app-Phi-mixed}
Let
\[
S_{\mathrm{mix}}
=
\bigl\{(1,1,1),\,(1,-1,1),\,(-1,1,1),\,(-1,-1,-1)\bigr\}
\]
and set \(a=\mathbf1_{S_{\mathrm{mix}}}-\mathbf1_{Q_3\setminus S_{\mathrm{mix}}}\).
Then
\[
U_a(x,y,z)
=
\min\bigl(1,\ 2+x+y-z,\ 1-x+z,\ 1-y+z\bigr),
\]
and \(\Phi(a)=\frac{10}3\).
\end{lemma}

\begin{proof}
The right-hand side is concave and matches \(a\) at all eight
vertices of \(Q_3\), hence dominates \(U_a\). The four regions on which
a single argument attains the minimum partition \(F=[-1,1]^3\):
\begin{align*}
R_0&=\{x\le z,\ y\le z,\ x+y-z\ge -1\}, & U_a&=1,\\
R_1&=\{x+y-z\le -1\}, & U_a&=2+x+y-z,\\
R_2&=\{x\ge\max(y,z)\}, & U_a&=1-x+z,\\
R_3&=\{y\ge\max(x,z)\}, & U_a&=1-y+z,
\end{align*}
with volumes \(\frac43\), \(\frac43\), \(\frac83\), \(\frac83\),
respectively. On each piece, the active affine function is interpolated
by the vertices of \(Q_3\) belonging to that piece, so equality with the
least concave majorant \(U_a\) holds throughout \(F\). On the tetrahedron
\(R_0=\operatorname{conv}(S_{\mathrm{mix}})\) one has \(U_a=1\) and the
integral is \(\frac43\). On the corner tetrahedron
\(R_1=\operatorname{conv}\{(-1,-1,1),(1,-1,1),(-1,1,1),(-1,-1,-1)\}\),
\(U_a\) is affine with vertex values \(-1,1,1,1\) and average \(\frac12\),
hence integral \(\frac23\). The region \(R_2\) splits into the two order
tetrahedra \(\{y\le z\le x\}\) and \(\{z\le y\le x\}\) of volume \(\frac43\)
each, on which \(1-x+z\) has vertex averages \(\frac12\) and \(0\),
contributing \(\frac23+0=\frac23\); by \(x\leftrightarrow y\) symmetry,
\(R_3\) also contributes \(\frac23\). The integrals sum to
\(\frac43+\frac23+\frac23+\frac23=\frac{10}3\).
\end{proof}

\begin{lemma}[Star and path restrictions]
\label{lem:app-Phi-star-path}
A vertex-star colouring and a path-of-length-three colouring both
satisfy \(\Phi(a)=\frac83\).
\end{lemma}

\begin{proof}
For the star
\(S_\ast=\{(1,1,1),(1,1,-1),(1,-1,1),(-1,1,1)\}\) the least concave
majorant is
\[
U_a(x,y,z)
=
\min\Bigl(1,\ \tfrac12(1+x+y+z),\ 1+x+y,\ 1+x+z,\ 1+y+z\Bigr).
\]
The right-hand side is concave and matches \(a\) at the eight vertices.
It is therefore at least \(U_a\). Equality holds because each of the five
affine arguments is interpolated by the vertices of one of the five
tetrahedra into which the cube decomposes: the corner
\(H=\operatorname{conv}\{(1,1,1),(1,1,-1),(1,-1,1),(-1,1,1)\}\) of volume
\(\frac43\), on which \(U_a=1\); the central tetrahedron
\(T_0=\operatorname{conv}\{(-1,-1,-1),(1,1,-1),(1,-1,1),(-1,1,1)\}\) of
volume \(\frac83\), on which \(U_a=\frac12(1+x+y+z)\) and the vertex
average is \(\frac12\); and three lateral tetrahedra
\(T_z=\operatorname{conv}\{(-1,-1,1),(-1,-1,-1),(1,-1,1),(-1,1,1)\}\),
\(T_y=\operatorname{conv}\{(-1,1,-1),(-1,-1,-1),(1,1,-1),(-1,1,1)\}\), and
\(T_x=\operatorname{conv}\{(1,-1,-1),(-1,-1,-1),(1,1,-1),(1,-1,1)\}\) of
volume \(\frac43\) each, on which \(U_a\) equals \(1+x+y\), \(1+x+z\), and
\(1+y+z\), respectively, with two \(+1\) and two \(-1\) vertex values so
that each lateral integral vanishes. The integrals sum to
\[
\frac43\cdot 1+\frac83\cdot\frac12+3\cdot\frac43\cdot 0=\frac83.
\]
For the path
\(S_\gamma=\{(1,1,1),(1,-1,1),(-1,-1,1),(-1,-1,-1)\}\) there is no free
axis: the three edges of \(S_\gamma\) run in the three coordinate
directions. The least concave majorant is
\[
U_a(x,y,z)
=
\min\bigl(1,\ 1+x-y,\ 1-y+z,\ 1-x+z\bigr).
\]
The right-hand side is concave and matches \(a\) at the eight
vertices. It is therefore at least \(U_a\). Equality holds because
each argument is realised as a convex combination of vertex values on
the region where it attains the minimum. Those four regions partition
\(F\) according to the coordinate orderings
\begin{align*}
R_0&=\{y\le x\le z\},
&
U_a&=1,
\\
R_1&=\{x\le\min(y,z)\},
&
U_a&=1+x-y,
\\
R_2&=\{z\le x\le y\},
&
U_a&=1-y+z,
\\
R_3&=\{x\ge\max(y,z)\},
&
U_a&=1-x+z.
\end{align*}
Each of the six order chambers \(\{u\le v\le w\}\cap F\) is a
tetrahedron of volume \(\frac43\). Thus
\(\operatorname{vol}(R_0)=\operatorname{vol}(R_2)=\frac43\) and
\(\operatorname{vol}(R_1)=\operatorname{vol}(R_3)=\frac83\). The
integral of an affine function on a tetrahedron is the volume times
the mean of the vertex values. Explicitly,
\(R_0=\operatorname{conv}\{( -1,-1,-1),( -1,-1,1),( 1,-1,1),( 1,1,1)\}\)
carries \(U_a=1\), hence integral \(\frac43\);
\(R_2=\operatorname{conv}\{( -1,-1,-1),( -1,1,-1),( 1,1,-1),( 1,1,1)\}\)
carries the values \(1,-1,-1,1\) of \(1-y+z\), hence integral \(0\);
\(R_1\) splits as \(\{x\le y\le z\}\cup\{x\le z\le y\}\), on which
\(1+x-y\) takes the values \(1,1,-1,1\) and \(1,-1,-1,1\), hence
integrals \(\frac23\) and \(0\);
\(R_3\) splits as \(\{y\le z\le x\}\cup\{z\le y\le x\}\), on which
\(1-x+z\) takes the values \(1,-1,1,1\) and \(1,-1,-1,1\), hence
integrals \(\frac23\) and \(0\).
The four regions contribute
\(\frac43+0+\frac23+\frac23=\frac83\).
\end{proof}

\subsection{Restriction counts and orbit bounds}
\label{app:orbit-bounds}

Write \(c^{(ij)}\) for the even labelling \(q\mapsto q_iq_j\) on
\(Q_4\), and \(c^{(1234)}\) for \(q\mapsto q_1q_2q_3q_4\).

\begin{lemma}[Facet types of the balanced orbits]
\label{lem:app-restrictions}
Let \(F_{k,\varepsilon}\) be a coordinate facet of \(\mathcal C_4\).
\begin{enumerate}
\item
  If \(c=c^{(12)}\), then \(c|_{F_{k,\varepsilon}}\) is affine for
  \(k\in\{1,2\}\) and is an opposite-edge colouring for
  \(k\in\{3,4\}\).
\item
  If \(c=c^{(1234)}\), then \(c|_{F_{k,\varepsilon}}\) is a parity
  colouring of \(F_{k,\varepsilon}\) for every \(k\).
\item
  A global star colouring restricts to a local star on two opposite
  facets and to a mixed colouring on the other six.
\item
  A global path-of-length-three colouring restricts to a path of
  length three on every coordinate facet.
\end{enumerate}
\end{lemma}

\begin{proof}
On \(F_{1,+}\) the representatives are \((1,q)\) with \(q\in Q_3\), so
\(c^{(12)}=q_1\) is affine. On \(F_{3,+}\) one has \(c^{(12)}=q_1q_2\),
which is the opposite-edge function of Lemma~\ref{lem:app-Phi-edges}.
The remaining sign and coordinate permutations give~(1). Full parity
restricts to three-variable parity on every facet, which is~(2). For
a star based at \((1,1,1,1)\), the two facets normal to the first
axis see the three neighbours inside that \(3\)-cube, while each of
the other six facets sees two neighbours together with the antipode,
which is the mixed type. For~(4), take the path
\(S_\gamma=\{(1,1,1),(1,-1,1),(-1,-1,1),(-1,-1,-1)\}\subset Q_3\) with
antipodal pairs \(\pm(1,q)\) in \(Q_4\). On \(F_{1,\pm}\) the
restriction is \(S_\gamma\) or \(-S_\gamma\). On \(F_{4,+}=\{q_4=1\}\),
the four positive vertices are
\((-1,1,1,1)\), \((1,-1,-1,1)\), \((1,1,-1,1)\), and \((1,1,1,1)\),
whose first three coordinates form the path
\((-1,1,1)\to(1,1,1)\to(1,1,-1)\to(1,-1,-1)\) of length three.
Similarly, on \(F_{2,+}=\{q_2=1\}\) the four positive vertices are
\((1,1,-1,1)\), \((1,1,1,1)\), \((-1,1,1,-1)\), and \((-1,1,1,1)\),
and on \(F_{3,+}=\{q_3=1\}\) they are
\((-1,-1,1,-1)\), \((-1,1,1,-1)\), \((-1,1,1,1)\), and \((1,1,1,1)\),
whose projections to the free coordinates \((q_1,q_2,q_4)\) form the path
\((-1,-1,-1)\to(-1,1,-1)\to(-1,1,1)\to(1,1,1)\) of length three; evenness
gives the negative facets.
\end{proof}

\begin{proposition}[First-order lower bounds along Boolean rays]
\label{prop:app-n4-bounds}
Along \(c_t=t\,s/(4\sqrt2)\) one has
\[
\liminf_{t\downarrow0}
\frac{\operatorname{vol}(P_t)-16}{t}
\ge
L_0(s),
\]
where
\[
L_0
=
\frac{\sqrt2}{6}
\times
\begin{cases}
8,&\text{face or opposite parallel edges},\\
16,&\text{path of length \(3\)},\\
19,&\text{star or mixed},\\
32,&\text{parity tetrahedron}.
\end{cases}
\]
In particular the twelve weight-two directions satisfy
\(L_0=\frac{4\sqrt2}{3}\) and the two parity directions satisfy
\(L_0=\frac{16\sqrt2}{3}\). Combined with
Proposition~\ref{prop:polar-quadratic},
\[
\liminf_{t\downarrow0}
\frac{\mathcal V_M(P_t)-\mathcal V_M(\mathcal C_4)}{t}
\ge
\frac23 L_0(s).
\]
\end{proposition}

\begin{proof}
Homogeneity of \(\Phi\) and~\eqref{eq:app-volume-from-Phi} give the
liminf as soon as \(\sum_F\Phi(s|_F)\) is known. For \(c^{(12)}\),
Lemma~\ref{lem:app-restrictions} and the local integrals produce
\(4\cdot0+4\cdot\frac83=\frac{32}3\), hence
\[
L_0
=
\frac{1}{4\sqrt2}\cdot\frac{32}{3}
=
\frac{4\sqrt2}{3}.
\]
The opposite-edge orbit is the \(B_4\)-image of this character. For
full parity, eight copies of \(\frac{16}3\) produce
\(\sum\Phi=\frac{128}3\) and \(L_0=\frac{16\sqrt2}{3}\). For a global
star,
\(2\cdot\frac83+6\cdot\frac{10}3=\frac{76}{3}\) produces
\(L_0=\frac{19\sqrt2}{6}\). By Lemma~\ref{lem:app-restrictions}(4), a
path of length three restricts on each of the eight facets to a path
of length three, of pocket \(\frac83\), hence \(\sum\Phi=\frac{64}{3}\)
and \(L_0=\frac{8\sqrt2}{3}\). A global mixed colouring produces the
same sum as a star. Divide~\eqref{eq:app-volume-from-Phi} by \(t\) and
pass to the limit. The Mahler statement is
\(\operatorname{vol}(\mathcal O_4)=\frac23\) times the primal liminf,
plus a quadratic polar remainder.
\end{proof}

The matching upper bound \(\operatorname{vol}(P_t)=16+L_0 t+O(t^2)\)
would require \(t_{\mathrm{exact}}=t_F+O(t^2)\) in every cone (or an
integrated error of order \(o(t)\) to match the first-order coefficient
\(L_0\) alone), which is not claimed. The coefficients \(L_0\) are
therefore one-sided lower bounds.

\begin{proposition}[Exact volume polynomial and first-order sharpness for the parity orbit]
\label{prop:app-n4-parity-exact}
Let \(\chi(q)=q_1q_2q_3q_4\) on \(Q_4=\{\pm1\}^4\), and for \(\tau\in[0,1/3]\) let \(P(\tau):=\operatorname{conv}\bigl\{y_q(\tau):q\in Q_4\bigr\}\), where \(y_q(\tau):=(1+\tau\chi(q))q\). Then for every \(\tau\in[0,1/3]\),
\begin{equation}
\label{eq:app-n4-parity-poly}
\operatorname{vol}\bigl(P(\tau)\bigr)
=
\frac{16}{3}(1+\tau)^4+\frac{32}{3}(1+\tau)^3(1-\tau)
=
16+\frac{128}{3}\tau+32\tau^2-\frac{16}{3}\tau^4.
\end{equation}
In the normalization \(\tau=t/(4\sqrt2)\), the exact polynomial is
\[
\operatorname{vol}(P_t)
=
16+\frac{16\sqrt2}{3}t+t^2-\frac{t^4}{192}.
\]
In particular, the pocket lower bound \(L_0=16\sqrt2/3\) of Proposition~\ref{prop:app-n4-bounds} is sharp to first order along the parity orbit.
\end{proposition}

\begin{proof}
Assume first that \(0<\tau<1/3\). On each reference facet \(F_{i,\varepsilon}=\{x\in[-1,1]^4:x_i=\varepsilon\}\), the standard five-tetrahedra subdivision decomposes \(F_{i,\varepsilon}\) into the central regular \(3\)-simplex \(\Delta_{i,\varepsilon}^{(0)}=\operatorname{conv}\{q\in Q_4\cap F_{i,\varepsilon}:\chi(q)=+1\}\) and four corner \(3\)-simplices \(\Delta_{i,\varepsilon}^{(w)}=\operatorname{conv}\bigl(\{w\}\cup\{q\in Q_4\cap F_{i,\varepsilon}:d_H(q,w)=1\}\bigr)\), one for each odd vertex \(w\in Q_4\cap F_{i,\varepsilon}\) (\(\chi(w)=-1\)). On every square \(2\)-face \(F_{i,\varepsilon}\cap F_{k,\eta}\), the subdivisions of \(F_{i,\varepsilon}\) and \(F_{k,\eta}\) induce the same two triangles separated by the even diagonal \(\{\chi=+1\}\). Hence the \(40\) reference tetrahedra \(\Delta_{i,\varepsilon}^{(0)},\Delta_{i,\varepsilon}^{(w)}\) triangulate \(\partial\mathcal C_4\).

Replacing each reference vertex \(q\) along its ray by \(y_q(\tau)=(1+\tau\chi(q))q\) maps \(\Delta_{i,\varepsilon}^{(0)}\) onto the central \(3\)-simplex \(T_{i,\varepsilon}^{(0)}(\tau)\subset\{\varepsilon y_i=1+\tau\}\) and each \(\Delta_{i,\varepsilon}^{(w)}\) onto a corner \(3\)-simplex \(T_{i,\varepsilon}^{(w)}(\tau)\) contained in the hyperplane
\[
L_{i,\varepsilon,w}(y)
:=
\tau\sum_{j\neq i}w_j y_j + (1-2\tau)\varepsilon y_i
=
1-\tau^2.
\]
We verify that for \(0<\tau<1/3\), all \(40\) halfspaces \(\varepsilon y_i\le 1+\tau\) and \(L_{i,\varepsilon,w}(y)\le 1-\tau^2\) strictly support \(P(\tau)\) outside the vertices of the corresponding \(3\)-simplex:
\begin{itemize}
\item For \(\varepsilon y_i\le 1+\tau\), any \(q\in Q_4\) gives \(\varepsilon(y_q)_i=\varepsilon q_i(1+\tau\chi(q))\), which equals \(1+\tau\) if and only if \(q\in F_{i,\varepsilon}\) and \(\chi(q)=+1\), and is at most \(\max\{1-\tau,-(1-\tau)\}<1+\tau\) otherwise.
\item Fix \((i,\varepsilon,w)\) with \(w\in Q_4\cap F_{i,\varepsilon}\) and \(\chi(w)=-1\). For any \(q\in Q_4\), set \(s:=\varepsilon q_i\in\{\pm1\}\) and let \(m:=\#\{j\neq i:q_j\neq w_j\}\in\{0,1,2,3\}\). Since \(\sum_{j\neq i}w_j q_j=3-2m\) and \(\chi(q)=-s(-1)^m\), we have
\[
L_{i,\varepsilon,w}(y_q)
=
\bigl(1-s(-1)^m\tau\bigr)\bigl(\tau(3-2m)+s(1-2\tau)\bigr).
\]
For \(s=+1\), this equals \(1-\tau^2\) when \(m\in\{0,1\}\) (the four vertices of \(T_{i,\varepsilon}^{(w)}(\tau)\)), whereas \(m=2\) gives
\[
(1-\tau)(1-3\tau)
=(1-\tau^2)-4\tau(1-\tau)
<1-\tau^2,
\]
and \(m=3\) gives
\[
(1+\tau)(1-5\tau)
=(1-\tau^2)-4\tau(1+\tau)
<1-\tau^2.
\]
For \(s=-1\), the case \(m=0\) gives \(L_{i,\varepsilon,w}(y_q)=(1+\tau)(5\tau-1)\). For \(0\le\tau\le 1/3\),
\[
1-\tau^2-(1+\tau)(5\tau-1)=2(1+\tau)(1-3\tau)\ge 0,
\]
with strict inequality on \((0,1/3)\) and equality at \(\tau=1/3\). The cases \(m=1,2,3\) with \(s=-1\) give respectively
\[
-(1-\tau)(1-3\tau),\qquad -(1-\tau^2),\qquad -(1-\tau^2),
\]
all of which are strictly less than \(1-\tau^2\) on \((0,1/3)\).
\end{itemize}
Thus, for \(0<\tau<1/3\), the \(40\) tetrahedra \(T_{i,\varepsilon}^{(0)}(\tau)\) and \(T_{i,\varepsilon}^{(w)}(\tau)\) are supporting facets of \(P(\tau)\). Since \(1+\tau>0\) and \(1-\tau>0\), the radial cones over the moved tetrahedra coincide with the radial cones over their reference tetrahedra. The latter form a complete simplicial fan because the reference tetrahedra triangulate \(\partial\mathcal C_4\) and the origin lies in the interior of \(\mathcal C_4\).

Every ray from the origin therefore lies in the cone over at least one of the \(40\) supporting tetrahedral facets of \(P(\tau)\). Since the origin lies strictly on the inner side of every supporting hyperplane, the corresponding tetrahedron contains the boundary point on that ray. Consequently, the \(40\) facets cover \(\partial P(\tau)\), and their cones from the origin partition \(P(\tau)\) up to sets of measure zero.

Finally, the four-dimensional volume of an origin-simplex with vertices \(0,y_0,y_1,y_2,y_3\in\mathbb R^4\) is
\[
\operatorname{vol}\operatorname{conv}(0,y_0,y_1,y_2,y_3)
=
\frac{|\det(y_0,y_1,y_2,y_3)|}{4!}.
\]
For each of the \(8\) reference facets \(F_{i,\varepsilon}\), the four reference vertices \(q_0,q_1,q_2,q_3\) of \(\Delta_{i,\varepsilon}^{(0)}\) satisfy \(|\det(q_0,q_1,q_2,q_3)|=16\) and each is scaled by \(1+\tau\), so the central origin-simplex has volume \(\frac{16}{4!}(1+\tau)^4=\frac23(1+\tau)^4\). For each of the \(4\) corner tetrahedra \(\Delta_{i,\varepsilon}^{(w)}\), the odd vertex \(w\) and its three neighbours \(w^{(1)},w^{(2)},w^{(3)}\) in \(F_{i,\varepsilon}\) satisfy \(|\det(w,w^{(1)},w^{(2)},w^{(3)})|=8\); since the three even rays carry the factor \(1+\tau\) and the odd ray \(w\) carries the factor \(1-\tau\), multilinearity of the determinant gives directly
\[
\operatorname{vol}\operatorname{conv}\bigl(\{0\}\cup T_{i,\varepsilon}^{(w)}(\tau)\bigr)
=
\frac{8}{4!}(1+\tau)^3(1-\tau)
=
\frac13(1+\tau)^3(1-\tau).
\]
Multiplying the per-facet sum \(\frac23(1+\tau)^4+4\cdot\frac13(1+\tau)^3(1-\tau)\) by \(8\) yields~\eqref{eq:app-n4-parity-poly} for \(0<\tau<1/3\), and the identity extends to the endpoints \(\tau=0\) and \(\tau=1/3\) by continuity of the volume.
\end{proof}

\subsection{A balanced \(4/4\) cut}
\label{app:symmetric_folding}

The following identity lives on a single pair-radial line and records
the four-volume of one lifted facet hull, not the exterior increment
of \(\mathcal C_4\).

Label the eight vertices of the facet \(F=\{x_1=1\}\) by
\(v_{\varepsilon\eta\zeta}=(1,\varepsilon,\eta,\zeta)\). Assign
pair-radii \(R_{\varepsilon\eta\zeta}=2+\frac{t}{\sqrt8}\,\varepsilon\eta\zeta\).
Set \(p_{\varepsilon\eta\zeta}=\frac{R_{\varepsilon\eta\zeta}}{2}v_{\varepsilon\eta\zeta}\)
and \(K_F=\operatorname{conv}\{p_{\varepsilon\eta\zeta}\}\). The eight
points split into complementary regular tetrahedra \(E,O\), placed in
parallel planes \(\{x_1=a\}\) and \(\{x_1=b\}\) with
\(a=\frac{R_+}{2}\) and \(b=\frac{R_-}{2}\). The body \(K_F\) is a
four-dimensional prismatoid. Integrating the mixed-volume polynomial
of the slices \(\mu O+\nu E\) produces
\begin{equation}
\label{eq:app-KF}
\operatorname{vol}_4(K_F)
=
\lvert a-b\rvert\cdot\frac23(a+b)^3
=
\frac{\lvert R_+-R_-\rvert\,(R_++R_-)^3}{24}
=
\frac{4\sqrt2}{3}\,t.
\end{equation}
The factor of two relative to the per-facet contribution
\(\frac{t}{4\sqrt2}\Phi(s|_F)=\frac{2\sqrt2}{3}t\) in
Proposition~\ref{prop:app-n4-bounds} arises because the inner face
of \(K_F\) lies at \(x_1=1-\frac{t}{4\sqrt2}\), inside \(\mathcal C_4\).
The cone increment subtracts the portion already in the cube.

This is consistent with the parity row of
Proposition~\ref{prop:app-n4-bounds} on a single facet; passing from
\(\operatorname{vol}_4(K_F)\) to the exterior increment of
\(\mathcal C_4\) still requires the cone assembly of
Section~\ref{sec:discrete-convexification}.

\section{Ancillary certificate for dimension three}
\label{app:n3-certificate}
\label{app:code}
The following self-contained Python script verifies the exact algebraic
identities in \(\mathbb Z[\alpha,\beta,\gamma]\) for
Lemmas~\ref{lem:n3-sixty}, \ref{lem:n3-vol-thin}, \ref{lem:n3-vol-odd},
and~\ref{lem:n3-wall-clear}. It expands the sixty incidence determinants of
the twelve positive generic triangles~\eqref{eq:n3-triangles}, the twelve
odd generic triangles~\eqref{eq:n3-triangles-odd}, the twelve thin/even-wall
triangles~\eqref{eq:n3-triangles-thin}, and the twelve odd-wall
triangles~\eqref{eq:n3-triangles-odd-wall}: each determinant tests
coplanarity of the triangle with one exterior vertex of \(Q_3\). The
check is an equality of polynomials in the ring, not a numerical
divisibility test: for all four lists, the twelve vanishing determinants
have content \(+16\) and primitive part equal to \(\pm F_1\), \(\pm F_2\)
or \(\pm F_3\) coefficient by coefficient, while the remaining
forty-eight evaluate to \(\pm 8\) at the origin and satisfy the
degree-by-degree coefficient bounds \((24,72,88)\). No polynomial
factorisation algorithm is invoked. The script then discovers, at each
sample, every supporting triangle among the \(56\) triples and compares
that set with the expected list: at \((3,2,1)/100\) the discovered set
is~\eqref{eq:n3-triangles}; at \((-1,-50,-50)/1000\) it
is~\eqref{eq:n3-triangles-odd}; at both \((1,50,50)/1000\) and
\((0,3,2)/100\) it is~\eqref{eq:n3-triangles-thin}; and at \((0,3,-2)/100\)
it is~\eqref{eq:n3-triangles-odd-wall}. For a three-dimensional convex
polytope with \(V=8\) vertices, Euler's relation \(V-E+F=2\) and
\(2E\ge 3F\) give \(F\le 2V-4=12\), so twelve distinct facets are
exhaustive at each sample.

\begin{lstlisting}[language=Python, basicstyle=\small\ttfamily, columns=fullflexible, keepspaces=true, showspaces=false, showstringspaces=false, breaklines=true, breakatwhitespace=true, frame=single, caption={Exact certificate for the sixty incidence determinants of the generic, odd, thin/even-wall, and odd-wall lists and sample evaluations}]
"""Exact certificate for the sixty incidence determinants and samples.

Ring: Z[alpha, beta, gamma]. Vertices in lexicographic order, last
coordinate varying fastest. Each determinant has rows (v, 1).
Layer 1 compares polynomials on the positive generic, odd generic,
thin/even-wall, and odd-wall face lists, and checks the degree-by-degree
coefficient bounds (24, 72, 88) on the 48 nonvanishing origin determinants.
Layer 2 discovers the supporting triangles among all 56 triples and
compares them with the expected list. Layer 3 checks the sample volume.
Run: python3 n3_incidence_certificate.py
"""
import itertools
import sympy as sp

alpha, beta, gamma = sp.symbols("alpha beta gamma")
verts = [(x, y, z) for x in (-1, 1) for y in (-1, 1) for z in (-1, 1)]
positive = [
    (6, 2, 0), (6, 2, 7), (3, 1, 0), (3, 2, 0),
    (3, 1, 7), (3, 2, 7), (5, 1, 7), (5, 1, 0),
    (4, 6, 0), (4, 5, 0), (4, 6, 7), (4, 5, 7),
]
odd_generic = [
    (0, 1, 2), (0, 1, 4), (0, 2, 4), (1, 2, 3),
    (1, 3, 5), (1, 4, 5), (2, 3, 6), (2, 4, 6),
    (3, 5, 7), (3, 6, 7), (4, 5, 6), (5, 6, 7),
]
thin_wall = [
    (0, 1, 3), (0, 1, 5), (0, 2, 3), (0, 2, 4),
    (0, 4, 5), (1, 3, 5), (2, 3, 7), (2, 4, 6),
    (2, 6, 7), (3, 5, 7), (4, 5, 7), (4, 6, 7),
]
odd_wall = [
    (2, 3, 1), (2, 3, 7), (2, 0, 1), (2, 0, 6),
    (2, 6, 7), (3, 1, 7), (0, 1, 5), (0, 6, 4),
    (0, 4, 5), (1, 7, 5), (6, 7, 5), (6, 4, 5),
]
F1 = alpha**3 - alpha*beta**2 - alpha*gamma**2 - alpha + 2*beta*gamma
F2 = alpha**2*beta - 2*alpha*gamma - beta**3 + beta*gamma**2 + beta
F3 = alpha**2*gamma - 2*alpha*beta + beta**2*gamma - gamma**3 + gamma
walls = (F1, -F1, F2, -F2, F3, -F3)

def point(q):
    x, y, z = q
    s = 1 + alpha*x*y + beta*x*z + gamma*y*z
    return (s*x, s*y, s*z)

def det4(ids):
    rows = [list(point(verts[i])) + [1] for i in ids]
    return sp.expand(sp.Matrix(rows).det())

for label, face_list in [
    ("positive", positive),
    ("odd_generic", odd_generic),
    ("thin_wall", thin_wall),
    ("odd_wall", odd_wall),
]:
    vanish = 0
    nonzero = 0
    for tri in face_list:
        for p in (j for j in range(8) if j not in tri):
            poly = det4(list(tri) + [p])
            val = sp.Integer(poly.subs({alpha: 0, beta: 0, gamma: 0}))
            if val == 0:
                vanish += 1
                content, primitive = sp.primitive(poly)
                assert content == 16
                assert primitive in walls
                assert poly == content * primitive
            else:
                nonzero += 1
                assert val in (-8, 8)
                p_obj = sp.Poly(poly - val, alpha, beta, gamma)
                d1 = sum(abs(c) for mon, c in p_obj.terms() if sum(mon) == 1)
                d2 = sum(abs(c) for mon, c in p_obj.terms() if sum(mon) == 2)
                d3 = sum(abs(c) for mon, c in p_obj.terms() if sum(mon) == 3)
                assert d1 <= 24 and d2 <= 72 and d3 <= 88
    assert vanish == 12 and nonzero == 48
    print(f"{label}: sixty=60 vanish=12 nonzero=48; each vanishing det is 16*(+/- Fi)")

def discover(P):
    found = []
    for tri in itertools.combinations(range(8), 3):
        i, j, k = tri
        if (P[j] - P[i]).cross(P[k] - P[i]) == sp.Matrix.zeros(3, 1):
            continue
        signs = []
        for p in range(8):
            if p in tri:
                continue
            M = sp.Matrix([list(P[t]) + [1] for t in list(tri) + [p]])
            signs.append(sp.sign(M.det()))
        assert all(v in (-1, 0, 1) for v in signs)
        if 0 not in signs and len(set(signs)) == 1:
            found.append(tri)
    return found

def certify(name, sample, expected, volume):
    P = []
    for q in verts:
        x, y, z = q
        s = 1 + sample[alpha]*x*y + sample[beta]*x*z + sample[gamma]*y*z
        P.append(sp.Matrix([s*x, s*y, s*z]))
    found = discover(P)
    actual = {tuple(sorted(t)) for t in found}
    wanted = {tuple(sorted(t)) for t in expected}
    assert actual == wanted
    assert len(actual) == 12
    # For a convex 3-polytope with V=8, Euler and 2E>=3F give F<=12.
    total = 0
    for i, j, k in found:
        total += sp.Abs(P[i].dot(P[j].cross(P[k])))
    assert sp.simplify(total / 6) == volume
    print(f"{name}: discovered facets equal the expected 12; volume {volume}")

certify(
    "positive (3,2,1)/100",
    {alpha: sp.Rational(3, 100), beta: sp.Rational(2, 100), gamma: sp.Rational(1, 100)},
    positive,
    sp.Rational(381706, 46875),
)
certify(
    "negative (-1,-50,-50)/1000",
    {alpha: sp.Rational(-1, 1000), beta: sp.Rational(-50, 1000), gamma: sp.Rational(-50, 1000)},
    odd_generic,
    sp.Rational(3091187099, 375000000),
)
certify(
    "thin (1,50,50)/1000",
    {alpha: sp.Rational(1, 1000), beta: sp.Rational(50, 1000), gamma: sp.Rational(50, 1000)},
    thin_wall,
    sp.Rational(1029602367, 125000000),
)
certify(
    "even wall (0,3,2)/100",
    {alpha: sp.Integer(0), beta: sp.Rational(3, 100), gamma: sp.Rational(2, 100)},
    thin_wall,
    sp.Rational(203153, 25000),
)
certify(
    "odd wall (0,3,-2)/100",
    {alpha: sp.Integer(0), beta: sp.Rational(3, 100), gamma: sp.Rational(-2, 100)},
    odd_wall,
    sp.Rational(203153, 25000),
)
\end{lstlisting}

\end{document}